\documentclass[reqno, oneside, 12pt]{amsart}
\usepackage[letterpaper]{geometry}
\usepackage{amsmath,amssymb}
\usepackage{amsthm}
\usepackage{thmtools}
\usepackage{mathtools}
\usepackage{dsfont}
\usepackage{color}
\usepackage{enumitem}
\usepackage{appendix}
\usepackage{comment}
\usepackage[normalem]{ulem}
\usepackage{cases}
\usepackage{multirow}
\usepackage{hhline}
\usepackage{mathrsfs}

\usepackage[hypertexnames=false]{hyperref}

\newcommand{\Z}{{\mathbb{Z}}}

\newcommand{\R}{{\mathbb{R}}}
\newcommand{\C}{{\mathbb{C}}}

\newcommand{\mR}{{\mathcal{R}}}
\newcommand{\e}{{\mathfrak{e}}}

\newcommand{\HH}{{\mathbb{H}}}
\newcommand{\Mform}{{\mathcal{M}}}

\newcommand{\Lform}{{\mathcal{L}}}
\newcommand{\LEigenform}{{\tilde{\mathcal{L}}}}

\DeclareMathOperator{\re}{Re}
\DeclareMathOperator{\im}{Im}
\DeclareMathOperator*{\Res}{Res}
\DeclareMathOperator{\SL}{SL}
\DeclareMathOperator{\GL}{GL}
\DeclareMathOperator{\rank}{rank}
\newcommand{\be}{\begin{equation}}
\newcommand{\ee}{\end{equation}}

\newcommand{\ep}{\varepsilon}
\newcommand{\mF}{{\mathbf{F}}}
\newcommand{\ma}{{\mathfrak{a}}}
\newcommand{\mb}{{\mathfrak{b}}}
\newcommand{\mm}{{\mathbf{m}}}

\newcommand{\defeq}{\vcentcolon=}
\newcommand{\floor}[1]{\left\lfloor #1 \right\rfloor}

\newcommand{\ceil}[1]{\left\lceil #1 \right\rceil}

\renewcommand{\(}{\left(}
\renewcommand{\)}{\right)}
\newcommand{\la}{\left|}
\newcommand{\ra}{\right|}

\newcommand{\condr}{\vartriangleright r \vartriangleleft}
\newcommand{\condar}{\vartriangleright a,r \vartriangleleft}

\newcommand{\Mod}[1]{\ (\mathrm{mod}\ #1)}
\DeclareMathOperator{\sgn}{sgn}

\DeclareMathOperator{\diag}{diag}

\newtheorem{theorem}{Theorem}

\newtheorem{lemma}[theorem]{Lemma}
\newtheorem{corollary}[theorem]{Corollary}
\newtheorem{proposition}[theorem]{Proposition}

\newtheorem{definition}[theorem]{Definition}

\newtheorem{notation}[theorem]{Notation}

\theoremstyle{remark}
\newtheorem*{remark}{Remark}

\numberwithin{equation}{section}
\numberwithin{theorem}{section}
\numberwithin{lemma}{section}
\numberwithin{proposition}{section} 
\numberwithin{example}{section}
\numberwithin{definition}{section}
\numberwithin{corollary}{section}
\numberwithin{condition}{section}
\numberwithin{notation}{section}
\numberwithin{claim}{section}

\numberwithin{table}{section}

\author{Qihang Sun}
\address{Department of Mathematics, University of Illinois Urbana-Champaign, Urbana, IL, USA}

\title{Exact formulae for ranks of partitions}

\email{qihangsun98@outlook.com}

\date{\today}

\begin{document}

\begin{abstract}
	In 2009, Bringmann \cite{BringmannTAMS} used the circle method to prove an asymptotic formula for the Fourier coefficients of rank generating functions. In this paper, we prove that Bringmann's formula, when summing up to infinity and in the case of prime modulus, gives a Rademacher-type exact formula involving sums of vector-valued Kloosterman sums. As a corollary, in another paper \cite{QihangVanishingKLsums2025}, we will provide a new proof of Dyson's conjectures by showing that the certain Kloosterman sums vanish. 
\end{abstract}

\maketitle

\section{Introduction}

Let $p(n)$ be the integer partition function, which denotes the number of ways to write the natural number $n$ as a sum of a non-increasing sequence of positive integers. For example, we have $p(4)=5$ and $p(100)=190\,569\,292$. In 1918, Hardy and Ramanujan \cite{HardyRamanujan1918Asymp} proved the asymptotic for $p(n)$:
\[p(n)\sim \frac1{4n\sqrt 3}\exp\(\pi \sqrt{\frac{2n}3}\). \] 
Later, in 1938, Rademacher \cite{Rademacher1937pn} proved an exact formula for $p(n)$. If we define
\begin{equation}\label{A_cn}
	A_c(n)\defeq \frac12\sqrt{\frac c{12}}\sum_{\substack{x\Mod{24c}\\x^2\equiv -24n+1\Mod{24c}}} \chi_{12}(x)e\(\frac x{12c}\),
\end{equation}
where $\chi_{12}$ is the Dirichlet character $(\frac{12}\cdot)$ and $e(z)\defeq e^{2\pi i z}$, then Rademacher's exact formula \cite[(1.8)]{Rademacher1937pn} can be written as
\begin{align}\label{RademacherExactFormula}
	\begin{split}
		p(n)=\frac1{\pi\sqrt2}\sum_{c=1}^\infty A_c(n)\sqrt c\,\frac{d}{dn}\Bigg(\frac{\sinh\(\frac\pi c\sqrt{\frac23(n-\frac1{24})}\)}{\sqrt{n-\frac 1{24}}}\Bigg). 
	\end{split}
\end{align}

Ramanujan also obtained the famous congruence properties of $p(n)$:
\begin{equation}\label{Ramanujan congruences}
	p(5n+4)\equiv 0\Mod 5,\quad p(7n+5)\equiv 0\Mod 7,\quad p(11n+6)\equiv 0\Mod{11}. 
\end{equation}  
In 1944, Dyson \cite{Dyson} defined the rank of a partition to strikingly interpret the above congruences. Suppose $\Lambda=\{\Lambda_1\geq \Lambda_2\geq \cdots\geq \Lambda_\kappa\}$ is a partition of $n$, i.e. $\sum_{j=1}^\kappa \Lambda_j=n$ and every $\Lambda_j>0$. Let 
\[\rank(\Lambda)\defeq\Lambda_1-\kappa\]
define the rank of this partition, and let the quantities $N(m,n)$ and $N(a,b;n)$ be defined by
\begin{equation}
	N(m,n)\defeq \#\{\Lambda \text{ is a partition of }n: \rank \Lambda=m\}
\end{equation} 
and 
\begin{equation}
	N(a,b;n)\defeq \#\{\Lambda \text{ is a partition of }n: \rank \Lambda\equiv a\Mod b\}.
\end{equation}
Let $q=\exp(2\pi i z)=e(z)$ for $z\in \HH$ and $w$ be a root of unity. It is well known (e.g. \cite[(1.4)]{BrmOno2010}) that the generating function of $N(m,n)$ can be written as
\begin{equation}\label{rankGeneratingFunction}
	\mR(w;q)\defeq 1+\sum_{n=1}^\infty\sum_{m=-\infty}^\infty N(m,n) w^mq^n=1+\sum_{n=1}^\infty\frac{q^{n^2}}{(wq;q)_n(w^{-1}q;q)_n},
\end{equation}
where $(a;q)_n\defeq\prod_{j=0}^{n-1}(1-aq^j)$. 
For example, $\mR(1;q)=1+\sum_{n=1}^\infty p(n)q^n$ is the generating function for partitions. For integers $b>a>0$, let $A(\frac ab;n)$ denote the Fourier coefficient of $\mR(\zeta_b^a;q)$: 
\[\mR(\zeta_b^a;q)=:1+\sum_{n=1}^\infty A\(\frac ab;n\)q^n\]
where $\zeta_b=\exp(\frac{2\pi i}b)$ is a $b$-th root of unity. The following identity is easy to get but helpful in understanding the relation between $A(\frac ab;n)$ and $N(a,b;n)$:
\begin{equation}\label{Relation between rank mod and A()}
	bN(a,b;n)=p(n)+\sum_{j=1}^{b-1}\zeta_b^{-aj}A\(\frac jb;n\). 
\end{equation}
It is not hard to show that $A(\frac jb;n)\in \R$ and $A(\frac jb;n)=A(1-\frac jb;n)$ for $1\leq j\leq b-1$, because $N(a,b;n)=N(-a,b;n)$ and $\zeta_b^{-aj}+\zeta_b^{-a(b-j)}=2\cos (\frac{\pi a j}b)$.   

The function $\mR(w;q)$ has many beautiful connections and properties. When $w=-1$, it is known that $N(0,2;n)-N(1,2;n)=A(\frac12;n)$ is the Fourier coefficient of Ramanujan's third order mock theta function $f(q)$ (see e.g. \cite[p. 127-131]{RamanujanLostNotebook}).  Dragonette \cite{Dragonette1952} and Andrews \cite{Andrews1966} improved the asymptotic formula of $A(\frac12;n)$ which was conjectured by Ramanujan. The exact formula of $A(\frac12;n)$ was later proven by Bringmann and Ono:
\begin{theorem}[{\cite[Theorem~1.1]{BrmOno2006ivt}}]\label{exactFmlThmMod2}
	The Andrews-Dragonette conjecture is true: 
	\begin{equation}\label{exactFmlMod2}
		A\(\frac 12;n\)=\frac{\pi}{(24n-1)^{\frac 14}}\sum_{c=1}^\infty \frac{(-1)^{\floor{\frac{c+1}2}}A_{2c}(n-\frac{c(1+(-1)^c)}4)}c I_{\frac12}\(\frac{\pi\sqrt{24n-1}}{12c}\). 
	\end{equation}
\end{theorem}

The author \cite{QihangFirstAsympt} provided a detailed proof of the exact formula of $A(\frac 13;n)=A(\frac 23;n)$, which is the Fourier coefficient of $\mR(\zeta_3;q)=\mR(\zeta_3^2;q)$. Let $S(m,n,c,\nu)$ denote the Kloosterman sum defined in \eqref{eq:kloos_def} with $m,n,c\in \Z$ and multiplier system $\nu$ (see Definition~\ref{multiplier system def}). Let $\nu_\eta$ denote the weight $\frac 12$ multiplier of the Dedekind $\eta$ function defined in \eqref{etaMultiplier}. 
\begin{theorem}[{\cite[Theorem~2.2]{QihangFirstAsympt}}]\label{exactFmlThmMod3}We have
	\begin{equation}\label{exactFmlMod3}
		A\(\frac 13;n\)=A\(\frac 23;n\)=\frac{2\pi\,e(-\frac18)}{(24n-1)^{\frac14}}\sum_{3|c>0}\frac{S(0,n,c,\,(\frac \cdot3)\overline{\nu_\eta})}{c}I_{\frac12}\(\frac{\pi\sqrt{24n-1}}{6c}\).
	\end{equation}
\end{theorem}

In 2009, Bringmann \cite{BringmannTAMS} used the circle method to find the asymptotics of $A(\frac \ell u;n)$ for general odd $u\geq 3$. Let $s(d,c)$ be the Dedekind sum defined in \eqref{etaMultiplier} and 
\begin{equation}\label{omega dc}
	\omega_{d,c}\defeq \exp(\pi i s(d,c)). 
\end{equation} 
When $(d,c)=1$, define $d'_{c}$ by $dd'_{c}\equiv -1\Mod c$ if $c$ is odd and $dd'_{c}\equiv -1\Mod{2c}$ if $c$ is even. 
Denote 
\[g_{(h)}\defeq \frac{g}{\gcd(g,h)}\]
for non-zero integers $g$ and $h$. If $u|c$, define
\begin{equation}\label{B lpc sum def}
	B_{\ell ,u,c}(n,m)\defeq (-1)^{\ell c+1}\sum_{d\Mod c^*}\frac{\sin(\tfrac {\pi \ell }u)\;\omega_{d,c}}{\sin(\frac{\pi \ell d'_{c}}u)\exp(\frac{3\pi i c_{(u)}d'_{c}\ell^2}{u})}e\(\frac{md'_{c}+nd}c\). 
\end{equation}
When $u\nmid a$ and $1\leq \ell<u_{(a)}$, let $0\leq [a_{(u)}\ell]<u_{(a)}$ be defined by $[a_{(u)}\ell]\equiv a_{(u)}\ell \Mod {u_{(a)}}$. Define
\begin{equation}\label{D lpa sum def}
	D_{\ell,u,a}(n,m)\defeq (-1)^{a\ell+[a_{(u)}\ell]}\sum_{b\Mod a^*}\omega_{b,a}e\(\frac{mb'_{a}+nb}a\). 
\end{equation}
When $u\nmid a$, define 
\begin{equation}\label{delta l u a r def}
	\delta_{\ell,u,a,r}\defeq \left\{\begin{array}{ll}
		-(\frac12+r)\frac {[a_{(u)}\ell ]}{u_{(a)}}+\frac32\(\frac {[a_{(u)}\ell]}{u_{(a)}}\)^2+\frac1{24},&\ \text{if }0<\frac{[a_{(u)}\ell]}{u_{(a)}}<\frac 16,\vspace{0.3em}\\
		-\frac{5[a_{(u)}\ell ]}{2u_{(a)}}+\frac32\(\frac {[a_{(u)}\ell]}{u_{(a)}}\)^2+\frac{25}{24}-r+\frac{r[a_{(u)}\ell]}{u_{(a)}},&\ \text{if }\frac 56<\frac{[a_{(u)}\ell]}{u_{(a)}}<1,\vspace{0.3em}\\
		0&\ \text{otherwise}, 
	\end{array}
	\right. 
\end{equation}
and when $0<\frac{[a_{(u)}\ell]}{u_{(a)}}<\frac 16$ or $\frac 56<\frac{[a_{(u)}\ell]}{u_{(a)}}<1$, define
\begin{align}\label{m l p a r def}
	\begin{split}
		&m_{\ell,u,a,r}\defeq \frac1{2u_{(a)}^2}\\
		&\cdot \left\{\begin{array}{ll}
			\Big(-3\(a_{(u)}\ell-[a_{(u)}\ell]\)^2-u_{(a)}(1+2r)\(a_{(u)}\ell-[a_{(u)}\ell]\)\Big),&0<\frac{[a_{(u)}\ell]}{u_{(a)}}<\frac 16,\vspace{0.3em}\\
			\Big(-3\(a_{(u)}\ell-[a_{(u)}\ell]\)^2+u_{(a)}(2r-5)\(a_{(u)}\ell-[a_{(u)}\ell]\)+2u_{(a)}^2(r-1)\Big),&\frac 56<\frac{[a_{(u)}\ell]}{u_{(a)}}<1.
		\end{array}
		\right. 
	\end{split}
\end{align}
By \cite[bottom of p. 3485]{BringmannTAMS}, or directly by $u_{(a)}\Big|\(a_{(u)}\ell-[a_{(u)}\ell]\)$, we can see that $m_{\ell,u,a,r}\in \Z$ always. 

Bringmann proved: 
\begin{theorem}[{\cite[Theorem~1.1]{BringmannTAMS}}]If $1\leq \ell<u$ are coprime integers and $u$ is odd, then for positive integers $n$ we have
	\begin{align}\label{Bringmann formula}
		\begin{split}
			&A\(\frac \ell u;n\)=\frac{4\sqrt 3\, i}{\sqrt{24n-1}}\sum_{c:\;u|c\leq \sqrt n}\frac{B_{\ell,u,c}(-n,0)}{\sqrt c} \sinh\(\frac{\pi\sqrt{24n-1}}{6c}\)\\
			&+\frac{8\sqrt{3}\sin(\frac{\pi \ell}u)}{\sqrt{24n-1}}\sum_{r\geq 0}\sum_{\substack{a\leq\sqrt n:\\u\nmid a,\\ \delta_{\ell,u,a,r}>0}}\frac{D_{\ell,u,a}(-n,m_{\ell,u,a,r})}{\sqrt a}\sinh\(\frac{\pi\sqrt{2\delta_{\ell,u,a,r}(24n-1)}}{a\sqrt{3}}\){+O_{u,\ep}(n^\ep)}. 
		\end{split}
	\end{align}
\end{theorem}
Note that the sum of $r\geq 0$ in the second line is a finite sum because when $u$ is fixed and $r$ is large enough, $\delta_{\ell,u,a,r}$ will be always negative. Here we have modified the notation in Bringmann's paper for convenience. 
Bringmann and Ono \cite{BringmannOno2012} claimed that the above sum, when summing up to infinity, should be the exact formula for $A(\frac \ell u;n)$.

When $u=p$ is a prime number, we prove their statement in this paper.  We also explain $B_{\ell,p,c}(-n,0)$ and $D_{\ell,p,a}(-n,m_{\ell,p,a,r})$ as components of vector-valued Kloosterman sums. Note that when $u=p$, we have $c_{(p)}=\frac cp$, $a_{(p)}=a$ and $p_{(a)}=p$ for $p|c$ and $p\nmid a$, hence the formulae from \eqref{B lpc sum def} to \eqref{m l p a r def} become simpler. Specifically, we rewrite $\delta_{\ell,p,a,r}$ at \eqref{delta l p a r def}. 

Let $\mu_p:\Gamma_0(p)\rightarrow \GL_{p-1}(\C)$ be defined as in \eqref{mu matrix define}. For every integer $r\geq 0$, let the vector $\mathbf{X}_r$ be defined as in \eqref{X r define} and let $\mathbf{S}_{\infty\infty}(m,n,c,\mu_p)$ and $\mathbf{S}_{0\infty}(\mathbf{X}_r,n,a,\mu_p;r)$ be the vector-valued Kloosterman sums defined in \eqref{S infty infty def} and \eqref{S 0 infty def, using X r}, with $S_{\infty\infty}^{(\ell)}(m,n,c,\mu_p)$ and $S_{0\infty}^{(\ell)}(X_r^{([a\ell])},n,a,\mu_p;r)$ as the scalar values at their $\ell$-th entry, respectively. Let $x_r$ be defined at \eqref{x r def}. 

For a prime $p$, we define $[a\ell]$ by $0\leq [a\ell]<p$ and $[a\ell]\equiv a\ell\Mod p$. 
Then we have the following theorem. 
\begin{theorem}\label{main theorem}
	For every prime $p\geq 5$, integer $1\leq \ell\leq p-1$ and positive integer $n$, with the Kloosterman sums defined in \eqref{S infty infty ell def} and \eqref{S 0 infty ell def, using X r}, we have 
	\begin{align}\label{main theorem formula}
		\begin{split}
			&A\(\frac \ell p;n\)=\frac{2\pi e(-\frac 18)\sin(\frac{\pi\ell}p)}{(24n-1)^{\frac14}} \sum_{c>0:\,p|c}\frac{S_{\infty\infty}^{(\ell)}(0,n,c,\mu_p)}{c} I_{\frac12}\(\frac{4\pi \sqrt{24n-1}}{24c}\)\\
			&+\frac{4\pi \sin(\frac{\pi\ell}p)}{(n-\frac1{24})^{\frac14}}\sum_{\substack{r\geq 0\\x_r^{-1}<p}}\!\!\!\sum_{\substack{a>0:\,p\nmid a,\\  \frac{[a\ell]}p\in (0,x_r)\\ {\mathrm{or\ }}\frac{[a\ell]}p\in(1-x_r,1)}}\!\!\!\!\!\!\!\!\!\frac{S_{0\infty}^{(\ell)}\(\ceil{-p\delta_{\ell,p,a,r}},n,a,\mu_p;r\)}{a\cdot \delta_{\ell,p,a,r}^{-\frac14}}I_{\frac12}\(\frac{4\pi \sqrt{\delta_{\ell,p,a,r}(n-\frac1{24})}}{a}\),
		\end{split}
	\end{align}
	where $\ceil{x}$ is the smallest integer $\geq x$ and $\floor{x}$ is the largest integer $\leq x$. 
\end{theorem} 

\begin{remark}We have several remarks regarding the our main theorem above. 
	\begin{itemize}
		\item[(1)]
	This theorem also proves that Bringmann's formula \eqref{Bringmann formula}, when summing up $c$ and $a$ to infinity, is the exact formula. Indeed, for all prime $p\geq 5$, $1\leq \ell\leq p-1$, $r\geq 0$, positive integers $a,c$ such that $p|c$, $p\nmid a$, and $\delta_{\ell,p,a,r}>0$, we have the following relations: 
	\begin{align}
		\label{KL sums match with Bringmann, cusp infty}
		\overline{i\cdot B_{\ell,p,c}(-n,0)}&=e(-\tfrac 18)\sin(\tfrac{\pi\ell}p)S_{\infty\infty}^{(\ell)}(0,n,c,\mu_p),\\
		\label{KL sums match with Bringmann, cusp 0}
		\overline{D_{\ell,p,a,r}(-n,m_{\ell,p,a,r})}&=S_{0\infty}^{(\ell)}\(\ceil{-p\delta_{\ell,p,a,r}},n,a,\mu_p;r\), \\
		I_{\frac 12}(z) &= (\tfrac 2{\pi z})^{\frac 12}\sinh(z) \quad (\text{\cite[(10.39.1)]{dlmf}}), \text{ and }\\
		\label{delta>0 iff alp in xr intervals}
		\delta_{\ell,p,a,r}>0 \quad &\text{if and only if}\quad  \tfrac{[a\ell]}p\in (0,x_r)\cup (1-x_r,1). 
	\end{align}
	The last relation \eqref{delta>0 iff alp in xr intervals} is clear from the definition. 
	Since $A(\frac \ell p;n)$, $\delta_{\ell,p,a,r}$ and $I_{\frac 12}(y)$ (for $y\in \R$) are all real (see \eqref{Relation between rank mod and A()} for $A$), we are safe to take the complex conjugation of \eqref{Bringmann formula}. 
	We include the proof of \eqref{KL sums match with Bringmann, cusp infty} and  \eqref{KL sums match with Bringmann, cusp 0} in \S\ref{Subsection: Proof of main theorem}. 

\item[(2)] Bringmann's formula \eqref{Bringmann formula} truncates the sum over $c$ and $a$ at $c\leq \sqrt n$ and $a\leq \sqrt n$, with the error term $O_{u,\ep}(n^\ep)$. This estimate is due to the individual Weil-type bound of the form $O(c^{1/2+\ep})$ and $O(a^{1/2+\ep})$ of Kloosterman-type sums in \cite[Lemma~3.2]{BringmannTAMS}. This kind of estimate can be traced back to \cite[Lemma~4.1]{Andrews1966} in the case $u=2$, where Andrews discussed the relation with generalized Kloosterman sums. Such estimates are insufficient to prove the convergence of these sums over $c$ and $a$ when we test the absolute value of the summands and apply the approximation $\sinh(x) = x + O(x^3)$ as $x \to 0$.

Bringmann and Ono achieved the exact formula for $u=2$ by introducing the weak Maass forms, and this issue is overcome by proving the convergence of sums of Kloosterman sums in \cite[\S4]{BrmOno2006ivt}. For the case $u=3$ in \eqref{Bringmann formula}, the author overcame this issue in \cite[\S10]{QihangFirstAsympt} with a systematic method estimating sums of generalized Kloosterman sums. Therefore, by explaining $B_{\ell,p,c}(-n,0)$ and $D_{\ell,p,a,r}(-n,m_{\ell,p,a,r})$ as Kloosterman sums and estimating their sums, we shall be able to generalize the method in \cite{BrmOno2006ivt} to exact formulae in other cases. This partially forms the motivation of our main theorem above.

\item[(3)] Bringmann also proved the Andrews-Lewis conjecture for partition biases of rank modulo 3 in \cite[Theorem~1.3]{BringmannTAMS}. Our exact formulae here may provide properties of partition biases in other cases. 
\end{itemize}
	
\end{remark}

\section{Dyson's conjectures: corollary of Theorem~\ref{main theorem}}
Dyson made the following conjectures which were proved by Atkin and Swinnerton-Dyer in 1953. 
\begin{theorem}[{\cite{AtkinSDrank}}]
	\label{Dyson's conjectures in ASD}
	For all $n\geq 0$, we have the following identities: 
	\begin{align*}
		N(1,5;5n+1)&=N(2,5;5n+1);			\tag{5-1} 		\\
		N(0,5;5n+2)&=N(2,5;5n+2);			\tag{5-2}		 \\
		N(0,5;5n+4)=N(1,5;5n+4)&=N(2,5;5n+4); 	\tag{5-4}\\
		N(2,7;7n)&=N(3,7;7n); 					\tag{7-0}\\
		N(1,7;7n+1)=N(2,7;7n+1)&=N(3,7;7n+1); 	\tag{7-1}\\
		N(0,7;7n+2)&=N(3,7;7n+2); 				\tag{7-2}\\
		N(0,7;7n+3)=N(2,7;7n+3), &\quad \ N(1,7;7n+3)=N(3,7;7n+3); \tag{7-3}\\
		N(0,7;7n+4)=N(1,7;7n+4)&=N(3,7;7n+4); 				\tag{7-4}\\
		N(0,7;7n+5)=N(1,7;7n+5)&=N(2,7;7n+5)=N(3,7;7n+5); 	\tag{7-5}\\
		N(0,7;7n+6)+N(1,7;7n+6)&=N(2,7;7n+6)+N(3,7;7n+6).  	\tag{7-6}
	\end{align*}
	
\end{theorem}

\begin{remark}
	By \eqref{Relation between rank mod and A()} and $N(a,b;n)=N(b-a,b;n)$, the identity (5-4) implies \[N(\ell,5;5n+4)=\frac 15 p(5n+4)\quad \text{ for all }\ell,\]
	which implies the Ramanujan congruence $p(5n+4)\equiv 0\Mod 5$. The identity (7-5) implies 
	\[N(\ell,7;7n+5)=\frac 17 p(7n+5)\quad \text{ for all }\ell,\] 
	which implies the Ramanujan congruence $p(7n+5)\equiv 0\Mod 7$. 
\end{remark}

By \eqref{Relation between rank mod and A()}, our formula \eqref{main theorem formula} can also show some relations of $N(\ell,p;n)$ for $1\leq \ell\leq p-1$ by properties of $A(\frac{\ell}p;n)$. When $p=5$ or $7$, Theorem~\ref{main theorem} reduces to the following corollary. 
\begin{corollary}\label{main theorem corollary}
	For every positive integer $n$, when $p=5$ and $1\leq \ell\leq 4$, we have
	\begin{align}\label{main theorem formula, mod 5}
		\begin{split}
			A\(\frac \ell 5;n\)=\frac{2\pi e(-\frac 18)\sin(\frac{\pi\ell}5)}{(24n-1)^{\frac14}} \sum_{c>0:\,5|c}\frac{S_{\infty\infty}^{(\ell)}(0,n,c,\mu_5)}{c} I_{\frac12}\(\frac{4\pi \sqrt{24n-1}}{24c}\); 
		\end{split}
	\end{align}
	when $p=7$ and $1\leq \ell \leq 6$, we have
	\begin{align}\label{main theorem formula, mod 7}
		\begin{split}
			A\(\frac \ell 7;n\)&=\frac{2\pi e(-\frac 18)\sin(\frac{\pi\ell}7)}{(24n-1)^{\frac14}} \sum_{c>0:\,7|c}\frac{S_{\infty\infty}^{(\ell)}(0,n,c,\mu_7)}{c} I_{\frac12}\(\frac{4\pi \sqrt{24n-1}}{24c}\)\\
			&+\frac{4\pi \sin(\frac{\pi\ell}7)}{(24n-1)^{\frac14}}\sum_{\substack{a>0:\,7\nmid a,\\ [a\ell]=1\text{ or }6}}\!\!\!\!\!\frac{S_{0\infty}^{(\ell)}\(0,n,a,\mu_7;0\)}{\sqrt 7\,a }I_{\frac12}\(\frac{4\pi \sqrt{24n-1}}{24\times 7a}\).
		\end{split}
	\end{align}
\end{corollary} 
\begin{proof}
	Note that $x_0=\frac 16$ is the only solution of \eqref{x r def} when $r=0$. 
Hence we only have the first sum for $5|c$ in \eqref{main theorem formula} when $p=5$. When $p=7$, only $r=0$ is allowed for the second sum in  \eqref{main theorem formula}. Recall $\delta_{\ell,p,a,r}$ at \eqref{delta l p a r def}. Since $[a\ell]=1$ or $6$ if and only if $\frac{[a\ell]}p\in (0,x_0)\cup (1-x_0,1)$, we have 
	\[\delta_{\ell,7,a,0}=\left\{
	\begin{array}{ll}
		(7^2\times 24)^{-1},&\text{ if }[a\ell]=1\text{ or }6,\\
		0,&\text{ otherwise,}
	\end{array}
	\right. 
	\]
	and $\ceil{-7\delta_{\ell,7,a,0}}=0$ always. 
\end{proof}

In another paper \cite{QihangVanishingKLsums2025}, the author proves the following specific vanishing properties of the Kloosterman sums $S_{\infty\infty}^{(\ell)}(0,5n+4,c,\mu_5)$,  $S_{\infty\infty}^{(\ell)}(0,7n+5,c,\mu_7)$, and $S_{0\infty}^{(\ell)}(0,7n+5,a,\mu_7;0)$. These properties provide a new proof of Dyson's conjectures (Theorem~\ref{Dyson's conjectures in ASD}). 

\begin{theorem}[{\cite[Theorem~1.3]{QihangVanishingKLsums2025}}]
	\label{Kloosterman sums vanish}
	(i) For all integers $n\geq 0$ and $1\leq\ell\leq p-1$ for $p=5,7$ (denoted by $p|c$ below), we have the following vanishing conditions for the Kloosterman sums appeared in Corollary~\ref{main theorem corollary}: 
	\begin{enumerate}
		\item[(5-4)] If $5|c$, we have $S_{\infty\infty}^{(\ell)}(0,5n+4,c,\mu_5)=0$.
		\item[(7-5)] If $7|c$, $\frac c7\cdot \ell\not\equiv 1 \Mod 7$, and $\frac c7\cdot\ell\not \equiv -1\Mod 7$, then $S_{\infty\infty}^{(\ell)}(0,7n+5,c,\mu_7)=0$. \\
		If $7|c$, $7\nmid a$, $a\ell\equiv \pm1 \Mod 7$, and $c=7a$, we have 
		\[ e(-\tfrac 18)S_{\infty\infty}^{(\ell)}(0,7n+5,c,\mu_7)+2{\sqrt 7}\,S_{0\infty}^{(\ell)}(0,7n+5,a,\mu_7;0)=0.\]
	\end{enumerate}
	
	(ii) Furthermore, we denote $C_p^{a,b}\defeq \cos(\frac{a\pi }p)-\cos(\frac{b\pi}p)$ and
	\begin{align*}
		S_7^{(\ell)}(n,c)\defeq \sin(\tfrac {\pi\ell}7) \(e(-\tfrac 18){S_{\infty\infty}^{(\ell)}(0,n,c,\mu_7)}+\mathbf{1}_{\substack{ a\defeq c/7 \\ {[a\ell]=1,6}}}\cdot
		2\sqrt 7\,S_{0\infty}^{(\ell)}(0,7n+5,a,\mu_7;0)\)
	\end{align*}
	for simplicity, where $\mathbf{1}_{\textrm{condition}}$ equals 1 if the condition is met and equals 0 otherwise. We also have the following vanishing conditions for all $c\in \Z$ with $p|c$, where $p=5$ or $7$ is denoted at the subscript of $C_p^{a,b}$: 
	\begin{align*}
		C_5^{2,4}\sin(\tfrac{\pi}5)S_{\infty\infty}^{(1)}(0,5n+1,c,\mu_5)+C_5^{4,2}\sin(\tfrac{2\pi}5)S_{\infty\infty}^{(2)}(0,5n+1,c,\mu_5)&=0,\tag{5-1}\\
		C_5^{0,4}\sin(\tfrac{\pi}5)S_{\infty\infty}^{(1)}(0,5n+2,c,\mu_5)+C_5^{0,2}\sin(\tfrac{2\pi}5)S_{\infty\infty}^{(2)}(0,5n+2,c,\mu_5)&=0,\tag{5-2}\\
		C_7^{4,6}S_7^{(1)}(7n,c)+C_7^{6,2}S_7^{(2)}(7n,c)+C_7^{2,4}S_7^{(3)}(7n,c)&=0,\tag{7-0}\\
		C_7^{2,4}S_7^{(1)}(7n+1,c)+C_7^{4,6}S_7^{(2)}(7n+1,c)+C_7^{6,2}S_7^{(3)}(7n+1,c)&=0,\tag{7-1,1}\\
		C_7^{4,6}S_7^{(1)}(7n+1,c)+C_7^{6,2}S_7^{(2)}(7n+1,c)+C_7^{2,4}S_7^{(3)}(7n+1,c)&=0,\tag{7-1,2}\\
		C_7^{0,6}S_7^{(1)}(7n+2,c)+C_7^{0,2}S_7^{(2)}(7n+2,c)+C_7^{0,4}S_7^{(3)}(7n+2,c)&=0,\tag{7-2}\\
		C_7^{0,4}S_7^{(1)}(7n+3,c)+C_7^{0,6}S_7^{(2)}(7n+3,c)+C_7^{0,2}S_7^{(3)}(7n+3,c)&=0,\tag{7-3,1}\\
		C_7^{2,6}S_7^{(1)}(7n+3,c)+C_7^{4,2}S_7^{(2)}(7n+3,c)+C_7^{6,4}S_7^{(3)}(7n+3,c)&=0,\tag{7-3,2}\\
		C_7^{0,2}S_7^{(1)}(7n+4,c)+C_7^{0,4}S_7^{(2)}(7n+4,c)+C_7^{0,6}S_7^{(3)}(7n+4,c)&=0,\tag{7-4,1}\\
		C_7^{2,6}S_7^{(1)}(7n+4,c)+C_7^{4,2}S_7^{(2)}(7n+4,c)+C_7^{6,4}S_7^{(3)}(7n+4,c)&=0,\tag{7-4,2}\\
		\(C_7^{0,4}+C_7^{2,6}\)S_7^{(1)}(7n+6,c)+\(C_7^{0,6}+C_7^{4,2}\)S_7^{(2)}(7n+6,c)&\\
		+\(C_7^{0,2}+C_7^{6,4}\)S_7^{(3)}(7n+6,c)&=0.\tag{7-6}
	\end{align*}

\end{theorem}

By \eqref{Relation between rank mod and A()} and Theorem~\ref{Kloosterman sums vanish}, we obtain the following corollary. 
\begin{corollary}[{\cite[Corollary~1.4]{QihangVanishingKLsums2025}}]
	\label{corollary Ramanujan congruence}
	For any pair {\rm{($p$-$k$)}} (or {\rm{($p$-$k$,$t$)}} for both $t=1,2$) with 
	\[p=5,\ k\in\{1,2,4\}\quad \text{or} \quad p=7,\ k \in\{0,1,2,3,4,5,6\},\]
	we have Dyson's conjecture {\rm{($p$-$k$)}} in Theorem~\ref{Dyson's conjectures in ASD}. 
\end{corollary}


The paper is organized as follows. In Section~\ref{Section: Definition and Notations} we define the notations for scalar-valued Maass forms, modular forms, and harmonic Maass forms. In Section~\ref{Section: V-val things} we introduce notations in vector-valued theory and define our vector-valued multiplier $\mu_p$. The proof of Theorem~\ref{main theorem} is included in Section~\ref{Section: Qihang's exact formulas}. We also include proofs of \eqref{KL sums match with Bringmann, cusp infty} in \S\ref{Subsection: Bringmann formula match, cusp infty} and of \eqref{KL sums match with Bringmann, cusp 0} in \S\ref{Subsection: Bringmann formula match, cusp 0}, which show that our exact formula \eqref{main theorem formula} aligns with Bringmann's asymptotic formula \eqref{Bringmann formula}. We generalize the estimates on sums of Kloosterman sums by Goldfeld and Sarnak \cite{gs} in Section~\ref{Section: sums of v-val KL sums} for our vector-valued Kloosterman sums. Such estimates ensure the convergence of the Fourier expansions of Maass-Poincar\'e series at $s=\frac 34$ in Section~\ref{Section: Qihang's exact formulas}.

\section{Scalar-valued theory}
\label{Section: Definition and Notations}
This section includes definitions and basic theorems in the theory of holomorphic modular forms, Maass forms and harmonic Maass forms. We focus on the half-integral weight $k\in \Z+\frac 12$ unless specified.  Let $\HH\defeq \{z\in \C:\ y=\im z>0\}$ denote the upper-half complex plane. 

\subsection{Multiplier systems and Kloosterman sums}
Fixing the argument $(-\pi,\pi]$, for any $\gamma\in\SL_2(\R)$ and $z=x+iy\in\HH$, we define the automorphic factor
\begin{equation}\label{j_factor_def}
	j(\gamma,z)\defeq \frac{cz+d}{|cz+d|}=e^{i\arg(cz+d)} 
\end{equation}
and the weight $k$ slash operator 
\begin{equation}\label{slash_k_operator_def}
	(f|_k\gamma)(z)\defeq j(\gamma,z)^{-k}f(\gamma z) 
\end{equation}
for $k\in \frac 12+\Z$. Let $\Gamma$ denote a congruence subgroup of $\SL_2(\Z)$ with $\begin{psmallmatrix}
	1&1\\0&1
\end{psmallmatrix}\in \Gamma$. 
\begin{definition}\label{multiplier system def}
	We say that $\nu:\Gamma\to \C^\times$ is a multiplier system of weight $k$ if
	\begin{enumerate}[label=(\roman*)]
		\item $|\nu|=1$,
		\item $\nu(-I)=e^{-\pi i k}$, and
		\item $\nu(\gamma_1 \gamma_2) =w_k(\gamma_1,\gamma_2)\nu(\gamma_1)\nu(\gamma_2)$ for all $\gamma_1,\gamma_2\in \Gamma$, where
		\[w_k(\gamma_1,\gamma_2)\defeq j(\gamma_2,z)^k j(\gamma_1,\gamma_2z)^k j(\gamma_1\gamma_2,z)^{-k}.\]
	\end{enumerate}
\end{definition}
If $\nu$ is a multiplier system of weight $k$, then it is also a multiplier system of weight $k'$ for any $k'\equiv k\pmod 2$, and its conjugate $\overline\nu$ is a multiplier system of weight $-k$. 
One can check the basic properties that for all $\gamma\in \Gamma$ and $b\in \Z$, we have
\begin{equation}\label{MultiplierSystemBasicProprety}
	\quad \nu(\gamma)\nu(\gamma^{-1})=1,\quad \nu(\gamma \begin{psmallmatrix}
		1&b\\0&1
	\end{psmallmatrix})=\nu(\gamma)\nu(\begin{psmallmatrix}
		1&1\\0&1
	\end{psmallmatrix})^b. 
\end{equation}

For any cusp $\mathfrak{a}$ of $\Gamma$, let $\Gamma_{\mathfrak{a}}$ denote its stabilizer in $\Gamma$. For example, $\Gamma_\infty=\{\pm\begin{psmallmatrix}
	1&b\\0&1
\end{psmallmatrix}:b\in \Z\}$. Let $\sigma_{\mathfrak{a}}\in\SL_2(\R)$ denote a scaling matrix of $\ma$, which means that $\sigma_\ma$ satisfies \begin{equation}\label{scaling matrix def}
	\sigma_{\mathfrak{a}}\infty=\mathfrak{a}\quad \text{and}\quad \sigma_{\mathfrak{a}}^{-1} \Gamma_{\mathfrak{a}}\sigma_{\mathfrak{a}}=\Gamma_\infty.
\end{equation}
We define $\alpha_{\nu,\mathfrak{a}} \in [0,1)$ by the condition
\begin{equation}\label{AlphaDefine}
	\nu\( \sigma_{\mathfrak{a}}\begin{psmallmatrix} 
		1 & 1 \\ 
		0 & 1
	\end{psmallmatrix}\sigma_{\mathfrak{a}}^{-1} \) = e(-\alpha_{\nu,\mathfrak{a}}).
\end{equation}
The cusp $\mathfrak{a}$ is called singular if $\alpha_{\nu,\mathfrak{a}}=0$. When $\mathfrak{a}=\infty$ we drop the subscript and denote $\alpha_{\nu}\defeq\alpha_{\nu,\infty}$. For $n\in \Z$, define $n_{\mathfrak{a}}\defeq n-\alpha_{\nu,\mathfrak{a}}$ and $n_{\infty}=\tilde{n} \defeq n-\alpha_{\nu}$.

The Kloosterman sums for the cusp pair $(\infty,\infty)$ with respect to $\nu$ are given by 
\begin{equation}
	\label{eq:kloos_def}
	S(m,n,c,\nu) :=\!\!\! \sum_{\substack{0\leq a,d<c \\ \gamma=\begin{psmallmatrix} a&b\\ c&d 
			\end{psmallmatrix}
			\in \Gamma}}\!\!\!  \overline\nu(\gamma) e\(\frac{ \tilde{m} a+ \tilde{n} d}{c}\)
	=\!\!\! \sum_{\substack{\gamma\in \Gamma_\infty\setminus\Gamma/\Gamma_\infty\\\gamma=\begin{psmallmatrix} a&b\\ c&d 
	\end{psmallmatrix}}} \!\!\!  \overline\nu(\gamma) e\(\frac{ \tilde{m} a+ \tilde{n} d}{c}\).
\end{equation}
They satisfy the relationships
\begin{equation}\label{KlstmSumConj}
	\overline{S(m,n,c,\nu)}=\left\{\begin{array}{ll}
		S(1-m,1-n,c,\overline \nu)&\ \  \text{if\ }\alpha_{\nu}>0,\\
		S(-m,-n,c,\overline\nu)  &\ \  \text{if\ }\alpha_{\nu}=0,
	\end{array}\right. 
\end{equation}
because 
\begin{equation}\label{tildeNConj}
	n_{\overline{\nu}}=\left\{\begin{array}{ll}
		-(1-n)_{\nu}&\ \  \text{if\ } \alpha_{\nu}>0,\\
		n  &\ \ \text{if\ }\alpha_{\nu}=0.
	\end{array}\right. 
\end{equation}

There are two fundamental multiplier systems of weight $\frac12$. The theta-multiplier $\nu_{\theta}$ on $\Gamma_{0}(4)$ is given by
\begin{equation}\label{thetaFunctionStandard}
	\theta(\gamma z) = \nu_{\theta}(\gamma) \sqrt{cz+d}\; \theta(z), \quad \gamma=\begin{psmallmatrix}
		a&b\\c&d
	\end{psmallmatrix}\in \Gamma_0(4)
\end{equation}
where
\[
\theta(z) \defeq \sum_{n\in\Z} e(n^2 z), \quad \nu_{\theta}(\gamma)=\(\frac cd\)\ep_d^{-1}, \quad \ep_d=\left\{ \begin{array}{ll}
	1&d\equiv 1\Mod 4,\\
	i&d\equiv 3\Mod 4, 
\end{array}\right.
\]
and $\(\frac \cdot\cdot\)$ is the extended Kronecker symbol. 
The eta-multiplier $\nu_\eta$ on $\SL_2(\Z)$ is given by
\begin{equation}
	\eta(\gamma z) = \nu_{\eta}(\gamma) \sqrt{cz+d}\; \eta(z), \quad \gamma=\begin{psmallmatrix}
		a&b\\c&d
	\end{psmallmatrix}\in \SL_2(\Z), 
\end{equation}
where
\begin{equation}
	\eta(z) \defeq q^{\frac1{24}}\prod_{n=1}^\infty (1-q^n),\quad q=e(z).
\end{equation}
Define
\[((x))\defeq \left\{
\begin{array}{ll}
	x-\floor{x}-\tfrac12,& \text{ when }x\in \R\setminus\Z,\\
	0,& \text{ when }x\in \Z. 
\end{array}
\right.
\]
We have the explicit formula \cite[(74.11), (74.12)]{Rad73Book}: 
\begin{equation}\label{etaMultiplier}
	\nu_\eta(\gamma)=e\(-\frac18\)e^{-\pi i s(d,c)}e\(\frac{a+d}{24c}\), \quad s(d,c)\defeq\sum_{r\Mod c}\(\(\frac rc\)\)\(\(\frac{dr}c\)\),
\end{equation}
for $c\in \Z\setminus \{0\}$. We also have $\nu_\eta\(\begin{psmallmatrix}
	1&b\\0&1
\end{psmallmatrix}\)=e\(\tfrac{b}{24}\)$. Another formula is given by \cite{Knopp70Book}: for $c>0$, we have
\begin{equation}\label{KnoppFmlEta}
	\nu_{\eta}(\gamma)=\left\{
	\begin{array}{ll}
		\(\tfrac dc\)e\left\{\tfrac1{24}\Big((a+d)c-bd(c^2-1)-3c\Big)\right\} & \text{if } c \text{ is odd,}\\
		\(\tfrac cd\)e\left\{\tfrac1{24}\Big((a+d)c-bd(c^2-1)+3d-3-3cd\Big)\right\} & \text{if } c \text{ is even.}
	\end{array}
	\right. 
\end{equation}
The properties $\nu_{\eta}(-\gamma)=i\nu_\eta(\gamma)$ when $c>0$ and $e(\frac{1-d}8)=(\frac 2d)\ep_d$ for odd $d$ are convenient. 

\subsection{Holomorphic modular forms}
Recall that we have fixed the argument in $(-\pi,\pi]$ and denoted $\Gamma$ as a congruence subgroup of $\SL_2(\Z)$ with $\begin{psmallmatrix}
	1&1\\0&1
\end{psmallmatrix}\in \Gamma$. For a weight $k$ multiplier system $\nu$ (Definition~\ref{multiplier system def}), let $M_k(\Gamma,\nu)$ (resp. $S_k(\Gamma,\nu)$) denote the space of holomorphic modular (resp. cusp) forms of weight $k$ for $(\Gamma,\nu)$ with the transformation law
\[f(\gamma z)=\nu(\gamma)(cz+d)^{k}f(z)\quad \text{for }\gamma=\begin{psmallmatrix}
	a&b\\c&d
\end{psmallmatrix}\in \Gamma.  \]
Recall $\tilde n=n-\alpha_{\nu}$. For $f\in M_k(\Gamma,\nu)$, $f$ has a Fourier expansion at the cusp $\infty$ given by
\begin{equation}
	f(z)=\sum_{n\in \Z,\ \tilde n\geq 0} a_f(n)e(\tilde n z). 
\end{equation}
We call $a_f(n)$ the Fourier coefficient of $f$ (at the cusp $\infty$). For any cusp $\ma$ of $\Gamma$, if we let $\sigma_\ma=\begin{psmallmatrix}
	A&B\\C&D
\end{psmallmatrix}$ denote a scaling matrix \eqref{scaling matrix def} of $\ma$, then we can write the Fourier expansion of $f$ at the cusp $\ma$ as
\begin{equation}
	(Cz+D)^{-k}f(\sigma_\ma z)=\sum_{n\in \Z,\ n_{\ma}\geq 0} a_{f,\ma}(n) e(n_\ma z). 
\end{equation}
Therefore, $f\in S_k(\Gamma,\nu)$ if and only if $a_{f,\ma}(0)=0$ for all cusps $\ma$ of $\Gamma$.  The spaces $M_k(\Gamma,\nu)$ and $S_k(\Gamma,\nu)$ are both finite dimensional for $k\in \frac 12\Z$.

Recall that we have already defined the two weight $\frac12$ multiplier systems: $\nu_\theta$ in \eqref{thetaFunctionStandard} and $\nu_\eta$ in \eqref{etaMultiplier}. Let $r(\psi)$ denote the conductor of a Dirichlet character $\psi$. When the weight $k=\frac12$, we have the Serre-Stark basis theorem: 

\begin{theorem}[{\cite[Corollary~1 of Theorem~A]{SerreStark}}]\label{Serre Stark}
	The space $M_{\frac12}(\Gamma_1(N),\nu_\theta)$ has a basis consisting of
	\[\theta_{\psi,t}(z)\defeq \sum_{n\in \Z}\psi(n)q^{tn^2},\]
	where $\psi$ is an {even} {primitive} Dirichlet character whose conductor $r(\psi)$ satisfies $4r(\psi)^2t|N$. 
\end{theorem}

\subsection{Maass forms}
\label{Subsection: Maass forms}
In this section we recall some basic facts about Maass forms with general weight and multiplier, which can be found in various references like \cite{Proskurin2005,pribitkin,DFI12,AAimrn,ahlgrendunn,QihangFirstAsympt}. Let $\Gamma$ denote our congruence subgroup with $\begin{psmallmatrix}
	1&1\\0&1
\end{psmallmatrix}\in \Gamma$.  Recall the definition of $j(\gamma,z)$ in \eqref{j_factor_def} and the definition of the weight $k\in \Z+\frac 12$ slash operator defined in \eqref{slash_k_operator_def}. 
We call a function $f:\HH\rightarrow \C$ automorphic of weight $k$ and multiplier $\nu$ on $\Gamma$ if 
\[f|_k\gamma =\nu(\gamma)f \quad\text{for all }\gamma\in\Gamma.\] Let $\mathcal{A}_k(\Gamma,\nu)$ denote the linear space consisting of all such functions and $\Lform_k(\Gamma,\nu)\subset \mathcal{A}_k(\Gamma,\nu)$ denote the space of square-integrable functions on $\Gamma\setminus\HH$ with respect to the measure \[d\mu(z)=\frac{dxdy}{y^2}\]
and the Petersson inner product 
\[\langle f,g\rangle \defeq\int_{\Gamma\setminus\HH} f(z)\overline{g(z)}\frac{dxdy}{y^2} \]
for $f,g\in\Lform_k(\Gamma,\nu)$. For $k\in \R$, the Laplacian
\begin{equation}\label{Laplacian}
	\Delta_k\defeq y^2\(\frac{\partial^2}{\partial x^2}+\frac{\partial^2}{\partial y^2}\)-iky \frac{\partial}{\partial x}
\end{equation}
can be expressed as
\begin{align}
	\Delta_k&=-R_{k-2}L_k-\frac k2\(1-\frac k2\)\\
	&=-L_{k+2}R_k+\frac k2\(1+\frac k2\)
\end{align}
where $R_k$ is the Maass raising operator
\begin{equation}
	R_k\defeq \frac k2+2iy\frac{\partial}{\partial z}=\frac k2+iy\(\frac{\partial}{\partial x}-i\frac{\partial}{\partial y}\)
\end{equation}
and $L_k$ is the Maass lowering operator
\begin{equation}\label{Maass Lowering op}
	L_k\defeq \frac k2+2iy\frac{\partial}{\partial \bar z}=\frac k2+iy\(\frac{\partial}{\partial x}+i\frac{\partial}{\partial y}\).
\end{equation}
These operators raise and lower the weight of an automorphic form as
\[(R_k f)|_{k+2}\;\gamma=R_k(f|_{k}\gamma),\quad (L_k f)|_{k-2}\;\gamma=L_k(f|_{k}\gamma),\quad  \text{for\ } f\in\mathcal{A}_k(\Gamma,\nu)\]
and satisfy the commutative relations
\begin{equation}\label{RLOperators}
	R_k\Delta_k=\Delta_{k+2}R_k,\quad L_k\Delta_k=\Delta_{k-2}L_k. 
\end{equation}
Moreover, $\Delta_k$ commutes with the weight $k$ slash operator for all $\gamma\in\SL_2(\R)$. 

We call a real analytic function $f:\HH\rightarrow \C$ an eigenfunction of $\Delta_k$ with eigenvalue $\lambda\in\C$ if
\[\Delta_k f+\lambda f=0. \]
From \eqref{RLOperators}, it is clear that an eigenvalue $\lambda$ for the weight $k$ Laplacian is also an eigenvalue for weight $k\pm 2$. 
We call an eigenfunction $f$ a Maass form if $f\in \mathcal{A}_k(\Gamma,\nu)$ is smooth and satisfies the growth condition
\[(f|_k\gamma)(x+iy)\ll y^\sigma+y^{1-\sigma} \]
for all $\gamma\in\SL_2(\Z)$ and some $\sigma$ depending on $\gamma$ when $y\rightarrow +\infty$. Moreover, if a Maass form $f$ satisfies
\[\int_0^1 (f|_k\sigma_{\mathfrak{a}})(x+iy)\;e(\alpha_{\nu,\mathfrak{a}}x)dx=0\]
for all cusps $\mathfrak{a}$ of $\Gamma$, then $f\in\Lform_k(\Gamma,\nu)$ and we call $f$ a Maass cusp form. For details see \cite[\S 2.3]{AAimrn}

Let $\mathcal{B}_k(\Gamma,\nu)\subset \Lform_k(\Gamma,\nu)$ denote the space of smooth functions $f$ such that both $f$ and $\Delta_k f$ are bounded. One can show that $\mathcal{B}_k(\Gamma,\nu)$ is dense in $\Lform_k(\Gamma,\nu)$ and $\Delta_k$ is self-adjoint on $\mathcal{B}_k(\Gamma,\nu)$. If we let $\lambda_0\defeq \lambda_0(k)=\tfrac {|k|}2(1-\tfrac {|k|}2)$, then for $f\in\mathcal{B}_k(\Gamma,\nu)$, we have
\[\langle f,-\Delta_k f\rangle\geq \lambda_0 \langle f,f\rangle,\]
i.e. $-\Delta_k$ is bounded from below. By the Friedrichs extension theorem, $-\Delta_k$ can be extended to a self-adjoint operator on $\Lform_k(\Gamma,\nu)$. 
The spectrum of $\Delta_k$ consists of two parts: the continuous spectrum $\lambda\in [\frac14,\infty)$ and a discrete spectrum of finite multiplicity contained in $[\lambda_0,\infty)$. 

Non-zero eigenfunctions corresponding to eigenvalue $\lambda_0$ come from holomorphic modular forms. To be precise, recall that $M_k(\Gamma,\nu)$ is the space of holomorphic modular forms of weight $k$ and multiplier $\nu$ on $\Gamma$. There is a one-to-one correspondence between all $f\in\Lform_k(\Gamma,\nu)$ with eigenvalue $\lambda_0$ and weight $k$ holomorphic modular forms $F$ by
\begin{equation}\label{CuspFormR0} 
	f(z)=\left\{\begin{array}{ll}
		y^{\frac k2}F(z)\quad & k\geq 0,\ F\in M_k(\Gamma,\nu),\\
		y^{-\frac k2}\overline{F(z)}\quad & k<0,\ F\in M_{-k}(\Gamma,\overline{\nu}). 
	\end{array}
	\right. 	\end{equation}
For the Fourier expansion $\sum_{n\in \Z} a_y(n)e(\tilde n x)$ of such $f$, we have
\begin{equation}\label{Coeffi CuspFormR0}
	\left\{
	\begin{array}{lll}
		k\geq 0 &\Rightarrow& a_y(n)=0 \text{\ for\ } \tilde n<0,\\
		k<0 &\Rightarrow& a_y(n)=0 \text{\ for\ } \tilde n>0.
	\end{array}
	\right.
\end{equation}


Let $\LEigenform_{k}(\Gamma,\nu)\subset \Lform_k(\Gamma,\nu)$ denote the subspace spanned by eigenfunctions of $\Delta_k$. For each eigenvalue $\lambda$, we write 
\[\lambda=\tfrac14+r^2=s(1-s), \quad s=\tfrac12+ir,\quad r\in i(0,\tfrac14]\cup[0,\infty). \]
So $r\in i\R$ corresponds to $\lambda<\frac14$ and any such $\lambda\in(\lambda_0,\frac 14)$ is called an exceptional eigenvalue. Set
\begin{equation}\label{FirstSpecPara}
	r_\Delta(N,\nu,k)\defeq i\cdot\sqrt{\tfrac14-\lambda_\Delta(\Gamma_0(N),\nu,k)}. 
\end{equation}

Let $\LEigenform_{k}(\Gamma,\nu,r)\subset\LEigenform_{k}(\Gamma,\nu)$ denote the subspace corresponding to the spectral parameter $r$. Complex conjugation gives an isometry 
\[\LEigenform_{k}(\Gamma,\nu,r)\longleftrightarrow\LEigenform_{-k}(\Gamma,\overline\nu,r)\]
between normed spaces. 
For each $v\in \LEigenform_{k}(n,\nu,r)$, we have the Fourier expansion
\[v(z)=v(x+iy)=c_0(y)+\sum_{ \tilde{n}\neq 0} \rho(n) W_{\frac k2\sgn \tilde{n},\;ir}(4\pi| \tilde{n}|y)e( \tilde{n} x)\]
where $W_{\kappa,\mu}$ is the Whittaker function as in \cite[(13.14.3)]{dlmf} and
\[c_0(y)=\left\{ 
\begin{array}{ll}
	0&\alpha_{\nu}\neq 0,\\
	0&\alpha_{\nu}=0 \text{ and } r\geq 0,\\
	\rho(0)y^{\frac12+ir}+\rho'(0)y^{\frac12-ir}&\alpha_{\nu}=0 \text{ and } r\in i(0,\frac14]. 
\end{array}
\right. 
\]
Using the fact that $W_{\kappa,\mu}$ is a real function when $\kappa$ is real and $\mu\in \R\cup i\R$ \cite[(13.4.4), (13.14.3), (13.14.31)]{dlmf}, if we denote the Fourier coefficient of $f_c\defeq \bar f$ as $\rho_c(n)$, then
\begin{equation}\label{FourierCoeffConj}
	\rho_c(n)=\left\{\begin{array}{ll}
		\overline{\rho(1-n)},&	\alpha_{\nu}>0,\ n\neq 0\\
		\overline{\rho(-n)}, &\alpha_{\nu}=0. 
	\end{array}
	\right.\end{equation}

\subsection{Harmonic Maass forms}
\label{Subsection: Harmonic Maass Forms}

The following construction can be found in \cite{BrmOno2006ivt,BringmannOno2012}. Let $k\in \frac12+\Z$, $z=x+iy$ for $x,y\in \R$ and $y\neq 0$, $s\in \C$, and $N$ be a positive integer with $4|N$. 
We define the weight $k$ hyperbolic Laplacian (different from \S\ref{Subsection: Maass forms}) by
\[\widetilde{\Delta}_k\defeq -y^2\(\frac{\partial^2}{\partial x^2}+\frac{\partial^2}{\partial y^2}\)+iky\(\frac{\partial}{\partial x}+i\frac{\partial}{\partial y}\). \]

\begin{definition}\label{harmonic Maass form definition}
	With the notations above, let $\chi$ be a Dirichlet character modulo $N$. A weight $k$ harmonic Maass form on $\Gamma_0(N)$ with Nebentypus $\chi$ is any smooth function $f:\HH\rightarrow \C$ satisfying: 
	
	(1) For all $\gamma\in \Gamma_0(N)$, we have $f(\gamma z)=\chi(d)\nu_\theta(\gamma)^{2k}(cz+d)^k f(z)$; 
	
	(2) $\widetilde{\Delta}_k f=0$;
	
	(3) There exists a polynomial $\mathcal{P}(z)=\sum_{n\leq 0} a^+(n) q^n$ with coefficients in $\C$ such that \[f(z)-\mathcal{P}(z)=O(e^{-Cy})\] for some $C>0$. Analogous conditions are required for all cusps. 
\end{definition}

\begin{remark}
	We denote the space of such harmonic Maass forms by $H_k(\Gamma_0(N),\chi\nu_\theta^{2k})$. The polynomial $\mathcal{P}$ is called the principal part of $f$ at the cusp $\infty$, with analogous definitions at other cusps. For a congruence subgroup $\Gamma$ of $\SL_2(\Z)$ with $\begin{psmallmatrix}
		1&1\\0&1
	\end{psmallmatrix}\in \Gamma$, when the transformation formula in condition (1) is replaced by $f(\gamma z)=\nu(\gamma)(cz+d)^k f(z)$ for all $\gamma\in \Gamma$ and some weight $k$ multiplier system $\nu$, then we call $f$ a weight $k$ harmonic Maass form for $(\Gamma,\nu)$ and the principal parts of $f$ are defined similarly for cusps of $\Gamma$. 
\end{remark}

Denote the anti-linear differential operator $\xi_k$ by
\[(\xi_k g)(z)\defeq 2iy^k\;\overline{\frac{\partial}{\partial \overline z}\,(g(z))}=R_{-k}(y^k\overline{g(z)}) \]
where $R_k$ is the Maass raising operator defined in \eqref{RLOperators}. If we let $G(z)=g(Bz)$ for some constant B, one can check that $(\xi_k G)(z)=B^{1-k}(\xi_k g)(Bz)$. We have the following lemmas. 
\begin{lemma}[{\cite[Proposition~3.2]{BruinierFunke},\cite[Lemma~2.2]{BringmannOno2012}}]\label{xiOperatorMapToCuspForms}
	The map 
	\[\xi_k: H_k(\Gamma_0(N),\chi\nu_\theta^{2k})\rightarrow S_{2-k}(\Gamma_0(N),\overline{\chi}\nu_{\theta}^{-2k})\]
	is a surjective map. Moreover, if $f\in H_k(\Gamma_0(N),\chi\nu_\theta^{2k})$ has Fourier expansion
	\[f(z)=\sum_{n\geq n_0}c_f^+(n)q^n+\sum_{n<0}c_f^-(n)\Gamma(1-k,4\pi|n|y)q^n\quad\ \text{for\ some\ }n_0\in \Z,\]
	then 
	\[(\xi_k f)(z)=-(4\pi)^{1-k}\sum_{n=1}^\infty \overline{c_f^-(-n)} n^{1-k}q^n. \]
\end{lemma}
\begin{remark}
We denote the holomorphic part of $f\in H_k(\Gamma_0(N),\chi\nu_\theta^{2k})$ by
		\[f_h(z)\defeq \sum_{n\geq n_0}c_f^+(n)q^n=\mathcal{P}(z)+\sum_{n>0}c_f^+(n)q^n\]
		and the non-holomorphic part of $f$ by
		\[f_{nh}(z)\defeq \sum_{n<0}c_f^-(n)\Gamma(1-k,4\pi|n|y)q^n,\]
		where $\Gamma(s,\beta)$ is the incomplete Gamma function defined by
		\[\Gamma(s,\beta)=\int_\beta^\infty t^{s-1}e^{-t}dt,\quad \beta>0. \]  
		We also define a mock modular form as the holomorphic part of a harmonic Maass form. 
\end{remark}

\begin{lemma}[{\cite[Lemma~2.3]{BringmannOno2012}}]\label{xiOperatorMapToCuspForms, 2}
	If $f\in H_{k}(\Gamma_0(N),\chi\nu_\theta^{2k})$ has the property that $\xi_{k}f\neq 0$, then the principal part of $f$ is non-constant for at least one cusp. 
\end{lemma}

\begin{remark}
	The proof of Lemma~\ref{xiOperatorMapToCuspForms} and Lemma~\ref{xiOperatorMapToCuspForms, 2} in \cite{BringmannOno2012} fully follows from \cite[\S3]{BruinierFunke}. Note that at the beginning of \cite[\S3]{BruinierFunke}, they only required that $\Gamma''$ is a subgroup of finite index. It is straightforward that Lemma~\ref{xiOperatorMapToCuspForms} and Lemma~\ref{xiOperatorMapToCuspForms, 2} both work if we change $H_{k}(\Gamma_0(N), \chi\nu_\theta^{2k})$ to $H_{k}(\Gamma_1(N), \nu_\theta^{2k})$ and change $S_{2-k}(\Gamma_0(N), \nu_\theta^{2k})$ to $S_{2-k}(\Gamma_1(N), \nu_\theta^{2k})$. 
\end{remark}

\section{Vector-valued theory}
\label{Section: V-val things}

For a vector or a matrix $M$, let $M^{\mathrm{T}}$ denote its transpose and $M^{\mathrm H}$ denote its conjugate transpose (Hermitian). Note that we are not using the language of Weil representations. 
\begin{notation}\label{VectorNotations}
	A boldface letter, e.g. $\mathbf{u}$ or $\mathbf{F}(z)$, always denotes a vector or a vector-valued function of some dimension $D\geq 2$, respectively. For $1\leq \ell\leq D$, let $\e_\ell\defeq (0,\cdots,0,1,0,\cdots,0)^{\mathrm T}$ denote the unit vector which has $1$ at its $\ell$-th entry and $0$ at the others.
	
	When the superscript $\cdot^{(\ell)}$ appears, $\mathbf{u}^{(\ell)}$, ${u}^{(\ell)}$, $\mathbf{F}^{(\ell)}(z)$ and $F^{(\ell)}(z)$ are defined by \[\mathbf{u}=\sum_{\ell=1}^{D}\mathbf{u}^{(\ell)}=(u^{(1)},u^{(2)},\cdots,u^{(D)})^\mathrm{T},\quad \mathbf{F}(z)=\sum_{\ell=1}^{D}\mathbf{F}^{(\ell)}(z)=\sum_{\ell=1}^{D}F^{(\ell)}(z)\e_\ell,  \]
	where $\mathbf{u}^{(\ell)}=u^{(\ell)}\e_\ell$ and $\mathbf{F}^{(\ell)}(z)=F^{(\ell)}(z)\e_\ell$. 
\end{notation}

Given two complex vectors $\mathbf{u},\mathbf{v}\in \C^{D}$, 
we define their inner product as $\mathbf{v}^{\mathrm H}\mathbf{u}=\sum_{\ell=1}^D u^{(\ell)}\overline{v^{(\ell)}}$. Let ${\mathrm M}_D(\C)$ denote the space of $D\times D$ complex matrices and take $M\in {\mathrm M}_D(\C)$. Then the inner product of $M\mathbf v$ and $\mathbf u$ is
\[(M\mathbf{v})^{\mathrm H}\mathbf{u}=\mathbf{v}^{\mathrm H}M^{\mathrm H}\mathbf{u}. \]

\subsection{Vector-valued Maass forms}
\label{Subsection: v-val Maass form}

Analogous to the scalar-valued case and to \cite{KnoppMasonVecValAndPoincare} for vector-valued modular forms, here we define vector-valued Maass forms on a congruence subgroup $\Gamma$ of $\SL_2(\Z)$ where $\begin{psmallmatrix}
	1&1\\0&1
\end{psmallmatrix}\in \Gamma$. 
As in the scalar-valued case, we fix the argument $(-\pi,\pi]$ and define the automorphic factor $j(\gamma, z)$ in \eqref{j_factor_def}. 
Denote our vector-valued function on the upper-half complex plane $\HH$ by
\[\mathbf{F}(z)=(F^{(1)}(z),F^{(2)}(z),\cdots,F^{(D)}(z))^T=\sum_{\ell=1}^{D} F^{(\ell)}(z) \e_\ell \]
and define the weight $k\in \R$ slash operator $|_k$ by
\[(\mathbf{F}|_k\gamma )(z)\defeq  \((F^{(1)}|_k \gamma)(z),\cdots,(F^{(D)}|_k\gamma)(z)\)^T\defeq j(\gamma,z)^{-k}\mathbf{F}(\gamma z). \]
\begin{definition}\label{vec val multi sys}
	For a congruence subgroup $\Gamma$ of $\SL_2(\Z)$ with $\begin{psmallmatrix}
		1&1\\0&1
	\end{psmallmatrix}\in \Gamma$, we say that $\xi:\Gamma\rightarrow \GL_{D}(\C)$ is a $D$-dimensional multiplier system if it satisfies the following compatibility conditions:
	\begin{itemize}
		\item[(1)] $\xi(\gamma)$ is a unitary matrix for all $\gamma\in \Gamma$, i.e. $\xi(\gamma)^{-1}=\xi(\gamma)^{\mathrm H}$. 
		\item[(2)] $\xi(-I)=e^{-\pi i k}I_{D}$. Here $I$ is the identity matrix in $\SL_2(\Z)$ and $I_{D}$ is the identity matrix in $\GL_{D}(\C)$. 
		\item[(3)] $\xi(\gamma_1\gamma_2)=w_k(\gamma_1,\gamma_2)\xi(\gamma_1)\xi(\gamma_2)$ for all $\gamma_1,\gamma_2\in \Gamma$, where
		\[w_k(\gamma_1,\gamma_2)\defeq j(\gamma_2,z)^k j(\gamma_1,\gamma_2 z)^k j(\gamma_1\gamma_2,z)^{-k}. \]
		\item[(4)] For every cusp $\ma$ of $\Gamma$, $\xi\(\sigma_{\ma}\begin{psmallmatrix}
			1&1\\0&1
		\end{psmallmatrix}\sigma_{\ma}^{-1}\)=\diag\{e(-\alpha_{\xi,\ma}^{(1)}),\cdots,e(-\alpha_{\xi,\ma}^{(D)})\}$ for some $\alpha_{\xi,\ma}^{(\ell)}\in [0,1)$. Here $\sigma_\ma$ is the scaling matrix of the cusp $\ma$ of $\Gamma$.  
	\end{itemize}
\end{definition}
\begin{remark}
	The multiplier system $\xi$ may not be a matrix representation of $\Gamma$ because when $k\notin \Z$, $w_k(\gamma_1,\gamma_2)$ may not always be $1$, hence $\xi$ is not multiplicative. In (4), if $\xi$ is clear from context, we will simply denote $\alpha_{\ma}^{(\ell)}=\alpha_{\xi,\ma}^{(\ell)}$. 
\end{remark}

For weight $k\in \R$ and a $D$-dimensional complex function $\mF$ such that each component $F^{(\ell)}$ is a smooth function, if
\[\Delta_k \mathbf{F}(z)+\lambda \mathbf{F}(z)=0\]
for some $\lambda\in \C$, then we call $\mathbf{F}$ a $D$-dimensional eigenfunction of $\Delta_k$ with eigenvalue $\lambda$. In this case, every component $F^{(\ell)}$ of $\mathbf{F}$ is a eigenfunction of $\Delta_k$ with eigenvalue $\lambda$.
\begin{definition}\label{vec val Maass form}
	A vector-valued Maass form $\mathbf{F}:\HH\rightarrow \C^{D}$ of weight $k\in \R$, eigenvalue $\lambda\in \C$ and $D$-dimensional multiplier system $\xi$ on $\Gamma$ is a vector-valued function $\mathbf{F}=(F^{(1)},\cdots,F^{(D)})$ satisfying:
	\begin{itemize}
		\item [(1)] Each $F^{(\ell)}$ is real-analytic on $\HH$;
		\item [(2)] $(\mathbf{F}|_k\gamma)(z)=\xi(\gamma)\mF(z)$ for all $\gamma\in \Gamma$. 
		\item [(3)] $(\Delta_k+\lambda) \mF(z)=0$. 
		\item [(4)] Each $F^{(\ell)}$ satisfies the growth condition 
		\[(F^{(\ell)}|_k\gamma) (z)\ll y^\sigma+y^{1-\sigma} \quad \text{for some }\sigma>0,\quad \text{as }y\rightarrow \infty\]
		for all $\gamma\in \SL_2(\Z)$. 
	\end{itemize}
	In addition, if $\mF$ also satisfies
	\begin{itemize}
		\item [(5)] For every $1\leq \ell\leq D$ and every cusp $\ma$ of $\Gamma$
		\[\int_0^1 (F^{(\ell)}|_k\sigma_{\mathfrak{a}})(x+iy)\;e(\alpha^{(\ell)}_{\mathfrak{a}}x)dx=0, \]
	\end{itemize}
	then we call $\mF$ a $D$-dimensional Maass cusp form. 
\end{definition}

Suppose $\mF$ is a weight $k$ Maass form with multiplier system $\xi$ on $\Gamma$.  By Definition~\ref{vec val Maass form}, each $e(\alpha_{\ma}^{(\ell)})(F^{(\ell)}|_k\sigma_\ma)(z)$ is a periodic function with period $1$ on $\HH$. Then $F^{(\ell)}$ admits a Fourier expansion at the cusp $\ma$ as
\[(F^{(\ell)}|_k\sigma_\ma)(x+iy)=\sum_{n\in \Z}c_\ma^{(\ell)}(n,y)e\(n_{\ma}^{(\ell)}x\),\quad \text{where }n_{\ma}^{(\ell)}\defeq n-\alpha_{\ma}^{(\ell)}.  \]
As in the classical case, since $F^{(\ell)}$ is an eigenfunction of $\Delta_k$ with eigenvalue $\lambda=\frac14+r^2$ for $r\in [0,\infty)\cup i[0,\infty)$, by solving the partial differential equation using the method of separation of variables, $F^{(\ell)}$ admits a Fourier expansion of the form
\begin{align}\label{Vec val Maass Fourier expansion}
	\begin{split}
		F^{(\ell)}(x+iy)=\rho_{F}^{(\ell)}(0,y)+\sum_{\substack{n\in \Z\\n_{\infty}^{(\ell)}\neq 0}} \rho_\infty^{(\ell)}(n)W_{\frac k2 \sgn n_{\infty}^{(\ell)},\;ir}(4\pi |n_{\infty}^{(\ell)}|y)e\(n_{\infty}^{(\ell)}x\), 
	\end{split}
\end{align}
where $\rho_F^{(\ell)}(0,y)=0$ if $n\neq 0$ or $\alpha_{\infty}^{(\ell)}\neq 0$, and $\rho_F^{(\ell)}(0,y)=c_1y^{\frac 12+ir}+c_2 y^{\frac 12-ir}$ for some constants $c_1,c_2\in \C$ if $n=\alpha_{\infty}^{(\ell)}=0$. Here $W_{\kappa,\mu}$ is the $W$-Whittaker function defined in \cite[(13.14.3)]{dlmf} which satisfies $W_{\kappa,\mu}(\alpha)\in \R$ and $W_{\kappa,\mu}=W_{\kappa,-\mu}$ if $\kappa,\alpha\in \R$ and $\mu\in \R\cup i\R$.

We call $\mF:\HH\rightarrow \C^{D}$ a vector-valued automorphic form of weight $k$ and $D$-dimensional multiplier system $\xi$ on $\Gamma$ if
\begin{equation}\label{vec val auto form}
	(\mF|_k\gamma)(z)=\xi(\gamma) \mF(z)\quad \text{for all } \gamma\in \Gamma. 
\end{equation}
and denote the linear space of all such automorphic forms as $\mathcal{A}_k(\Gamma,\xi)$. For $\mF,\mathbf{G}\in \mathcal{A}_k(\Gamma,\xi)$, we define (formally) their Petersson inner product by
\begin{equation}\label{vec val Petersson inner prod}
	\langle \mF,\mathbf{G}\rangle \defeq \int_{\Gamma\setminus \HH} \sum_{\ell=1}^{D} F^{(\ell)}(z)\overline{G^{(\ell)}(z)}\, \frac{dxdy}{y^2}=\int_{\Gamma\setminus \HH} \mathbf{G}^{\mathrm H}(z)\mathbf{F}(z)\frac{dxdy}{y^2}. 
\end{equation}
This inner product is well-defined: for all $\gamma\in \Gamma$, since $\frac{dxdy}{y^2}$ is invariant under $\gamma$, we have
\[\int_{\Gamma\setminus \HH} \mathbf{G}^{\mathrm H}(\gamma z) \mathbf{F}(\gamma z)\frac{dxdy}{y^2}=\int_{\Gamma\setminus \HH} \mathbf{G}^{\mathrm H}(z)\xi(\gamma)^{\mathrm H}\overline{j(\gamma, z)^k}j(\gamma,z)^k\xi(\gamma)\mathbf{F}(z)\frac{dxdy}{y^2}=\langle \mathbf {F},\mathbf{G}\rangle. \]
Let $\mathcal{L}_k(\Gamma,\xi)\subset \mathcal{A}_k(\Gamma,\xi)$ denote the Hilbert space of square-integrable functions under the above inner product. Then if $\mF\in \mathcal{L}_k(\Gamma,\xi)$, we have
\[\int_{\Gamma\setminus \HH} |F^{(\ell)}(z)|^2 \frac{dxdy}{y^2}<\infty\quad \text{for }1\leq \ell\leq D. \]

\subsection{\texorpdfstring{A representation on $\Gamma_0(p)$ twisted by $\overline{\nu_\eta}$}{A twisted representation}}
\label{Subsection: mup def}

In this section we review the notations and results in \cite{GarvanTransformationDyson2017}. Denote the $q$-Pochhammer symbol $(a,q)_n\defeq \prod _{j=0}^{n-1} (1-aq^{j-1})$.  For a prime $p\geq 5$ and $1\leq \ell\leq p-1$, define
\begin{equation}\label{M l p definition}
	M\(\frac{\ell}p;z\)\defeq \frac1{(q,q)_{\infty }}\sum_{n\in \Z}\frac{(-1)^nq^{n+\frac \ell p}}{1-q^{n+\frac \ell p}}q^{\frac 32 n(n+1)}
\end{equation}
and
\begin{equation}\label{N l p definition}
	N\(\frac{\ell}p;z\)\defeq \frac1{(q,q)_{\infty }}\(1+\sum_{n=1}^\infty\frac{(-1)^n (1+q^n)\(2-2\cos(\frac{2\pi \ell}p)\)}{1-2\cos(\frac{2\pi \ell}p)q^n+q^{2n}}q^{\frac 12 n(3n+1)}\).
\end{equation}
Note that 
\[N\(\frac \ell p;z\)=\mathcal{R}(\zeta_p^\ell;q),\]
which was defined in \eqref{rankGeneratingFunction} (see \cite[(2.4)]{GarvanTransformationDyson2017}). 
Further define
\begin{equation}\label{MN mathcal l p definition}
	\Mform\(\frac \ell p;z\)\defeq 2q^{\frac{3\ell}{2p}(1-\frac\ell p)-\frac 1{24}}M\(\frac \ell p;z\),\quad 
	\mathcal{N}\(\frac \ell p;z\)\defeq \csc\(\frac{\pi \ell}p\)q^{-\frac1{24}}N\(\frac \ell p;z\). 
\end{equation}
We also define the following functions: $\Mform(a,b,p,z)$ and $\mathcal{N}(a,b,p,z)$ as in \cite[(2.7), (2.8)]{GarvanTransformationDyson2017}; non-holomorphic functions $T_1(\frac \ell p;z)$, $T_2(\frac \ell p;z)$, $T_1(a,b, p;z)$ and $T_2(a,b,p;z)$ as in \cite[(3.1)-(3.4)]{GarvanTransformationDyson2017};
\begin{equation}\label{Epsilon2Def}
	\ep_2\(\frac \ell p;z\)\defeq \left\{\begin{array}{ll}
		2\exp\(-3\pi i z(\frac ac-\frac 16)^2\),&\ 0<\frac\ell p<\frac16,\\
		0,&\ \frac 16<\frac\ell p<\frac56,\\
		2\exp\(-3\pi i z(\frac ac-\frac 56)^2\),&\ \frac 56<\frac\ell p<1\\
	\end{array}
	\right. 
\end{equation}
and $\ep_2(a,b,p;z)$ as \cite[before Theorem~2.4]{GarvanTransformationDyson2017}; and 
\begin{align}
	\label{G1}
	\mathcal{G}_1\(\frac \ell p;z\)&\defeq \mathcal{N}\(\frac \ell p;z\)-T_1\(\frac \ell p;z\)=\csc\(\frac{\pi \ell}p\)q^{-\frac 1{24}}\mathcal{R}(\zeta_p^\ell;q)-T_1\(\frac \ell p;z\),\\
	\label{G2}
	\mathcal{G}_2\(\frac \ell p;z\)&\defeq \mathcal{M}\(\frac \ell p;z\)+\ep_2\(\frac \ell p;z\)-T_2\(\frac \ell p;z\),\\
	\mathcal{G}_1\(a,b, p;z\)&\defeq \mathcal{N}\(a,b,p;z\)-T_1\(a,b, p;z\),\\
	\mathcal{G}_2\(a,b,p;z\)&\defeq \mathcal{M}\(a,b,p;z\)+\ep_2(a,b,p;z)-T_2\(a,b,p;z\)
\end{align}
as in \cite[(3.5)-(3.8)]{GarvanTransformationDyson2017}. 
Bringmann and Ono proved the following result. 
\begin{proposition}[{\cite[Theorem~3.4]{BrmOno2010}, \cite[Corollary~3.2]{GarvanTransformationDyson2017}}]
	Suppose $p\geq 5$ is a prime. Then 
	\[\left\{\mathcal{G}_1\(\frac \ell p;z\),\mathcal{G}_2\(\frac \ell p;z\):1\leq \ell<p\right\}\cup \{\mathcal{G}_1\(a,b, p;z\),\mathcal{G}_1\(a,b, p;z\):0\leq a<p,\ 1\leq b<p\}\]
	is a vector valued Maass form of weight $\frac12$ for $\SL_2(\Z)$. 
\end{proposition}

We clarify the notations to use in the remaining paper. 
\begin{notation}\label{Congru inverse notations}
	For integers $A$ and $n>0$, let $[A]_{\{n\}}$ denote the least non-negative residue of $A\Mod n$. If the prime $p\geq 5$ is clear from the context, then we simply denote $[A]_{\{p\}}$ as $[A]$. 
	
	If $(A,n)=1$, let $\overline{A_{\{n\}}}$ denote the inverse of $A$ modulo $n$. Let $A'_n$ be defined by $AA'_n\equiv -1\Mod n$ if $n$ is odd and $AA'_n\equiv -1\Mod{2n}$ if $n$ is even. 
\end{notation} 

Garvan computed the following transformation laws on $\Gamma_0(p)$:
\begin{proposition}[{\cite[Theorem~4.1]{GarvanTransformationDyson2017}}]\label{mu cdlp def prop}
	Let $p\geq 5$ be a prime.  Then 
	\[\mathcal{G}_1\(\frac \ell p;\gamma z\)=\mu(c,d,\ell,p)\overline{\nu_\eta}(\gamma )(cz+d)^{\frac 12}\mathcal{G}_1\(\frac{[d\ell]}p;z\)\quad \text{for }\gamma=\begin{psmallmatrix}
		a&b\\c&d
	\end{psmallmatrix}\in \Gamma_0(p),
	\]
	where 
	\begin{equation}\label{mu cdlp def}
		\mu(c,d,\ell,p)\defeq \exp\(\frac{3\pi  i cd\ell^2}{p^2}\)(-1)^{\frac {c\ell}p}(-1)^{\floor{\frac{d\ell }p}}. 
	\end{equation}
\end{proposition}

Note that in the original paper \cite{GarvanTransformationDyson2017}, the notation for the right hand side of \eqref{mu cdlp def} was $\mu(\gamma,\ell)$ where $\gamma=\begin{psmallmatrix}
	a&b\\c&d
\end{psmallmatrix}$. Since it depends only on the value of $c$ and $d$, we modify it to $\mu(c,d,\ell,p)$ here for convenience. We use the above transformation formula to construct our $(p-1)$-dimensional multiplier system $\mu_p$.
\begin{definition}\label{mu matrix define}
	Let $p\geq 5$ be a prime. Define $M_p,\mu_p:\Gamma_0(p)\rightarrow \mathrm{M}_{p-1}(\C)$ by
	\[M_p(\gamma)\defeq\sum_{\ell=1}^{p-1}\mu(c,d,\ell,p)E_{\ell,[d\ell]} \quad \text{and}\quad \mu_p(\gamma)\defeq\overline{\nu_\eta}(\gamma)M_p(\gamma),\quad \text{for }\gamma=\begin{psmallmatrix}
		a&b\\c&d
	\end{psmallmatrix}\in \Gamma_0(p),\]
	where $E_{j,k}$ is the $(p-1)\times (p-1)$ matrix unit whose $(j,k)$-entry equals 1 and all the other entries equal 0. 
\end{definition}

This matrix has the following compatibility properties: 
\begin{proposition}\label{mu matrix property}
	Let $p\geq 5$ be a prime and $\mu_p$ be defined as in Definition~\ref{mu matrix define}. Then for all $\gamma,\gamma_1,\gamma_2\in \Gamma_0(p)$, we have
	\begin{itemize}
		\item[(1)] $\det (\mu_p(\gamma))=\pm\overline{\nu_\eta}(\gamma)^{p-1}$; 
		\item[(2)] $\mu_p(\gamma)^{-1}=\mu_p(\gamma)^\mathrm{H}$, i.e. $\mu_p(\gamma)$ is a unitary matrix;  
		\item[(3)] $\mu_p(-I)=e^{-\frac{\pi i}2}I_{p-1}$, where $I_{p-1}\in \mathrm{M}_{p-1}(\C)$ is the identity matrix; 
		\item[(4)] $\mu_p(\gamma_1\gamma_2)=w_{\frac12}(\gamma_1,\gamma_2)\mu_p(\gamma_1)\mu_p(\gamma_2)$. 
	\end{itemize}
\end{proposition}

\begin{proof}
	Since $\overline{\nu_\eta}$ is a weight $-\frac12$ multiplier system on $\SL_2(\Z)$, it suffices to prove the corresponding properties for $M_p$ in weight $1$. Write $\gamma=\begin{psmallmatrix}
		a&b\\c&d
	\end{psmallmatrix}$.  When $(d,c)=1$, we have $p\nmid d$ and for $1\leq \ell\leq p-1$, $d\ell$ runs over all residue classes modulo $p$ and vice versa. Thus, $\mu_p(\gamma)$ is a matrix with only one non-zero entry in every row and every column. Let $\sgn(\sigma)\in \{\pm 1\}$ be the signature of the permutation $\sigma:\ell\rightarrow [d\ell]$ in $(\Z/p\Z)^\times$. By \eqref{mu cdlp def}, we have
	\begin{align*}
		\det M_p(\gamma)&=\sgn(\sigma)\prod_{\ell=1}^{p-1} \mu(c,d,\ell,p)\\
		&=\sgn(\sigma)\exp\(\frac{3\pi i cd}{p^2}\cdot\frac{(p-1)p(2p-1)}6\)(-1)^{\frac cp\cdot \frac{(p-1)p}2}\prod_{\ell=0}^{p-1}(-1)^{\floor{\frac{d\ell}p}}\\
		&=\sgn(\sigma)(-1)^{c(d+1)\cdot\frac{p-1}2}\prod_{\ell=0}^{p-1}(-1)^{d\ell+[d\ell]}\\
		&=\sgn(\sigma)(-1)^{(c+1)(d+1)\cdot\frac{p-1}2}=\sgn(\sigma), 
	\end{align*}
	where we have used the following facts for any $x,y\in \Z$: $(-1)^{xp}=(-1)^{x}$; $\floor{\frac xp}\equiv x+[x]\Mod 2$; if $(x,y)=1$ then $(x+1)(y+1)$ is even. 
	
	For (2), it suffices to show that $M_p(\gamma)$ is a unitary matrix. Since $P_\sigma\defeq \sum_{\ell=1}^{p-1}E_{\ell,[d\ell]}$ is a permutation matrix with $P_\sigma^{-1}=P_\sigma^T$, we have
	\begin{align*}
		M_p(\gamma)^{-1}&=\Big(\diag\{\mu(c,d,\ell,p):1\leq \ell\leq p-1\}\cdot P_\sigma\Big)^{-1}\\
		&=P_{\sigma}^T\cdot \diag\{\overline{\mu(c,d,\ell,p)}:1\leq \ell\leq p-1\}=M_p(\gamma)^{\mathrm H}. 
	\end{align*}
	For (3), we have $\mu(0,-1,\ell,p)=-1$ for all $\ell$ and $p$. For (4), it suffices to show that $M_p:\Gamma_0(p)\rightarrow \GL_{p-1}(\C)$ is multiplicative, which was proved in \cite[Theorem~4.1]{GarvanTransformationDyson2017}. 
\end{proof}

Note that Proposition~\ref{mu matrix property} has established that $\mu_p$ meets all of the requirements in Definition~\ref{vec val multi sys} except (4), which we now verify. Since $\mu_p$ and $\overline{\mu_p}$ will appear together in the next section, for simplicity we denote $\alpha_{+\ma}^{(\ell)}\defeq \alpha_{\mu_p,\ma}^{(\ell)}$ and $\alpha_{-\ma}^{(\ell)}\defeq \alpha_{\overline{\mu_p},\ma}^{(\ell)}$ for the cusp $\ma$. 
Since 
\[\mu_p\(\begin{psmallmatrix}
	1&1\\0&1
\end{psmallmatrix}\)= e(-\tfrac 1{24}) I_{p-1}\]
is a diagonal matrix, we have
\begin{equation}\label{alpha infty l def}
	\alpha_{+\infty}^{(\ell)}=\tfrac1{24}\quad \text{ and }n_{+\infty}\defeq n-\tfrac1{24} \text{ for }n\in \Z. 
\end{equation}
For $\overline{\mu_p}$ we have $\alpha_{-\infty}^{(\ell)}=\frac{23}{24}$ and $n_{-\infty}\defeq n-\frac{23}{24}$. Moreover, we can take the scaling matrix (see \eqref{scaling matrix def}) of the cusp $0$ of $\Gamma_0(p)$ as $\sigma_0=\begin{psmallmatrix}
	0&-1/\sqrt p\\\sqrt p&0
\end{psmallmatrix}$. Since $\sigma_0 \begin{psmallmatrix}
	1&1\\0&1
\end{psmallmatrix}\sigma_0^{-1}=\begin{psmallmatrix}
	1&0\\-p&1
\end{psmallmatrix}$ and (by \eqref{KnoppFmlEta})
\begin{align*}
	\nu_\eta\(\sigma_0\begin{psmallmatrix}
		1&1\\0&1
	\end{psmallmatrix}\sigma_0^{-1}\)=i\nu_\eta \(\begin{psmallmatrix}
		-1&0\\p&-1
	\end{psmallmatrix}\)=i\(\tfrac {-1}p\)e\(-\tfrac {5p}{24}\)=e\(\tfrac{p}{24}\),
\end{align*}
we have
\[\mu_p\(\sigma_0\begin{psmallmatrix}
	1&1\\0&1
\end{psmallmatrix}\sigma_0^{-1}\)=\diag\left\{e\(-\tfrac {3\ell^2}{2p}-\tfrac{p}{24}\)(-1)^\ell:\ 1\leq \ell\leq p-1\right\}. \]
Therefore, we define $\alpha_{+0}^{(\ell)}\in [0,1)$ such that
\begin{equation}\label{alpha 0 l def}
	e\(-\alpha_{+0}^{(\ell)}\)=e\(-\tfrac {3\ell^2}{2p}-\tfrac{p}{24}\)(-1)^\ell \quad  \text{and define } n_{+0}^{(\ell)}\defeq n-\alpha_{+0}^{(\ell)} \text{ for } n\in \Z. 
\end{equation} 
Note that $\alpha_{+0}^{(\ell)}\neq 0$ because $1\leq \ell\leq p-1$ and $(p,24)=1$. 

We will need the following properties in Section~\ref{Section: Qihang's exact formulas} where we construct certain linear combinations of Maass-Poincar\'e series.  
For each integer $r\geq 0$, we denote
\begin{equation}\label{x r def}
	x_r\quad \text{as the only solution of }\quad \tfrac 32 x^2-(\tfrac 12+r)x+\tfrac1{24}=0 \quad \text{in }(0,\tfrac 12). 
\end{equation} 
The sequence $\{x_r:r\geq 0\}$ has $\frac 16=x_0>x_1>x_2>\cdots>0$. Fix the prime $p\geq 5$. For each integer $r\geq 0$ and positive integer $a$ with $(a,p)=1$, when $x_r>\frac 1p$ (otherwise the following set will be empty), we define the condition set
\begin{equation}\label{(a,r) def}
	\condar\defeq \Big\{1\leq \ell\leq p-1:\frac{[a\ell]}p \in (0,x_r)\cup(1-x_r,1)\Big\}\text{ \ and\ \ } \condr\defeq\vartriangleright \!1,r\!\vartriangleleft.  
\end{equation} 

\begin{remark}
Drawing the graph of $\frac 32 x^2-(\frac 12+r)x+\frac1{24}$ in $x\in (0,x_r)$ and the graph of $\frac 32 (1-x)^2-(\frac 12+r)(1-x)+\tfrac1{24}$ in $x\in (1-x_r,1)$, we can see two curved triangles which looks like our notation $\vartriangleright\ \vartriangleleft$. 
\end{remark}

In \eqref{delta l u a r def}, if we let $u=p$ as a prime, we have $p\nmid a $ and observe
\begin{equation}\label{delta l p a r def}
	\delta_{\ell,p,a,r} =\left\{\begin{array}{ll}
		-(\frac12+r)\frac {[a\ell]}{p}+\frac32\(\frac {[a\ell]}{p}\)^2+\frac1{24},&\ \text{if }0<\frac{[a\ell]}{p}<\frac 16,\vspace{0.3em}\\
		-\frac{5[a\ell ]}{2p}+\frac32\(\frac {[a\ell]}{p}\)^2+\frac{25}{24}-r+\frac{r[a\ell]}{p},&\ \text{if }\frac 56<\frac{[a\ell]}{p}<1,\vspace{0.3em}\\
		0&\ \text{otherwise}.  
	\end{array}
	\right. 
\end{equation}
Note that for the first line in \eqref{delta l p a r def}, if we change $\frac{[a\ell]}p$ to $1-\frac{[a\ell]}p$, then we get its second line. Therefore, we find that
\[\delta_{\ell,p,a,r}>0\quad \text{if and only if}\quad \ell\in \condar. \]
By \eqref{alpha 0 l def}, we find that $\alpha_{+0}^{(\ell)}$ is the fractional part of 
\begin{align}\label{alpha 0 l def, expressed in delta}
	\begin{split}
		\tfrac{3\ell ^2}{2p}-\tfrac{1+2r}2 \ell +\tfrac p{24},&\quad  \text{when }0<\tfrac{\ell}p<x_r,\\
		\tfrac{3p}2\(1-\tfrac \ell p\)^2-\tfrac{(1+2r)p}2\(1-\tfrac \ell p\)+\tfrac p{24}, &\quad  \text{when }1-x_r<\tfrac {\ell}p<1.  
	\end{split}
\end{align}
By \eqref{delta l p a r def}, we observe that $\alpha_{+0}^{(\ell)}$ is the fractional part of $p\delta_{\ell,p,1,r}$ when $\delta_{\ell,p,1,r}>0$. Hence, for every integer $r\geq 0$, we define a special vector $\mathbf{X}_r=(X_r^{(1)},\cdots,X_r^{(p-1)})^{\mathrm T}\in \Z^{p-1}$ such that 
\begin{equation}\label{X r define}
	X_r^{(\ell)}=\left\{\begin{array}{ll}
		\ceil{-p\delta_{\ell,p,1,r}},&\text{if }\delta_{\ell,p,1,r}>0,\text{ i.e. }\ell\in \condr,\\
		0,&\text{otherwise and never used}. 
	\end{array}
	\right.
\end{equation}
Then we have $X_r^{([a\ell])}=\ceil{-p\delta_{\ell,p,a,r}}$ by \eqref{delta l p a r def} and 
\[X_{r,+0}^{([a\ell])}\defeq X_{r}^{([a\ell])}-\alpha_{+0}^{([a\ell])}=-p\delta_{\ell,p,a,r} \quad  \text{when }\delta_{\ell,p,a,r}>0,\text{ i.e. when }\ell\in \condar. \]
Our $X_{r,+0}^{(\ell)}$ is the ``correct order" to be matched for the Maass-Poincar\'e series in Section~\ref{Section: Qihang's exact formulas}.

In general, for any vector $\mathbf{m}\in \Z^{p-1}$, we denote $m_{+0}^{(\ell)}\defeq m^{(\ell)}-\alpha_{+0}^{(\ell)}$. 
For simplicity we write $\mathbf{1}-\mathbf{m}\defeq \sum_{\ell=1}^\infty (1-m^{(\ell)})\e_\ell$ and denote $\mathbf{m}\leq 0$ if $m^{(\ell)}\leq 0$ for all $\ell$. For $\overline{\mu_p}$, we have $\alpha_{-0}^{(\ell)}=1-\alpha_{+0}^{(\ell)}\in (0,1)$ and define $m^{(\ell)}_{-0}=m^{(\ell)}_{-0}\defeq m^{(\ell)}-\alpha_{-0}^{(\ell)}$. We have the property
\begin{equation} 
	(1-m)_{\pm 0}^{(\ell)}=-m_{\mp 0}^{(\ell)}. 
\end{equation}

We have already proved Lemma~\ref{mu p is a vec val mult sys}. 
\begin{lemma}\label{mu p is a vec val mult sys}
	We have the following $(p-1)$-dimensional multiplier systems on $\Gamma_0(p)$: $\mu_p$ of weight $\frac12$ and $\overline{\mu_p}$ of weight $-\frac12$, in the sense of Definition~\ref{vec val multi sys}. 
\end{lemma}

In addition, by \cite[Corollary~4.2]{GarvanTransformationDyson2017}, or directly by \eqref{mu cdlp def}, one important property for $\mu_p$ is
\begin{equation}\label{mu p as eta on 0p^2}
	\mu_p(\gamma)=\overline{\nu_\eta}(\gamma) I_{p-1}\quad \text{for }\gamma\in \Gamma_0(p^2)\cap \Gamma_1(p).
\end{equation}
Suppose $\mF\in \mathcal{A}_{\frac 12}(\Gamma_0(p),\mu_p)$ (see \eqref{vec val auto form}), then for each $\ell$, we have $F^{(\ell)}\in \mathcal{A}_{\frac12}(\Gamma_0(p^2)\cap \Gamma_1(p),\overline{\nu_\eta})$. This fact allows us to use the notations for (scalar-valued) Maass forms in \S\ref{Subsection: Maass forms} here for vector-valued Maass forms with Petersson inner product defined in \eqref{vec val Petersson inner prod}. For example, $\Lform_{\frac 12}(\Gamma_0(p),\mu_p)$ is the space of weight $\frac 12$ vector-valued square-integrable functions in $\mathcal{A}_{\frac 12}(\Gamma_0(p),\mu_p)$. For any $\mF\in \Lform_{\frac12}(\Gamma_0(p),\mu_p)$, we have $F^{(\ell)}\in \Lform_{\frac12}(\Gamma_0(p^2)\cap \Gamma_1(p),\overline{\nu_\eta})$. 
It clearly follows that $\Delta_k$ is a self-adjoint operator on $\Lform_{\frac12}(\Gamma_0(p),\mu_p)$. The spectrum of $\Delta_{\frac12}$ on $\Lform_{\frac12}(\Gamma_0(p),\mu_p)$ is contained in the spectrum of $\Delta_{\frac 12}$ on $\Lform_{\frac12}(\Gamma_0(p^2)\cap \Gamma_1(p),\overline{\nu_\eta})$, which includes a discrete spectrum $\frac3{16}=\lambda_0\leq \lambda_1\leq \cdots $ of finite multiplicity and a continuous spectrum $[\frac 14,\infty)$. For each eigenvalue $\lambda$ of $\Delta_{\frac12}$ on $\Lform_{\frac12}(\Gamma_0(p),\mu_p)$, we write $\lambda=\frac14+r^2$ for $r\in i[0,\frac14)\cup [0,\infty)$ and call $r$ the spectral parameter.
We still denote $\LEigenform_{\frac12}(\Gamma_0(p),\mu_p,r)$ as the space of Maass eigenforms with spectral parameter $r$. With the property
\[\mathbf{F}\in \LEigenform_{\frac12}(\Gamma_0(p),\mu_p,r)\quad\Rightarrow\quad F^{(\ell)}\in \LEigenform_{\frac 12}(\Gamma_0(p^2)\cap \Gamma_1(p),\overline{\nu_\eta},r)\quad \text{for }1\leq \ell\leq p-1,  \]
we have Proposition~\ref{LEigenform mu p r is finite dim}. 
\begin{proposition}\label{LEigenform mu p r is finite dim}
	The space $\LEigenform_{\frac12}(\Gamma_0(p),\mu_p,r)$ is finite dimensional. 
\end{proposition}
\begin{proof}
	For any $\mathbf{V}(z,r)=\(V^{(1)}(z,r),\cdots,V^{(p-1)}(z,r)\)\in \LEigenform_{\frac12}(\Gamma_0(p),\mu_p,r)$, we have 
	\[V^{(1)}(z,r)\in \LEigenform_{\frac12}(\Gamma_0(p^2)\cap \Gamma_1(p),\overline{\nu_\eta},r),\]
	which is a finite-dimensional space. Moreover, for each $2\leq d\leq p-1$, we have $(d,p)=1$ and we can pick $\gamma_d=\begin{psmallmatrix}
		*&*\\p&d
	\end{psmallmatrix}\in \Gamma_0(p)$. By the definition of $\mu_p$, we have
	\[V^{(d)}(z,r)=\overline{\mu(p,d,1,p)}\nu_\eta(\gamma_d)V^{(1)}(z,r), \]
	i.e. the other components are determined by the first one. We conclude with
	\[\dim \LEigenform_{\frac12}(\Gamma_0(p),\mu_p,r)\leq \dim \LEigenform_{\frac12}(\Gamma_0(p^2)\cap \Gamma_1(p),\overline{\nu_\eta},r). \]
\end{proof}

We summarize these properties in the following lemma for future convenience. 
\begin{lemma}\label{Transformation laws by z to 24z}
	Suppose $\mF:\HH\rightarrow \C^{p-1}$ satisfies
	\[\mathbf{F}(\gamma z)=\mu_p(\gamma)(cz+d)^{\frac 12}\mathbf{F}(z)\quad \text{for }\gamma\in \Gamma_0(p). \]
	Then for each $\ell$, $1\leq \ell\leq p-1$, $F^{(\ell)}$ satisfies
	\[F^{(\ell)}(\gamma z)=\overline{\nu_\eta}(\gamma )(cz+d)^{\frac 12}F^{(\ell)}(z)\quad \text{for }\gamma\in \Gamma_0(p^2)\cap \Gamma_1(p).\]
	If we denote $\mathbf{G}(z)\defeq \mathbf{F}(24z)$ and hence $G^{(\ell)}(z)=F^{(\ell)}(24z)$, we have
	\[G^{(\ell)}(\gamma z)=\nu_\theta(\gamma)(cz+d)^{\frac12}G^{(\ell)}(z)\quad  \text{for }\gamma\in \Gamma_1(576p^2). \]
	Moreover, the map $z\rightarrow 24z$ gives an injection
	\begin{align*}
		\begin{array}{ccl}
			S_{\frac12}\(\Gamma_0(p^2)\cap \Gamma_1(p),\overline{\nu_\eta}\)&\rightarrow& S_{\frac12}(\Gamma_1(576p^2),\nu_\theta)\\
			f&\rightarrow& g \text{ defined by }g(z)\defeq f(24z).  
		\end{array}
	\end{align*}
\end{lemma}
\begin{proof}
	This is directly proved by our discussion above and the fact that
	\[w_{\frac12}\(\begin{psmallmatrix}
		a&24b\\c/24&d
	\end{psmallmatrix},\begin{psmallmatrix}
		\sqrt{24}&0\\0&1/\sqrt{24}
	\end{psmallmatrix}\)
	\overline{w_{\frac12}\(\begin{psmallmatrix}
			\sqrt{24}&0\\0&1/\sqrt{24}
		\end{psmallmatrix}, \begin{psmallmatrix}
			a&b\\c&d
		\end{psmallmatrix}\)}=1\]
	for $\gamma=\begin{psmallmatrix}
		a&b\\c&d
	\end{psmallmatrix}\in \Gamma_0(24)$.  Note that for $\gamma=\begin{psmallmatrix}
	a&b\\c&d
\end{psmallmatrix}\in \Gamma_1(576p^2)$, we have $d\equiv 1\Mod 4$, hence $\nu_\theta(\gamma)=\overline{\nu_\theta}(\gamma)=(\frac cd)$. 
\end{proof}

\subsection{Vector-valued Kloosterman sums}
\label{Subsection: v-val KL sums def}

In this subsection we define the vector-valued Kloosterman sums with $(k,\mu)=(\frac12,\mu_p)$ or $(-\frac12,\overline{\mu_p})$.  First we consider the cusp pair $\infty\infty$. Let $m,n,c\in \Z$ with $p|c$. Define
\begin{align}\label{S infty infty def}
	\mathbf{S}_{\infty\infty}(m,n,c,\mu)\defeq\!\!\!\!\sum_{\gamma=\begin{psmallmatrix}
			a&b\\c&d
		\end{psmallmatrix}\in \Gamma_\infty\setminus\Gamma_0(p)/\Gamma_\infty}\!\!\!\! e\(\frac{m_{\pm\infty} a+ n_{\pm\infty} d}c\)\mu(\gamma)^{-1}\sum_{\ell=1}^{p-1}\frac{\e_{\ell}}{\sin(\frac{\pi \ell}p)}. 
\end{align}
Since $\mu\(\begin{psmallmatrix}
	a& * \\ * &d
\end{psmallmatrix}\)^{-1}$ maps the entry at $[a\ell]$ to $\ell$, we extract the $\ell$-th entry of the vector $\mathbf{S}_{\infty\infty}(m,n,c,\mu)$ as
\begin{align}\label{S infty infty ell def}
	\begin{split}
		\mathbf{S}_{\infty\infty}^{(\ell)}(m,n,c,\mu)		&=\!\!\!\!\sum_{\substack{d\Mod c^*\\ad\equiv 1\Mod c}}\!\!\!\! e\(\frac{m_{\pm\infty} a+ n_{\pm\infty} d}c\)\frac{\mu(\begin{psmallmatrix}
				a&*\\c&d
			\end{psmallmatrix})^{-1}\e_{[a\ell]}}{\sin(\frac{\pi [a\ell]}p)}=:S_{\infty\infty}^{(\ell)}(m,n,c,\nu)\e_\ell.
	\end{split}
\end{align}

For the cusp pair $0\infty$ there are more requirements for our application. For every integer $r\geq 0$, recall our definitions for $x_r$ in \eqref{x r def}, $\alpha_{\pm 0}$ in \eqref{alpha 0 l def}, and $\condr$ in \eqref{(a,r) def}. For the cusp pair $(\ma,\mb)$, we define $\mu_{\ma\mb}(\gamma)$ for $\gamma\in \sigma_{\ma}^{-1}\Gamma_0(p)\sigma_\mb$ as in \cite[(3.4)]{iwaniecTopClassicalMF}. Hence, by $\sigma_\infty=I$ and $w_k(\gamma,I)=1$, $\mu_{0\infty}(\gamma)$ is defined for $\gamma\in \sigma_0^{-1}\Gamma_0(p)$ and given by
\begin{equation}\label{mu 0 infty def}
	\mu_{0\infty}(\gamma)= \mu(\sigma_{0}\gamma)w_k(\sigma_{0}^{-1},\sigma_{0}\gamma). 
\end{equation}
For every integer $r\geq 0$ and any vector $\mm\in \Z^{p-1}$, we define the Kloosterman sum
\begin{align}\label{S 0 infty def}
	\begin{split}
		\mathbf{S}_{0\infty}(\mathbf m,n,a,\mu;r)
		\defeq \sum_{\substack{\gamma=\begin{psmallmatrix}
					\frac c{\sqrt p}&\frac d{\sqrt p}\\-a\sqrt p&-b\sqrt p
				\end{psmallmatrix}\\\gamma\in \Gamma_\infty\setminus \sigma_0^{-1}\Gamma_0(p)/\Gamma_\infty}}\mu_{0\infty}(\gamma)^{-1}
		\sum_{\ell\in \condr}e\(\frac{m_{\pm0}^{(\ell)}\frac c{\sqrt p}-n_{\pm\infty} b\sqrt p}{-a\sqrt p}\)\e_{\ell}. 
	\end{split}
\end{align}
Note that $\sigma_0\gamma=\begin{psmallmatrix}
	a&b\\c&d
\end{psmallmatrix}\in \Gamma_0(p)$ for $\gamma\in \Gamma_\infty\setminus \sigma_0^{-1}\Gamma_0(p)/\Gamma_\infty$ in the summation above, hence $\mu_{0\infty}(\gamma)^{-1}$ maps the entry at $[a\ell]$ to $\ell$. Also note that only the values $m^{(\ell)}$ for $\ell\in \condar$ are used because $\ell\in \condar$ is equivalent to $[a\ell]\in\condr$. Therefore, by denoting $\gamma=\begin{psmallmatrix}
	\frac c{\sqrt p}&\frac d{\sqrt p}\\-a\sqrt p&-b\sqrt p
\end{psmallmatrix}$, the $\ell$-th entry of $\mathbf{S}_{0\infty}(\mathbf m,n,a,\mu;r)$ is 
\begin{align}\label{S 0 infty ell def}
	\begin{split}
		\mathbf{S}_{0\infty}^{(\ell)}(m^{([a\ell])},n,a,\mu;r)
		&=
		\left\{\begin{array}{ll}
			\displaystyle{\!\!\!\sum_{\substack{b\Mod a^*\\0<c<pa,\,p|c\\\text{s.t. }ad-bc=1}}\!\!\!\!\!\! \mu_{0\infty}(\gamma)^{-1}
				e\(\frac{m_{\pm 0}^{([a\ell])}\frac c{\sqrt p}-n_{\pm \infty} b\sqrt p}{-a\sqrt p}\)\e_{[a\ell]},}&\text{if }\ell\in \condar, \vspace{5px}\\
			0\e_\ell,&\text{otherwise}\\
		\end{array}
		\right.\\
		&=:S_{0\infty}^{(\ell)}(m^{([a\ell])},n,a,\mu_p;r)\e_\ell. 
	\end{split}
\end{align}

In Theorem~\ref{main theorem}, we pick $\mathbf{X}_r$ defined in \eqref{X r define} for every integer $r\geq 0$ and have 
\begin{equation}\label{S 0 infty def, using X r}
	\mathbf{S}_{0\infty}(\mathbf{X}_r,n,a,\mu_p;r)=\sum_{\substack{\gamma=\begin{psmallmatrix}
				\frac c{\sqrt p}&\frac d{\sqrt p}\\-a\sqrt p&-b\sqrt p
			\end{psmallmatrix}\\\gamma\in \Gamma_\infty\setminus \sigma_0^{-1}\Gamma_0(p)/\Gamma_\infty}}\mu_{0\infty}(\gamma)^{-1}
	\sum_{\ell\in \condr}e\(\frac{X_{r,+0}^{(\ell)}\frac c{\sqrt p}-n_{+\infty} b\sqrt p}{-a\sqrt p}\)\e_{\ell}. 
\end{equation}
We extract the $\ell$-th entry of the vector $\mathbf{S}_{0\infty}(\mathbf {X}_r,n,a,\mu_p;r)$: 
\begin{align}\label{S 0 infty ell def, using X r}
	\begin{split}
		\mathbf{S}_{0\infty}^{(\ell)}(X_r^{([a\ell])},n,a,\mu_p;r)
		&=
		\left\{\begin{array}{ll}
			\displaystyle{\!\!\!\sum_{\substack{b\Mod a^*\\0<c<pa,\,p|c\\\text{s.t. }ad-bc=1}}\!\!\!\!\!\! \mu_{0\infty}(\gamma)^{-1}
				e\(\frac{X_{r,+0}^{([a\ell])}\frac c{\sqrt p}-n_{+\infty} b\sqrt p}{-a\sqrt p}\)\e_{[a\ell]},}&\text{if }\ell\in \condar, \vspace{5px}\\
			0\e_\ell,&\text{otherwise}\\
		\end{array}
		\right.\\
		&=:S_{0\infty}^{(\ell)}(X_r^{([a\ell])},n,a,\mu_p;r)\e_\ell. 
	\end{split}
\end{align}

\subsection{Vector-valued holomorphic modular forms}
\label{Subsection: v-val holo mod forms}

Let $\nu$ be a weight $k=\pm \frac12$ (one-dimensional) multiplier system on the congruence subgroup $\Gamma$. Recall $M_k(\Gamma,\nu)$ as the space of weight $k$ holomorphic modular forms and $S_k(\Gamma,\nu)$ as the space of weight $k$ holomorphic cusp forms. Every $f\in M_k(\Gamma,\nu)$ satisfies the transformation property
\[f(\gamma z)=\nu(\gamma)(cz+d)^kf(z)\quad \text{for }\gamma=\begin{psmallmatrix}
	a&b\\c&d
\end{psmallmatrix}\in \Gamma. \]
Similarly, if $\mu$ is a weight $k=\pm \frac12$ $D$-dimensional multiplier system on $\Gamma$ (Definition~\ref{vec val multi sys}), then we denote the space of weight $k$ holomorphic modular forms on $(\Gamma,\mu)$ by $M_k(\Gamma,\mu)$ and the corresponding space of cusp forms by $S_k(\Gamma,\mu)$. 

From now on we fix the prime $p\geq 5$ and let $(k,\mu)=(\frac12,\mu_p)$ or $(-\frac12,\overline{\mu_p})$ on $\Gamma_0(p)$.  By Lemma~\ref{Transformation laws by z to 24z} and using the fact that $\alpha_{\pm \infty}\neq 0$ \eqref{alpha infty l def}, we have 
\begin{equation}
	\mathbf{F}\in M_{\frac12}(\Gamma_0(p),\mu_p)=S_{\frac12}(\Gamma_0(p),\mu_p)\quad \Rightarrow\quad F^{(\ell)}\in S_{\frac12}\(\Gamma_0(p^2)\cap \Gamma_1(p),\overline{\nu_\eta}\)
\end{equation}
for all $1\leq \ell\leq p-1$. 
As in Proposition~\ref{LEigenform mu p r is finite dim}, we also have
\begin{equation}
	\dim M_{\frac12}(\Gamma_0(p),\mu_p)\leq \dim S_{\frac12}\(\Gamma_0(p^2)\cap \Gamma_1(p),\overline{\nu_\eta}\). 
\end{equation}
Lemma~\ref{Transformation laws by z to 24z} also shows that, for any $f\in S_{\frac12}\(\Gamma_0(p^2)\cap \Gamma_1(p),\overline{\nu_\eta}\)\subseteq S_{\frac12}\(\Gamma_1(p^2),\overline{\nu_\eta}\)$, the map $z\rightarrow 24 z$ gives 
\begin{equation}\label{level lift cusp form 576p^2}
	g\in S_{\frac12}(\Gamma_1(576p^2),\nu_\theta) \quad \text{for } g(z)\defeq f(24z).  
\end{equation}

The Serre-Stark basis theorem (Theorem~\ref{Serre Stark}) implies the following lemma. 


\begin{lemma}\label{Theta function space only eta pz}
	Fix a prime $p\geq 5$. Let $f\in M_{\frac12}(\Gamma_1(576p^2),\nu_\theta)$ have the Fourier expansion
	\[f(z)=\sum_{n=0}^{\infty} a_f(n)q^n\] 
	where $a_f(m)=0$ for all $m\not \equiv -1\Mod{24}$. Then if $p\not\equiv -1\Mod{24}$, we have $f=0$; if $p\equiv -1\Mod{24}$, we have that $f$ is a multiple of $\eta(24pz)$.  
\end{lemma}

\begin{proof}
	By Theorem~\ref{Serre Stark}, if $\theta_{\psi,t}\in M_{\frac12}(\Gamma_1(576p^2),\nu_\theta)$ satisfies the condition in the lemma, then whenever $\psi(n)\neq 0$, i.e. whenever $(n,r(\psi))=1$, we have $tn^2\equiv -1\Mod {24}$. Since $t|144p^2$ and $p^2\equiv 1\Mod {24}$ for primes $p\geq 5$, we only have the possibility if $t=p\equiv -1\Mod {24}$. Then we have $r(\psi)|12$, hence $r(\psi)=1,3$ or $12$. 
	
	Since $\psi$ is primitive, $r(\psi)=1$ means $\psi(n)=1$ for all $n$, hence $\psi(2)=1$ while $p\cdot 2^2\not\equiv -1\Mod{24}$. When $r(\psi)=3$, the only primitive character is $(\frac{-3}\cdot)$ which is odd, not to mention $\psi(2)=-1$ and $p\cdot 2^2\not\equiv -1\Mod{24}$. 
	
	When $r(\psi)=12$, the only primitive character is $(\frac{12}\cdot)$. Note that 
	\[\eta(24z)=\sum_{n=1}^{\infty} \(\frac{12}n\)q^{n^2} \]
	and the lemma follows. 
\end{proof}

\begin{lemma}\label{mu p space is empty zero}
	If $\mathbf{F}=\sum_{\ell=1}^{p-1}F^{(\ell)}(z)\e_\ell\in M_{\frac12}(\Gamma_0(p),\mu_p)$ has Fourier expansion 
	\[F^{(\ell)}(z)=\sum_{n=1}^{\infty} a_F^{(\ell)}(n) q^{n-\frac1{24}}\quad \text{for each }1\leq \ell\leq p-1,\]
	then $\mathbf{F}=\mathbf{0}$. 
\end{lemma}

\begin{proof}
	Consider $\mathbf{F}(24z)$. By Lemma~\ref{Theta function space only eta pz} and \eqref{level lift cusp form 576p^2}, if $p\not\equiv -1\Mod {24}$ we already have the desired result. If $p\equiv -1\Mod{24}$, we have $F^{(\ell)}(z)=c^{(\ell)}\eta(pz)$ for some constant $c^{(\ell)}\in \C$ and for each $\ell$. By \cite[Corollary~3.5]{AhlgrenAndersenDicks23}, $\eta(pz)\in M_{\frac12}(\Gamma_0(p),(\frac \cdot p)\overline{\nu_\eta})$ because $p\equiv -1\Mod {24}$. 
	
	Now we take $\gamma=\begin{psmallmatrix}
		1&0\\p&1
	\end{psmallmatrix}$. By Proposition~\ref{mu cdlp def prop}, we have
	\[(pz+1)^{-\frac 12}F^{(\ell)}(\gamma z)= \mu(p,1,\ell,p)\overline{\nu_\eta(\gamma)}F^{(\ell)}(z),\]
	while by $F^{(\ell)}(z)=c^{(\ell)}\eta(pz)$, we have
	\[(pz+1)^{-\frac 12}F^{(\ell)}(\gamma z)=(\tfrac 1p)\overline{\nu_\eta(\gamma)}F^{(\ell)}(z). \]
	However, $\mu(p,1,\ell,p)=\exp(\frac{3\pi i \ell^2}p)(-1)^{\ell}$ cannot be $\pm 1$. Then the only possible case is $c^{(\ell)}=0$ for all $1\leq \ell\leq p-1$ and we have $\mathbf{F}=\mathbf{0}$. 
\end{proof}

\section{Proof of Theorem~\ref{main theorem}}
\label{Section: Qihang's exact formulas}

In this section we prove Theorem~\ref{main theorem}. We also prove \eqref{KL sums match with Bringmann, cusp infty} and \eqref{KL sums match with Bringmann, cusp 0} to show that this formula agrees with the asymptotic formula by Bringmann \cite{BringmannTAMS} in \eqref{Bringmann formula}. Recall our notations in Section~\ref{Section: V-val things}. 

For convenience, we define the principal part of our vector-valued functions here. For a vector-valued smooth function $\mathbf{P}(z)$ which satisfies $\mathbf{P}(\gamma z)=\mu_p(\gamma)(cz+d)^{\frac 12}\mathbf{P}(z)$ for $\gamma\in \Gamma_0(p)$, if there exist $\mathbf{R}_\infty(z)$ and $\mathbf{R}_0(z)$ such that $R_\infty^{(\ell)}(z),R_0^{(\ell)}(z)\in \C[q^{-1}]$ for $1\leq \ell\leq p-1$ and 
\begin{align*}
	P^{(\ell)}(z)-R_\infty^{(\ell)}(z)&=O(e^{-Cy}), \\
	(\sqrt p z)^{-\frac 12}{P}^{(\ell)}(-\tfrac{1}{pz})-R_0^{(\ell)}(z)&=O(e^{-Cy}) \quad \text{for }y\rightarrow\infty \text{ and some }C>0,
\end{align*}
then we call $\mathbf{R}_\infty(z)$ and $\mathbf{R}_0(z)$ the principal parts of $\mathbf{P}(z)$ at the cusps $\infty$ and $0$ of $\Gamma_0(p)$, respectively. Moreover, if the Fourier expansion of $\mathbf{P}(z)$ can be written as
\[\mathbf{P}(z) = \sum_{\ell=1}^{p-1} \sum_{n\geq \mathfrak M}a_{+}^{(\ell)}(n)q^{n_{\infty}}\e_\ell +\sum_{\ell=1}^{p-1} \sum_{n<0}a_{-}^{(\ell)}(n) \Gamma(\tfrac 12,4\pi |n_\infty| y)q^{n_\infty}\e_\ell,\]
for some $\mathfrak M\in \Z$, then the principal part of $\mathbf{P}(z)$ at the cusp $\infty$ is 
\begin{equation}\label{vec val Principal part}
	\mathbf{R}_\infty(z)=\sum_{\ell=1}^{p-1} \sum_{\mathfrak M\leq n \leq 0}a_{+}^{(\ell)}(n)q^{n_{\infty}}\e_\ell.
\end{equation}
We take $n\leq 0$ because $\alpha_{\infty}=\frac 1{24}>0$. The principal part of $\mathbf{P}(z)$ at the cusp $0$ is clearly the principal part of $(\sqrt p z)^{-\frac 12}\mathbf{P}(-\frac 1{pz})$ at the cusp $\infty$. 

For a prime $p\geq 5$ and $1\leq \ell\leq p-1$, recall $\mathcal{G}_1(\frac \ell p;z)$ defined in \eqref{G1}. Let the vector-valued function $\mathbf{G}_1(z;p)$ be defined as
\begin{equation}
	\mathbf{G}_1(z;p)\defeq \sum_{\ell=1}^{p-1} \mathcal{G}_1\(\frac \ell p;z\)\e_\ell. 
\end{equation}
In the following subsection \S\ref{Subsection: Principal parts of G1}, we determine the principal parts of $\mathbf{G}_1$ at the cusps $\infty$ and $0$ of $\Gamma_0(p)$. To match the principal parts, we construct Maass-Poincar\'e series in \S\ref{Subsection: v-val Maass Poincare series} and collect their properties in Proposition~\ref{Maass Poincare series are harmonic Maass form proposition}. Combining the results in \S\ref{Subsection: Principal parts of G1} and \S\ref{Subsection: v-val Maass Poincare series}, we finish the proof of Theorem~\ref{main theorem} in \S\ref{Subsection: Proof of main theorem}.

\subsection{\texorpdfstring{Principal parts of $\mathbf{G}_1$}{Principal parts of G1}}
\label{Subsection: Principal parts of G1}

Fix a prime $p\geq 5$ and let $1\leq \ell\leq p-1$. Recall the definition of $\mathcal{G}_1(\frac \ell p;z)$ in \eqref{G1} and $\mathcal{G}_2(\frac \ell p;z)$ in \eqref{G2}. Also recall that the holomorphic part of $\mathcal{G}_1(\frac \ell p;z)$ at the cusp $\infty$ is
\begin{equation}\label{G1 holomorphic part}
	\csc\(\frac{\pi\ell}p\)\sum_{n=0}^{\infty} A\(\frac \ell p;n\)q^{n-\frac1{24}}.
\end{equation}

By Proposition~\ref{mu cdlp def prop} and Definition~\ref{mu matrix define}, $\mathbf{G}_1(\cdot;p)$ has the property
\begin{equation}\label{G1 (bf) transformation}
	\mathbf{G}_1(\gamma z;p)=\mu_p(\gamma)(cz+d)^{\frac12}\mathbf{G}_1(z;p),\quad \text{for }\gamma=\begin{psmallmatrix}
		a&b\\c&d
	\end{psmallmatrix}\in \Gamma_0(p). 
\end{equation}
By \eqref{G1 holomorphic part} and \eqref{vec val Principal part}, the principal part of $\mathbf{G}_1(z;p)$ at the cusp $\infty$ of $\Gamma_0(p)$ is
\begin{equation}\label{Principal part of G1 at infty}
	\csc\(\frac{\pi\ell}p\)\sum_{\ell=1}^{p-1}  q^{-\frac 1{24}} \e_\ell. 
\end{equation}

By \cite[(3.13)]{GarvanTransformationDyson2017}, the behavior of $\mathcal{G}_1$ at the cusp $0$ is given by
\begin{equation}\label{Principal part of G1 at 0 written in G2}
	(\sqrt p\,z)^{-\frac12}\mathcal{G}_1\(\frac \ell p;\sigma_0 z\)=(\sqrt p\,z)^{-\frac12}\mathcal{G}_1\(\frac \ell p;-\frac1{pz}\)=e(-\tfrac 18)p^{\frac 14}\mathcal{G}_2\(\frac \ell p;pz\).
\end{equation}
By the discussion after \eqref{vec val Principal part}, the principal part of $\mathbf{G}_1(z;p)$ at the cusp $0$ can be derived from the principal parts of $\mathcal{G}_2(\frac \ell p;pz)$ at the cusp $\infty$ for $1\leq \ell\leq p-1$. 

Recall that $\ep_2$ is defined in \eqref{Epsilon2Def} by
\[\ep_2\(\frac \ell p;z\)=\left\{\begin{array}{ll}
	2q^{-\frac 32 \(\frac \ell p\)^2+\frac {\ell}{2p} -\frac1{24}}, &\ \frac \ell p\in (0,\frac 16),\\
	
	2q^{-\frac 32 \(1-\frac \ell p\)^2+\frac 12 \(1-\frac \ell p\)-\frac1{24}}, &\ \frac \ell p\in (\frac 56,1), \\
	0, &\ \text{otherwise}. 
\end{array}
\right.
\]
Here $\frac 16$ is the only root of the quadratic equation $-\frac 32 x^2+\frac 12 x-\frac1{24}=0$, hence the order of $\ep_2(\frac \ell p;z)$ at $\infty$ is less than 0 in the first two cases.
By \eqref{G2}, the holomorphic part of $\mathcal{G}_2(\frac \ell p;z)$ is
\begin{equation}\label{G2 holomorphic part}
	\ep_2\(\frac \ell p;z\)+2q^{-\frac 32 \(\frac \ell p\)^2+\frac {3\ell}{2p}-\frac1{24} }M\(\frac \ell p;z\). 
\end{equation}

Recall \eqref{x r def} that $x_r$ is the only solution in $(0,\frac 12)$ of the quadratic equation
\[-\tfrac 32 x^2+(\tfrac12+r)x-\tfrac1{24}=0.\]
Now $x_0=\frac 16$ and the contribution from $\ep_2(\frac \ell p;z)$ to the principal part of $\mathcal{G}_2(\frac\ell p;z)$ at $\infty$ is: 
\begin{equation}\label{Principal part of G1 at 0, r=0 so ep 2}
	\left\{\begin{array}{ll}
		2q^{-\frac 32 \(\frac \ell p\)^2+\frac{\ell}{2p}-\frac1{24}} &\ \text{when }0<\tfrac{\ell}{ p}<x_0,\\
		2q^{-\frac 32 \(1-\frac \ell p\)^2+\frac12\(1-\frac \ell p\)-\frac1{24}} &\ \text{when } 1-x_0<\tfrac{\ell}{ p}<1,\\
		0,&\text{otherwise.} 
	\end{array}
	\right. 
\end{equation}

For the principal part of $\mathcal{G}_2(\frac \ell p;z)$ at $\infty$ contributed from the part other than $\ep_2$, we need the Fourier expansion of $M(\frac \ell p;z)$ defined in \eqref{M l p definition}:
\begin{lemma}\label{Mlp First few terms}
	Let $p\geq 5$ be a prime. When $1\leq \ell\leq \frac{p-1}2$, the first few terms of the Fourier expansion of $M(\frac \ell p;z)$ are
	\[M\(\frac \ell p;z\)=\sum_{T=0}^{\floor{\frac{p}{2\ell}}}q^{\frac{T\ell}p}+O(q^{\frac 1{2}}). \]
	When $\frac{p+1}2\leq \ell\leq p-1$, we have $1\leq p-\ell\leq \frac{p-1}2$ and the first few terms of the Fourier expansion of $M(\frac \ell p;z)$ are
	\[M\(\frac \ell p;z\)=\sum_{T=0}^{\floor{\frac{p}{2(p-\ell)}}}q^{T\(1-\frac{\ell}p\)}+O(q^{\frac 1{2}}). \]
\end{lemma}

\begin{proof}
	It suffices to prove the first case $1\leq \ell\leq \frac{p-1}2$ because $\mathcal{M}(\frac \ell p;z)=\mathcal{M}(1-\frac \ell p;z)$ by \eqref{MN mathcal l p definition}. We have: 
	\begin{align*}
		M\(\frac \ell p;z\)&=\prod_{j=1}^\infty (1-q^j)^{-1}\sum_{n\in \Z} \frac{(-1)^nq^{n+\frac \ell p}}{1-q^{n+\frac \ell p}}q^{\frac 32 n^2+\frac 32 n}\\
		&=\(1+q+2q^2+O(q^3)\)\(\sum_{n\geq 0} (-1)^nq^{n+\frac \ell p}q^{\frac 32 n^2+\frac 32 n}\sum_{T=0}^\infty q^{T(n+\frac \ell p)}\right.\\
		&\quad \left .+\sum_{\substack{n<0\\m=-n}} (-1)^{m+1}q^{\frac 32 m^2-\frac 32 m}\sum_{T=1}^\infty q^{T(m-\frac{\ell}p)} \)\\
		&=\(1+O(q)\)\(\(q^{\frac \ell p}\sum_{T=0}^\infty q^{\frac{T\ell}p}+O(q)\)+\(1+\sum_{T=1}^\infty q^{T(1-\frac \ell p)}+O(q)\)\)\\
		&=\sum_{T=0}^{\floor{\frac{p}{2\ell}}} q^{\frac{T\ell}p}+O(q^{\frac12}). 
	\end{align*}
\end{proof}

\begin{proposition}\label{Principal part of G1 at 0, other than ep 2}
	Let $p\geq 5$ be a prime, let $x_r$ be defined in \eqref{x r def}, 
	and let $R$ be the maximal integer such that $x_R^{-1}<p$. 
	Then the sequence $\{x_r: r\geq 0\}$ is strictly decreasing and the principal part of $\mathcal{G}_2(\frac \ell p;z)$ contributed from the term involving $\mathcal{M}(\frac \ell p;z)$ (i.e. the part other than $\ep_2(\frac \ell p;z)$) equals
	\begin{equation}\label{Principal part of G1 at 0, other than ep 2, equation}
		\sum_{r=1}^R\left\{\begin{array}{ll}
			2q^{-\frac 32 \(\frac \ell p\)^2+\(\frac12+r\) \frac \ell p-\frac1{24}} &\ \text{when }0<\tfrac{\ell}{ p}<x_r,\\
			2q^{-\frac 32 \(1-\frac \ell p\)^2+\(\frac12+r\)\(1-\frac \ell p\)-\frac1{24}} &\ \text{when } 1-x_r<\tfrac{\ell}{ p}<1\\
			0& \ \text{otherwise.} 
		\end{array}
		\right. 
	\end{equation}
\end{proposition}

\begin{proof}
	By \eqref{G2 holomorphic part}, we see that the term
	\[q^{-\frac 32(\frac \ell p)^2+\frac {\ell}{2p}-\frac 1{24}}\cdot q^{\frac {r\ell}p}\]
	contributes to the principal part if and only if $0<\frac \ell p<x_r$, (otherwise the exponent will be positive). The analogous fact also holds if $\frac \ell p$ is replaced by $1-\frac \ell p$ and $x_r$ is replaced by $1-x_r$. 
	
	Recall Lemma~\ref{Mlp First few terms} for the Fourier expansion of $M(\frac \ell p;z)$. By \eqref{G2 holomorphic part}, 
	in order to prove \eqref{Principal part of G1 at 0, other than ep 2, equation}, it suffices to show that $\frac{r\ell} p<\frac 12$ for $0<\frac \ell p< x_r$. Since $x_r\in (0,\frac12)$, we have 
	\[rx_r=\frac 32 x_r^2-\frac 12 x_r+\frac1{24}=\frac 32\(x_r-\frac 16\)^2<\frac 32\times \(\frac 12-\frac 16\)^2=\frac 16.\]
	Thus $\frac {r\ell}p<rx_r< \frac 16<\frac12$. The analogous case for $1-\frac \ell p$ can be proved in a similar way. 
\end{proof}

Combining \eqref{Principal part of G1 at 0, r=0 so ep 2} and \eqref{Principal part of G1 at 0, other than ep 2, equation}, we get the principal part of $\mathcal{G}_2(\frac \ell p;z)$ at the cusp $\infty$: 
\begin{equation}\label{Principal part of G1 at 0, equation}
	\sum_{r=0}^R\left\{\begin{array}{ll}
		2q^{-\frac 32 \(\frac \ell p\)^2+\(\frac12+r\) \frac \ell p-\frac1{24}} &\ \text{when }0<\tfrac{\ell}{ p}<x_r,\\
		2q^{-\frac 32 \(1-\frac \ell p\)^2+\(\frac12+r\)\(1-\frac \ell p\)-\frac1{24}} &\ \text{when } 1-x_r<\tfrac{\ell}{ p}<1\\
		0& \ \text{otherwise,} 
	\end{array}
	\right. 
\end{equation}
where $R$ is the maximal integer such that $x_R^{-1}<p$.
\begin{remark}
	Here we give a hint about the relation between $r$ and the prime $p$. Since $x_0=\frac 16$, when $p\leq 5$, there is no principal part of $\mathcal{G}_2(\frac \ell p;z)$ at the cusp $\infty$, hence no principal part of $\mathbf{G}_1$ at the cusp $0$. This agrees with \cite[Remark~1 of Theorem~1.1]{BringmannTAMS} (even we have not proved Theorem~\ref{main theorem} yet). Since $1/{x_1}=34.9706\cdots$, for $7\leq p\leq 31$, we only have $r=0$. Here is a table for the first few conditions, where $[a,b]$ means the set of primes $p$ for $a\leq p\leq b$.  
	\begin{table}[!htbp]
		\centering
		\begin{tabular}{|c|c|c|c|c|c|c|c|}
			\hline
			Range of $p$ & $p=3,5$ & $[7,31]$ & $[37,59]$ & $[61,83]$ & $[89,107]$ & $[109,131]$ & $\cdots$\\
			\hline
			Allowed $r$ & No $r$ & $r=0$ & $r\leq 1$ & $r\leq 2$ & $r\leq 3$ & $r\leq 4$ & $\cdots$\\
			\hline
		\end{tabular}
	\end{table}
\end{remark}

\subsection{Vector-valued Maass-Poincar\'e series}
\label{Subsection: v-val Maass Poincare series}

In this subsection we construct Maass-Poincar\'e series to match the principal parts of $\mathbf{G}_1$. Let $M_{\beta,\mu}$ and $W_{\beta,\mu}$ denote the $M$- and $W$-Whittaker functions defined in \cite[(13.14.2-3)]{dlmf}. For $s\in \C$, $x,y\in \R$, and $k\in \Z+\frac 12$, we define 
\begin{equation}\label{Ms and varphi def}
	\mathcal{M}_s(y)\defeq |y|^{-\frac k2}M_{\frac k2 \sgn y,\,s-\frac 12}(|y|) \quad \text{and} \quad \varphi_{s,k}(x+iy)\defeq \mathcal{M}_s(4\pi y)e(x). 
\end{equation}
We also define 
\begin{equation}\label{Ws def}
	\mathcal{W}_s(y)\defeq |y|^{-\frac k2}W_{\frac k2 \sgn y,\,s-\frac 12}(|y|). 
\end{equation}
These functions have the following properties.  For $y>0$, by \cite[(13.18.4)]{dlmf} we have
\begin{equation}\label{Mfrack2 property}
	\mathcal{M}_{1-\frac k2}(-y)= y^{-\frac k2}M_{-\frac k2,\,\frac 12-\frac k2}(y)=(1-k)\(\Gamma(1-k)-\Gamma(1-k,y)\)e^{\frac y2}, 
\end{equation}
and by \cite[(13.18.2)]{dlmf}, we have
\begin{equation}\label{Wfrack2 property}
	W_{-\frac k2,\,\frac 12-\frac k2}(y)=y^{\frac k2}e^{\frac y2}\Gamma(1-k,y)\quad \text{and}\quad  	W_{\frac k2,\,\frac 12-\frac k2}(y)=y^{\frac k2}e^{-\frac y2}. 
\end{equation}
Moreover, $\varphi_{s,k}(z)$ is an eigenfunction of $\widetilde{\Delta}_k$ with eigenvalue $s(1-s)+\frac{k^2-2k}4$. Specifically, when $s=1-\frac k2$, we have 
\begin{equation}\label{varphik/2 vanished by harmonic Maass Laplacian}
	\widetilde{\Delta}_k \varphi_{1-\frac k2,k}=0. 
\end{equation}

From now on we fix the prime $p\geq 5$ and focus on our $(p-1)$-dimensional weight $k=\frac12$ multiplier system $\mu_p$ in Definition~\ref{mu matrix define}. For $m\in \Z$, recall $m_{+\infty}=m-\frac1{24}$ defined in \eqref{alpha infty l def}. Since we do not need the weight $-\frac12$ case of $\overline{\mu_p}$ and only have $\mu_p$ in this section, we simply write $m_{\infty}$ instead of $m_{+\infty}$. For $m_\infty<0$, we define the Maass-Poincar\'e series at the cusp $\infty$ by
\begin{equation}\label{Maass Poincare Series at infty}
	\mathbf{P}_\infty(z;p,s,\tfrac12,m,\mu_p)\defeq \frac{2}{\sqrt{\pi}}\sum_{ \ell=1}^{p-1}\sum_{\gamma\in \Gamma_\infty\setminus \Gamma_0(p)} \mu_p(\gamma)^{-1}\frac{\varphi_{s,k}(m_{\infty}\gamma z)}{(cz+d)^{\frac12}\sin(\tfrac{\pi\ell}p)}\e_\ell. 
\end{equation}
By \cite[Lemma~3.1]{BringmannOno2012}, this series is absolutely and uniformly convergent on any compact subset of $\re s>1$. 
The transformation formula for $\mathbf{P}_\infty(z;p,s,\tfrac12,m,\mu_p)$: 
\[\mathbf{P}_\infty(\gamma_1 z;p,s,\tfrac12,m,\mu_p)=\mu_p(\gamma_1)(Cz+D)^{\frac 12}\mathbf{P}_\infty(z;p,s,\tfrac12,m,\mu_p) \quad \text{for }\gamma_1=\begin{psmallmatrix}
	A&B\\C&D
\end{psmallmatrix}\in \Gamma_0(p)\]
can be proved similarly as \eqref{TransformationLawOfPoincare Cusp Infty}.

For an integer $r\geq 0$, recall the definition of $x_r$ in \eqref{x r def},  the stabilizer group $\Gamma_0$ of the cusp $0$ of $\Gamma_0(p)$: $\Gamma_0=\{\pm\begin{psmallmatrix}
	1&0\\c&1
\end{psmallmatrix}: c\in p\Z\}$ and the scaling matrix $\sigma_0=\begin{psmallmatrix}
	0&-1/\sqrt{p}\\\sqrt{p}&0
\end{psmallmatrix}$. Recall the notations $\mathbf{X}_r$ and $X_{r,+0}^{(\ell)}$ in \eqref{X r define}. We denote $X_{r,0}^{(\ell)}$ instead of $X_{r,+0}^{(\ell)}$ for simplicity. Recall the notations $\condr$ and $\condar$ in \eqref{(a,r) def}.

For every integer $r\geq 0$, we define the Maass-Poincar\'e series at the cusp $0$ by 
\begin{align}\label{Maass Poincare Series at zero}
	\begin{split}
		&\mathbf{P}_0(z;p,s,\tfrac12,r,\mu_p)\\
		&\defeq \frac{2 e(-\frac 18)p^{\frac 14}}{\sqrt{\pi}}\sum_{\ell\in\condr }\sum_{\substack{\gamma\in \Gamma_0\setminus \Gamma_0(p)\\
				\gamma=\begin{psmallmatrix}
					a&b\\c&d
		\end{psmallmatrix}}} \mu_p(\gamma)^{-1}\overline{\omega_{\frac12}(\sigma_0^{-1},\gamma)}\frac{\varphi_{s,k}\(X_{r,0}^{(\ell)}\sigma_0^{-1}\gamma z\)}{(-a\sqrt p z-b\sqrt p)^{\frac12}}\e_\ell.  
	\end{split}
\end{align}
Note that $\sigma_0^{-1}\gamma=\begin{psmallmatrix}
	\frac c{\sqrt p}&\frac d{\sqrt p}\\-a\sqrt p & -b\sqrt p
\end{psmallmatrix}$. By \cite[Lemma~3.1]{BringmannOno2012}, the above series is absolutely and uniformly convergent on any compact subset of $\re s>1$. 
The transformation formula for $\mathbf{P}_0(z;p,s,\tfrac12,r,\mu_p)$: 
\[\mathbf{P}_0(\gamma_1 z;p,s,\tfrac12,r,\mu_p)=\mu_p(\gamma_1)(Cz+D)^{\frac 12}\mathbf{P}_0(z;p,s,\tfrac12,r,\mu_p) \quad \text{for }\gamma_1=\begin{psmallmatrix}
	A&B\\C&D
\end{psmallmatrix}\in \Gamma_0(p)\]
can be proved similarly as \eqref{TransformationLawOfPoincare Cusp 0}.

\subsubsection{Fourier expansions of \texorpdfstring{$\mathbf{P}_{\infty}$ at $\infty$}{Poincar\'e series (defined at the cusp infinity) at infinity}}

Here we compute the Fourier expansions of\\ $\mathbf{P}_\infty(z;p,s,\tfrac12,m,\mu_p)$ at $s=\frac 34$. It is important to note that we only have the absolute and uniform convergence for $\re s>1$ by definition. However, the Fourier expansion in the following theorem is guaranteed to be convergent by Proposition~\ref{convergence in general, main contribution} when $s=\frac 34$. By analytic continuation, $\mathbf{P}_\infty(z;p,s,\tfrac12,m,\mu_p)$ is convergent at $s=\frac 34$ and has the Fourier expansion as below. The proof of Proposition~\ref{convergence in general, main contribution} is independent from this section. 

There are similar arguments in \cite[Proof of Theorem~3.1]{BrmOno2006ivt} for $p=2$ and \cite[Theorem~4.3]{QihangFirstAsympt} for $p=3$. 

\begin{proposition}\label{Fourier exp of Maass Poincare at infty}
	When $m_{\infty}<0$, 
	the Maass-Poincar\'e series $\mathbf{P}_\infty(z;p,s,\tfrac12,m,\mu_p)$ is convergent at $s=\frac 34$, 
	and we have the following Fourier expansion: 
	\begin{align*}
		\mathbf{P}_\infty(z;p,\tfrac 34,\tfrac12,m,\mu_p)&=\sum_{\ell=1}^{p-1}\(1-\frac{\Gamma(\frac12,4\pi |m_{\infty}|y)}{\sqrt{\pi}}\)\frac{q^{ m_{\infty}}}{\sin(\frac{\pi \ell}p)}\e_\ell\\
		&+\sum_{n_{\infty}>0} \mathbf{B}_\infty(n) q^{n_{\infty}}+\sum_{n_{\infty}<  0} \mathbf{B}'_\infty(n)\frac{\Gamma\(\tfrac 12,4\pi |n_{\infty}| y\)}{\sqrt \pi}q^{n_{\infty}},
	\end{align*}
	where
	\begin{align}
		\left. \begin{array}{r}
			\mathbf{B}_\infty(n)\\
			\mathbf{B}'_\infty(n)
		\end{array}\right\}=2\pi e(-\tfrac 18)\left|\frac{m_{\infty}}{n_{\infty}}\right|^{\frac 14} \sum_{\ell=1}^{p-1}\sum_{N|c>0}\frac{\mathbf{S}_{\infty\infty}^{(\ell)}(m,n,c,\mu_p)}c \left\{\begin{array}{l}
			I_{\frac12}\(\dfrac{4\pi|m_{\infty}n_{\infty}|^{\frac12}}c\)\\
			J_{\frac12}\(\dfrac{4\pi|m_{\infty}n_{\infty}|^{\frac12}}c\). 
		\end{array}\right.
	\end{align}
	Here $\mathbf{S}_{\infty\infty}^{(\ell)}(m,n,c,\mu_p)$ is defined by \eqref{S infty infty ell def} and its scalar value can be written as
	\begin{align}\label{S infty infty def, in scalar val}
		S_{\infty\infty}^{(\ell)}(m,n,c,\mu_p)=e(-\tfrac 18)\!\!\!\!\sum_{\substack{d\Mod c^*\\ad\equiv 1\Mod c}}\!\!\!\!\frac{\overline{\mu(c,d,[a\ell],p)}}{\sin(\frac{\pi[a\ell]}p)}e^{-\pi i s(d,c)}e\(\frac{ma+nd}c\). 
	\end{align}
\end{proposition}

\begin{proof}
	The following process is well-known and we provide details for completeness. Recall the properties of Whittaker functions from \eqref{Ms and varphi def} to \eqref{varphik/2 vanished by harmonic Maass Laplacian}.
	
	The contribution to $\mathbf{P}_\infty(z;p,s,\frac12,m,\mu_p)$ from $c=0$ equals
	\[\frac2{\sqrt \pi}\sum_{\ell=1}^{p-1}\csc(\tfrac{\pi \ell}p)\varphi_{s,\frac12}( m_\infty z)\e_\ell. \]
	When $s=\frac 34$, by \eqref{Mfrack2 property}, this contribution is
	\[\frac2{\sqrt \pi}\sum_{\ell=1}^{p-1}\csc(\tfrac{\pi \ell}p)\varphi_{\frac 34,\frac12}( m_\infty z)\e_\ell=\(1-\frac{\Gamma(\frac12,4\pi| m_\infty |y)}{\sqrt\pi}\)\sum_{\ell=1}^{p-1}\csc(\tfrac{\pi \ell}p)e^{2\pi  m_\infty z}\e_\ell. \]
	
	Recall (Definition~\ref{mu matrix define}) that $\mu_p\(\begin{psmallmatrix}
		a&* \\ *&d
	\end{psmallmatrix}\)^{-1}$ maps the value at the $[a\ell]$-th entry to the $\ell$-th entry. Using the properties \eqref{MultiplierSystemBasicProprety} for $\nu_\eta$ and Proposition~\ref{mu matrix property} for $\mu_p$ and $M_p$, for $\re s>0$, the contribution to $\mathbf{P}_\infty(z;p,s,\frac12,m,\mu_p)$ from some $c>0$ equals
	\begin{align}\label{P infty contribution c>0}
		\begin{split}
			\frac2{\sqrt \pi}\sum_{\ell=1}^{p-1}&\sum_{t\in \Z}\sum_{d(c)^*}\mu_p^{-1}
			\begin{psmallmatrix}
				a&b+ta\\c&d+tc
			\end{psmallmatrix}
			(cz+d+tc)^{-\frac12} \csc(\tfrac{\pi \ell}p)\\
			&\cdot\mathcal{M}_{s}\(\frac{4\pi \tilde{m}y}{|cz+d+t c|^2}\)e\(\frac{ m_\infty a}c- \re\(\frac{ m_\infty  }{c(cz+d+tc)}\) \)\e_\ell\\
			=\ &\frac2{\sqrt{\pi c}}\sum_{\ell=1}^{p-1}\sum_{\substack{d(c)^*}} \frac{\overline{\mu(c,d,[a\ell],p)}}{\sin(\frac{\pi[a\ell]}p)}\,\nu_\eta(\begin{psmallmatrix}
				a&b\\c&d
			\end{psmallmatrix})e\(\frac{ m_\infty a}c\) \sum_{t\in \Z}  e\(t\alpha_{\infty}\) \(z+\frac dc+t\)^{-\frac12} \\
			&\cdot\mathcal{M}_{s}\(\frac{4\pi \tilde{m}y}{c^2|z+\frac dc+t|^2}\)e\(-\frac{ m_\infty  } {c^2} \re\(\frac1{z+\frac dc+t}\) \)\e_{\ell}.
		\end{split}
	\end{align}
	Here we use $\sum_{d(c)^*}$ to abbreviate the following summation condition: for $d\Mod c^*$, we choose $a$ by $ad\equiv 1\Mod c$ and  $b$ by $ad-bc=1$. 
	
	Let
	\[f(z)\defeq \sum_{t\in \Z} \frac{e(t\alpha_{\infty})}{(z+t)^{\frac 12} } \;\mathcal{M}_{s}\(\frac{4\pi m_{\infty}y}{c^2|z+t|^2}\)e\(-\frac{ m_\infty  } {c^2} \re\(\frac1{z+t}\) \).\]
	Then $f(z)e(\alpha_{\infty}x)$ has period $1$ and $f$ has Fourier expansion
	\begin{equation}\label{fzdc_vval}
		f(z)=\sum_{n\in \Z}a_y(n)e( n_\infty  x)\quad \text{and}\quad f\(z+\frac dc\)=e\(\frac { n_\infty  d}c\)f(z).  
	\end{equation}
	Here by \cite[Proof of Theorem~1.9]{BruinierBookBorcherds}, we can compute
	\begin{align*}
		&a_y(n)=\frac{e(-\frac18)\Gamma(2s)}{|4\pi  m_\infty  y|^{\frac 14}\sqrt{c}}\\
		&\cdot\left\{ \begin{array}{ll}
			{\displaystyle \frac{2\pi}{\Gamma(s-\frac 14)}\left|\frac{ m_\infty }{ n_\infty }\right|^{\frac12}\, W_{-\frac 14,s-\frac 12}(4\pi | n_\infty | y) J_{2s-1}\(\dfrac{4\pi| m_\infty  n_\infty |^{\frac12}}c\),}&  n_\infty <0;\\
			{\displaystyle \frac{2\pi}{\Gamma(s+\frac 14)} \left|\frac{ m_\infty }{ n_\infty }\right|^{\frac12}\, W_{\frac 14,s-\frac 12}(4\pi  n_\infty  y) I_{2s-1}\(\dfrac{4\pi| m_\infty  n_\infty |^{\frac12}}c\),}&  n_\infty >0.\\
		\end{array}
		\right. 
	\end{align*}
	Thus, for $\re s>1$, we have the Fourier expansion of $\mathbf{P}_{\infty}(z;p,s,\tfrac 12,m,\mu_p)$: 
	\begin{align*}
		&\mathbf{P}_{\infty}(z;p,s,\tfrac 12,m,\mu_p)=\frac2{\sqrt \pi}\sum_{\ell=1}^{p-1}\frac{\varphi_{s,\frac 12}(z)}{\sin(\frac{\pi \ell}p)}\e_\ell
		+\sum_{n\in \Z} e^{2\pi i  n_\infty z }\frac{2\Gamma(2s)e(-\tfrac18)|m_{\infty}|^{\frac 14}}{\sqrt \pi|n_\infty|^{\frac 12}|4\pi y|^{\frac 14}} \\
		&\ \cdot \left\{
		\begin{array}{ll}
			{\displaystyle
				\frac{2\pi W_{-\frac14,s-\frac12}(4\pi |n_\infty|y)}{\Gamma(s-\frac 14)}\sum_{\ell=1}^{p-1}\sum_{p|c>0}\dfrac{\mathbf{S}_{\infty\infty}^{(\ell)}(m,n,c,\mu_p)}c J_{2s-1}\(\frac{4\pi| m_\infty  n_\infty |^{\frac12}}c\),}&  n_\infty <0;\\
			{\displaystyle
				\frac{2\pi W_{\frac14,s-\frac12}(4\pi n_\infty y)}{\Gamma(s+\frac14)} \sum_{\ell=1}^{p-1}\sum_{p|c>0}\dfrac{\mathbf{S}_{\infty\infty}^{(\ell)}(m,n,c,\mu_p)}c I_{2s-1}\(\dfrac{4\pi| m_\infty  n_\infty |^{\frac12}}c\),}&  n_\infty >0.
		\end{array}
		\right.
	\end{align*}
	For the right side of the expansion above, if we let $s=\frac 34$, by \eqref{Wfrack2 property} we get
	\begin{align*}
		&\sum_{\ell=1}^{p-1}\(1-\frac{\Gamma(\frac12,4\pi | m_\infty |y)}{\sqrt{\pi}}\)\frac{e^{2\pi  m_\infty z}}{\sin(\frac{\pi \ell}p)}\e_\ell
		+\sum_{n\in \Z} e^{2\pi i  n_\infty z }\cdot 2\pi e(-\tfrac18)\Big|\dfrac{ m_\infty }{ n_\infty }\Big|^{\frac14}\\
		&\ \cdot \left\{
		\begin{array}{ll}
			{\displaystyle
				\frac{\Gamma(\tfrac12,4\pi| n_\infty |y)}{\sqrt \pi} \sum_{\ell=1}^{p-1}\sum_{p|c>0}\dfrac{\mathbf{S}_{\infty\infty}^{(\ell)}(m,n,c,\mu_p)}c J_{\frac 12}\(\frac{4\pi| m_\infty  n_\infty |^{\frac12}}c\),}&  n_\infty <0;\\
			{\displaystyle
				\sum_{\ell=1}^{p-1}\sum_{p|c>0}\dfrac{\mathbf{S}_{\infty\infty}^{(\ell)}(m,n,c,\mu_p)}c I_{\frac 12}\(\dfrac{4\pi| m_\infty  n_\infty |^{\frac12}}c\),}&  n_\infty >0.
		\end{array}
		\right.
	\end{align*}
	By Proposition~\ref{convergence in general, main contribution}, the above expression is convergent. Therefore, by analytic continuation, the series $\mathbf{P}_{\infty}(z;p,s,\tfrac 12,m,\mu_p)$ is convergent at $s=\frac 34$ and has the Fourier expansion as above. 
	
	The last expression \eqref{S infty infty def, in scalar val} is easily deduced by combining \eqref{S infty infty ell def}, Definition~\ref{mu matrix define}, and \eqref{etaMultiplier}. 
\end{proof}

\subsubsection{Fourier expansion of \texorpdfstring{$\mathbf{P}_0$ at $\infty$}{Poincar\'e series (defined at the cusp 0) at infinity}}

Here we compute the Fourier expansions of\\ $\mathbf{P}_0(z;p,s,\tfrac12,r,\mu_p)$ at $s=\frac 34$. Also note that the convergence of the Fourier expansion in Proposition~\ref{Fourier exp of Maass Poincare at zero} is guaranteed by Proposition~\ref{convergence in general, main contribution} when $s=\frac 34$. Hence we have the convergence of $\mathbf{P}_0(z;p,s,\tfrac12,r,\mu_p)$ at $s=\frac 34$ by analytic continuation. Recall our notations $\condr$ and $\condar$ in \eqref{(a,r) def} and $\alpha_{+0}^{(\ell)}$ in \eqref{alpha 0 l def}. Since we do not consider the weight $-\frac 12$ case of $\overline{\mu_p}$ here, we write $X_{r,0}^{(\ell)}=X_{r,+0}^{(\ell)}$ for simplicity.

\begin{proposition}\label{Fourier exp of Maass Poincare at zero}
	For each integer $r\geq 0$, the series $\mathbf{P}_0(z;p,s,\tfrac12,r,\mu_p)$ is convergent at $s=\frac 34$ and we have
	\begin{align*}
		\mathbf{P}_0(z;p,\tfrac 34,\tfrac12,r,\mu_p)=\sum_{n_{\infty}>0} \mathbf{B}_0(n) q^{n_{\infty}}+\sum_{ n_{\infty}<  0} \mathbf{B}'_0(n)\frac{\Gamma\(\tfrac 12,4\pi |n_{\infty}| y\)}{\sqrt \pi}q^{n_{\infty}},
	\end{align*}
	where
	\begin{align}
		\left.\begin{array}{r}
			\mathbf{B}_0(n)\\
			\mathbf{B}'_0(n)
		\end{array}
		\right\}
		=2\pi\sum_{\ell=1}^{p-1}\sum_{\substack{a>0:\,p\nmid a,\\ [a\ell]\in \condr}}	\left|\frac{X_{r,0}^{([a\ell])}}{p n_\infty }\right|^{\frac14}\frac{\mathbf{S}_{0\infty}^{(\ell)}(X_{r}^{([a\ell])},n,a,\mu_p;r)}{a} 
		\left\{\begin{array}{l}
			I_{\frac12}\(\dfrac{4\pi}{a}\left|\dfrac{X_{r,0}^{([a\ell])}n_\infty }{p}\right|^{\frac12}\)\\
			J_{\frac12}\(\dfrac{4\pi}{a} \left|\dfrac{X_{r,0}^{([a\ell])}n_\infty}{p}  \right|^{\frac12}\). 
		\end{array}
		\right.
	\end{align}
	Here $\mathbf{S}_{0\infty}^{(\ell)}(X_r^{([a\ell])},n,a,\mu_p;r)=S_{0\infty}^{(\ell)}(X_r^{([a\ell])},n,a,\mu_p;r)\e_\ell$ is defined in \eqref{S 0 infty ell def, using X r}. If $[a\ell]\in \condr$, we have
	\begin{equation}\label{S 0 infty def, in scalar val}
		S_{0\infty}^{(\ell)}(X_r^{([a\ell])},n,a,\mu_p;r)=e(-\tfrac 18)\!\!\!\!\!\!\sum_{\substack{b:\ b\Mod a^*\\p|c,\,0<c<pa\\\text{s.t. }ad-bc=1}}\!\!\!\overline{\mu(c,d,[a\ell],p)e^{\pi i s(d,c)}}e\(\frac{m_{r,0}^{([a\ell])}\cdot\frac cp- n_\infty b}{-a}+\frac{a+d}{24c}\);
	\end{equation}
	if $[a\ell]\notin \condr$, we have $S_{0\infty}^{(\ell)}(X_r^{([a\ell])},n,a,\mu_p;r)=0$. 
\end{proposition}

\begin{remark}
	In the Fourier expansion, when $\ell$ is fixed, for the summation on $a$ we only select $a$ such that $p\nmid a$ and $[a\ell]\in \condr$. It is also important to note that the denominator in the last exponential term in \eqref{S 0 infty def, in scalar val} is $-a$, which is negative. 
\end{remark}

\begin{proof}
Recall the scaling matrices $\sigma_0=\begin{psmallmatrix}
	0&-1/\sqrt p\\\sqrt p&0
\end{psmallmatrix}$ and $\sigma_{\infty}=I$. We have the following double coset decomposition by \cite[(2.32)]{iwaniecTopClassicalMF}: 
\begin{equation}\label{double coset decomposition}
	\sigma_0^{-1}\Gamma_0(p)\sigma_\infty=\sigma_0^{-1}\Gamma_0(p)=\bigcup_{\substack{a>0\\p\nmid a}}\bigcup_{b\Mod a^*}\Gamma_{\infty}\begin{pmatrix}
		\frac{c}{\sqrt p}&\frac{d}{\sqrt p}\\
		-a\sqrt p&-b\sqrt p
	\end{pmatrix}\Gamma_\infty. 
\end{equation}
Then $\gamma_1\in \Gamma_0\setminus \Gamma_0(p)\Leftrightarrow \gamma_2=\sigma_0^{-1}\gamma_1\in \Gamma_\infty\setminus \sigma_0^{-1}\Gamma_0(p)$ and all choices of $\gamma_2$ can be described as 
\[\gamma_2\in\{\sigma_0^{-1}\begin{psmallmatrix}
	a&b\\c&d
\end{psmallmatrix}\begin{psmallmatrix}
	1&t\\0&1
\end{psmallmatrix}:\, a>0,\ p\nmid a,\ b\Mod a^*,\  t\in \Z\}.\]
	One can check that for $c\geq 0$ and $a>0$, we have
	\begin{align*}
		w_{\frac12}(\sigma_0^{-1},\sigma_0\gamma)(-a\sqrt p z-b\sqrt p)^{\frac12}=\(\frac{-a\sqrt p\,z-b\sqrt p}{cz+d}\)^{\frac12}(cz+d)^{\frac12}=-ip^{\frac 14}(az+b)^{\frac12}. 
	\end{align*}
	In the double coset decomposition, we can take the representative $\begin{psmallmatrix}
		c/\sqrt p & *\\-a\sqrt p&*
	\end{psmallmatrix}$
	with $a>0$ and $c\geq 0$ because $\begin{psmallmatrix}
		1 & \beta \\0&1
	\end{psmallmatrix}\begin{psmallmatrix}
		c/\sqrt p & *\\-a\sqrt p&*
	\end{psmallmatrix}=\begin{psmallmatrix}
		(c-\beta ap)/\sqrt p & *\\-a\sqrt p&*
	\end{psmallmatrix}$ for any $\beta \in \Z$. 
	Then from \eqref{Maass Poincare Series at zero}, for $\re s>1$ we have
	\begin{align*}
		\mathbf{P}_0&(z;p,s,\tfrac12,r,\mu_p)\\
		&=\frac{2 e(-\frac 18)p^{\frac 14}}{\sqrt{\pi}}\sum_{\ell\in \condr}\!\!\!\sum_{\substack{\gamma=\begin{psmallmatrix}
					\frac c{\sqrt p}&\frac d{\sqrt p}\\
					-a\sqrt p&-b\sqrt p
				\end{psmallmatrix}\\\gamma \in \Gamma_\infty\setminus\sigma_0^{-1}\Gamma_0(p)}}\!\!\!\!\mu_p(\sigma_0\gamma)^{-1}\overline{w_{\frac12}(\sigma_0^{-1},\sigma_0\gamma)}\frac{\varphi_{s,\frac12}(X_{r,0}^{(\ell)}\gamma z)}{(-a\sqrt p z-b\sqrt p)^{\frac12}}\e_\ell\\
		&=\frac{2e(\frac 18)}{\sqrt{\pi}} \sum_{\ell\in \condr}\sum_{\substack{a>0\\p\nmid a}}\sum_{b(a)^*}\sum_{t\in \Z}\mu_p\(\begin{psmallmatrix}
			a&b+ta\\c&d+tc
		\end{psmallmatrix}\)^{-1}\frac{\varphi_{s,\frac12}\(X_{r,0}^{(\ell)}\sigma_0^{-1}\begin{psmallmatrix}
			a&b+ta\\c&d+tc
		\end{psmallmatrix}z\)}{(az+b+ta)^{\frac12}}\e_\ell.  
	\end{align*}
	Here and below we use $\sum_{b(a)^*}$ to abbreviate the following summation condition: $c$ and $d$ are determined by $p|c$, $0<c<pa$, and $ad-bc=1$. 
	
	Observe that $\sigma_0^{-1}\begin{psmallmatrix}
		a&b+ta\\c&d+tc
	\end{psmallmatrix}z=\frac{cz+d+tc}{-paz-pb-pta}=-\frac{c}{pa}-\frac 1{pa(az+b+ta)}$, $\mu(c,d+tc,\ell,p)=\mu(c,d,\ell,p)$ for all $\ell$ and $t$, and $\nu_\eta\(\begin{psmallmatrix}
		a&b+ta\\c&d+tc
	\end{psmallmatrix}\)=\nu_\eta\(\begin{psmallmatrix}
		a&b\\c&d
	\end{psmallmatrix}\)e(t\alpha_{\infty})$ by \eqref{MultiplierSystemBasicProprety}. The contribution from a single $a$ for $p\nmid a$ is then
	\begin{align*}
		&\frac{2e(\frac 18)}{\sqrt{\pi a}}\sum_{\ell\in\condr} \sum_{b(a)^*} \overline{\mu(c,d,\ell,p)}\nu_\eta(\begin{psmallmatrix}
			a&b\\c&d
		\end{psmallmatrix})e\(\frac{-X_{r,0}^{(\ell)}c}{pa}\)\sum_{t\in \Z}  e(t\alpha_{\infty})\\
		& \qquad \cdot \(z+\tfrac ba+t\)^{-\frac12} 
		\mathcal{M}_{s}\(\frac{4\pi X_{r,0}^{(\ell)}y}{pa^2|z+\frac ba+t|^2}\)e\(\frac{-X_{r,0}^{(\ell)}} {pa^2} \re\(\frac1{z+\frac ba+t}\) \)\e_{[d\ell]}\\
		&=\frac{2e(\frac 18)}{\sqrt{\pi a}}\sum_{\ell\in\condar} \sum_{b(a)^*} \overline{\mu(c,d,[a\ell],p)}\nu_\eta(\begin{psmallmatrix}
			a&b\\c&d
		\end{psmallmatrix})e\(\frac{-X_{r,0}^{([a\ell])}c}{pa}\)\sum_{t\in \Z}  e(t\alpha_{\infty})\\
		& \qquad \cdot \(z+\tfrac ba+t\)^{-\frac12} 
		\mathcal{M}_{s}\(\frac{4\pi X_{r,0}^{([a\ell])}y}{pa^2|z+\frac ba+t|^2}\)e\(\frac{-X_{r,0}^{([a\ell])}} {pa^2} \re\(\frac1{z+\frac ba+t}\) \)\e_{\ell}. 
	\end{align*}
	Here we have changed $[d\ell]$ to $\ell$, hence $\ell$ to $[a\ell]$ and $\ell\in \condr$ to $\ell\in \condar$. 
	
	As in the case of $\mathbf{P}_\infty$ in Proposition~\ref{Maass Poincare Series at infty}, we let
	\[f(z)\defeq \sum_{t\in \Z} \frac{e(t\alpha_{\infty})}{(z+t)^{k} } \;\mathcal{M}_{s}\(\frac{4\pi X_{r,0}^{([a\ell])}y}{pa^2|z+t|^2}\)e\(\frac{-X_{r,0}^{([a\ell])} } {pa^2} \re\(\frac1{z+t}\) \).\]
	Then $f(z)e(\alpha_{\infty} x)$ has period $1$ and $f$ has Fourier expansion
	\begin{equation}\label{fzba}
		f(z)=\sum_{n\in \Z}a_{y}(n)e(n_{\infty} x)\quad \text{and}\quad f\(z+\frac ba\)=e\(\frac {n_{\infty} b}a\)f(z).  
	\end{equation}
	Here by \cite[Proof of Theorem~1.9]{BruinierBookBorcherds}, we have	
	\begin{align*}
		a_{y}(n)&=\frac{e(-\frac18)\Gamma(2s)}{\big|4\pi X_{r,0}^{([a\ell])}y\big|^{\frac 14}p^{\frac14}\sqrt a}\\
		&\cdot\left\{ \begin{array}{ll}
			{\displaystyle 
				\frac{2\pi}{\Gamma(s-\frac 14)}\left|\frac{X_{r,0}^{([a\ell])}}{n_\infty}\right|^{\frac12}\, W_{-\frac 14,s-\frac 12}(4\pi |n_\infty| y) J_{2s-1}\(\frac{4\pi}{a}\left|\frac{X_{r,0}^{([a\ell])}}{p}\cdot n_\infty\right|^{\frac12}\),}&  n_\infty <0;\\
			{\displaystyle
				\frac{2\pi}{\Gamma(s+\frac 14)} \left|\frac{X_{r,0}^{([a\ell])}}{n_\infty}\right|^{\frac12}\, W_{\frac 14,s-\frac 12}(4\pi n_\infty y) I_{2s-1}\(\frac{4\pi}{a}\left|\frac{X_{r,0}^{([a\ell])}}{p}\cdot n_\infty\right|^{\frac12}\),}&  n_\infty >0.\\
		\end{array}
		\right. 
	\end{align*}
	Thus, the Fourier expansion of $\mathbf{P}_0(z;p,s,\tfrac12,r,\mu_p)$ at the cusp $\infty$ for $\re s>1$ is: 
	\begin{align*}
		&\mathbf{P}_0(z;p,s,\tfrac12,\mathbf X_r,\mu_p)=\sum_{\ell=1}^{p-1}\sum_{n\in \Z} e^{2\pi i  n_\infty  z}\cdot \frac{2\Gamma(2s)}{\sqrt \pi|4\pi y|^{\frac 14}p^{\frac 14}|n_\infty|^{\frac 12}}\\
		\cdot&\left\{\begin{array}{ll}
			{\displaystyle 
					\!\!\!\!\sum_{\substack{a>0:\,p\nmid a,\\ [a\ell]\in \condr }}\frac{2\pi W_{-\frac 14,s-\frac 12}(4\pi |n_\infty |y)}{\Gamma(s-\frac 14)\left|X_{r,0}^{([a\ell])}\right|^{-\frac14}	} \frac{\mathbf{S}_{0\infty}^{(\ell)}(X_r^{([a\ell])},n,a,\mu_p;r)}a  J_{\frac12}\(\frac{4\pi}{a}\left|\frac{X_{r,0}^{([a\ell])}}{p}\cdot n_\infty \right|^{\frac12}\),}&n_{\infty}<0;\\
			{\displaystyle 
					\!\!\!\!\sum_{\substack{a>0:\,p\nmid a,\\ [a\ell]\in \condr }}\frac{2\pi W_{\frac 14,s-\frac 12}(4\pi n_\infty y)}{\Gamma(s+\frac 14)\left|X_{r,0}^{([a\ell])}\right|^{-\frac14}	}
				\frac{\mathbf{S}_{0\infty}^{(\ell)}(X_r^{([a\ell])},n,a,\mu_p;r)}{a} I_{\frac12}\(\frac{4\pi}{a}\left|\frac{X_{r,0}^{([a\ell])}}{p}\cdot n_\infty \right|^{\frac12}\),}&n_\infty>0. 
		\end{array}
		\right.
	\end{align*}
	
	For the right side of the expansion above, if we let $s=\frac 34$, by \eqref{Wfrack2 property} we get
	\begin{align*}
		&\sum_{\ell=1}^{p-1}\sum_{n\in \Z} 2\pi e^{2\pi i  n_\infty  z}\\
		\cdot&\left\{\begin{array}{ll}
			{\displaystyle 
				\!\!\!\!\sum_{\substack{a>0:\,p\nmid a,\\ [a\ell]\in \condr }}\!\!\!\!\frac{\Gamma(\frac 12,4\pi|n_\infty|y)}{\sqrt \pi} \left|\frac{X_{r,0}^{([a\ell])}}{pn_{\infty}}\right|^{\frac14}	 \frac{\mathbf{S}_{0\infty}^{(\ell)}(X_r^{([a\ell])},n,a,\mu_p;r)}a  J_{\frac12}\(\frac{4\pi}{a}\left|\frac{X_{r,0}^{([a\ell])}}{p}\cdot n_\infty \right|^{\frac12}\),}&n_{\infty}<0;\\
			{\displaystyle 
				\!\!\!\!\sum_{\substack{a>0:\,p\nmid a,\\ [a\ell]\in \condr }}\left|\frac{X_{r,0}^{([a\ell])}}{pn_{\infty}}\right|^{\frac14}	
				\frac{\mathbf{S}_{0\infty}^{(\ell)}(X_r^{([a\ell])},n,a,\mu_p;r)}{a} I_{\frac12}\(\frac{4\pi}{a}\left|\frac{X_{r,0}^{([a\ell])}}{p}\cdot n_\infty \right|^{\frac12}\),}&n_\infty>0. 
		\end{array}
		\right.
	\end{align*}
	By Proposition~\ref{convergence in general, main contribution}, where we take $\mathbf{m}=\mathbf{X}_r\leq 0$, the above expression is convergent. Therefore, by analytic continuation, the series $\mathbf{P}_0(z;p,s,\frac 12,r,\mu_p)$ is convergent at $s=\frac 34$ and has the Fourier expansion as above. 
	The expression \eqref{S 0 infty def, in scalar val} is deduced by combining \eqref{S 0 infty def, using X r}, \eqref{mu 0 infty def}, Definition~\ref{mu matrix define}, and \eqref{etaMultiplier}. 
\end{proof}

We combine the properties of the Maass-Poincar\'e series in the following proposition. 
\begin{proposition}\label{Maass Poincare series are harmonic Maass form proposition}
	Let $\mathbf{P}(z)$ denote either 
	\[\mathbf{P}_{\infty}(z)\defeq \mathbf{P}_\infty(z;p,\tfrac 34,\tfrac12,m,\mu_p)\quad \text{or}\quad \mathbf{P}_0(z)\defeq \mathbf{P}_0(z;p,\tfrac 34,\tfrac12,r,\mu_p).\]
	Then,
	\begin{enumerate}
		\item[(1)] For all $\gamma\in \Gamma_0(p)$, $\mathbf{P}(\gamma z)=\mu_p(\gamma )(cz+d)^{\frac 12}\mathbf{P}(z)$. 
		\item[(2)] For $1\leq \ell \leq p-1$, the $\ell$-th entry $P^{(\ell)}(z)$ of $\mathbf{P}(z)$ is a harmonic Maass form in $H_{\frac 12}(\Gamma_0(p^2)\cap \Gamma_1(p),\overline{\nu_\eta})$. 
		\item[(3)] For $1\leq \ell \leq p-1$, $P^{(\ell)}(24z)$ is a harmonic Maass form in $H_{\frac 12}(\Gamma_1(576p^2),\nu_\theta)$. 
		\item[(4)] The principal part of $\mathbf{P}_\infty(z;p,\tfrac 34,\tfrac12,m,\mu_p)$ at the cusp $\infty$ of $\Gamma_0(p)$ is \[\sum_{\ell=1}^{p-1} q^{m_\infty} \csc(\tfrac{\pi \ell}p)\e_\ell,\]
		and at the cusp $0$ of $\Gamma_0(p)$ is $\mathbf{0}$.
		\item[(5)] \label{Fourier exp of Maass Poincare at zero, principal part} For every integer $r\geq 0$, the principal part of $\mathbf{P}_0^{(\ell)}(z;p,\frac 34,\frac12,r,\mu_p)$ at the cusp $\infty$ of $\Gamma_0(p)$ is $\mathbf{0}$, and at the cusp $0$ of $\Gamma_0(p)$ is
		\[ e(-\tfrac 18)p^{\frac 14} \sum_{\ell\in \condr}
		q^{X_{r,0}^{(\ell)}}\e_\ell. 
		\]
	\end{enumerate}
\end{proposition}
\begin{proof}
	First we prove (1) and (2). We have discussed the transformation laws of $\mathbf{P}_\infty(z;p,s,\tfrac12,m,\mu_p)$ and $\mathbf{P}_0(z;p,s,\tfrac12,r,\mu_p)$ directly after their definitions. Since we have proved their convergence at $s=\frac 34$, by analytic continuation, the transformation laws are kept. When we focus on each entry $P^{(\ell)}(z)$ and $G^{(\ell)}(z)\defeq P^{(\ell)}(24z)$, the transformation laws
	\begin{align*}
		P^{(\ell)}(\gamma z)=\overline{\nu_\eta}(\gamma)(cz+d)^{\frac 12}P^{(\ell)}( z),& \quad \gamma\in \Gamma_0(p^2)\cap\Gamma_1(p),\\
		G^{(\ell)}(\gamma z)=\nu_\theta(\gamma)(cz+d)^{\frac 12}G^{(\ell)}( z),& \quad  \gamma\in \Gamma_1(576p^2)
	\end{align*}
	follow from Lemma~\ref{Transformation laws by z to 24z}.

	By \eqref{varphik/2 vanished by harmonic Maass Laplacian} we have proved (3). 
	Recall the definition for the principal parts before \eqref{vec val Principal part}. 
	For (4), the principal part of $\mathbf{P}_\infty$ at the cusp $\infty$ can be read from Proposition~\ref{Fourier exp of Maass Poincare at infty}. Note that 
	\[\begin{pmatrix}
		a&b\\c&d
	\end{pmatrix}\sigma_0=\begin{pmatrix}
		b\sqrt p&-a/\sqrt p\\d\sqrt p & -c/\sqrt p
	\end{pmatrix}\]
	and $d\neq 0$ for $\gamma=\begin{psmallmatrix}
		a&b\\c&d
	\end{psmallmatrix}\in \Gamma_0(p)$. Following from a similar process as in Proposition~\ref{Fourier exp of Maass Poincare at zero}, we can conclude that the principal part of $\mathbf{P}_\infty$ at the cusp $0$ is $\mathbf{0}$. 
	
	For (5), the principal part of $\mathbf{P}_0$ at the cusp $\infty$ is just $\mathbf 0$ from Proposition~\ref{Fourier exp of Maass Poincare at zero}. To compute its principal part at $0$, recall \eqref{Maass Poincare Series at zero} for the definition and \eqref{Mfrack2 property}. The Fourier expansion of $\mathbf{P}_0$ at the cusp $0$ is given by
	\[(\sqrt p\,z)^{-\frac12}\mathbf{P}_0(\sigma_0z;p,\tfrac 34,\tfrac12,r,\mu_p). \]
	Then the contribution from $c=0$ equals
	\begin{align*}
		\frac{2 e(-\frac 18)p^{\frac 14}}{\sqrt{\pi}}&(\sqrt p\,z)^{-\frac12}\sum_{\ell\in\condr}\frac{\varphi_{\frac 34,\frac12}(X_{r,0}^{(\ell)}z)}{(\sqrt p\,z)^{-\frac12}}\e_\ell\\
		&=e(-\tfrac 18)p^{\frac 14}\sum_{\ell\in\condr}\(1-\frac{\Gamma(\frac12,4\pi |X_{r,0}^{(\ell)}|y)}{\sqrt{\pi}}\)q^{X_{r,0}^{(\ell)}}\e_\ell, 
	\end{align*}
	and (5) follows.  
\end{proof}

\subsection{Proof of Theorem~\ref{main theorem}}
\label{Subsection: Proof of main theorem}

Fix a prime $p\geq 5$. 
To match the principal part of $\mathbf{G}_1(z;p)$ at the cusp $\infty$ (see \eqref{Principal part of G1 at infty}), we take $\mathbf{P}_\infty(z;p,\tfrac 34,\tfrac12,0,\mu_p)$. 

For the cusp $0$, we recall the definition of $\mathbf{X}_r$ in \eqref{X r define} and have
\begin{equation}
	X_r^{(\ell)}\defeq \left\{\begin{array}{ll}
		\ceil{-\tfrac {3\ell^2} {2p}+(\tfrac 12+r)\ell -\tfrac p{24}}, &\quad \text{when }0<\tfrac \ell p<x_r,\vspace{5px}\\
		\ceil{-\frac {3p}2(1-\tfrac \ell p)^2+(\tfrac 12+r)p(1-\tfrac \ell p) -\tfrac p{24}}, &\quad \text{when }1-x_r<\tfrac \ell p<1, \vspace{5px}\\
		0,&\quad \text{otherwise and will never be used}, 
	\end{array}
	\right.
\end{equation}
where $\ceil{x}$ is the smallest integer $\geq x$. Moreover, recalling $x_r$ and $\alpha_{0}^{(\ell)}$ (denoted as $\alpha_{+0}^{(\ell)}$ in \eqref{alpha 0 l def} and \eqref{alpha 0 l def, expressed in delta}), we see that 
\[X_{r,0}^{(\ell)}=-p\delta_{\ell,p,1,r},\quad X_{r}^{([a\ell])}=\ceil{-p\delta_{\ell,p,a,r}},\quad  X_{r,0}^{([a\ell])}=-p\delta_{\ell,p,a,r}, \]
and $X_{r,0}^{(\ell)}$ match the order of the $r$-th principal part of $\mathcal{G}_2(\frac \ell p;pz)$ in \eqref{Principal part of G1 at 0, equation}. 
Combining Proposition~\ref{Maass Poincare series are harmonic Maass form proposition}, \eqref{Principal part of G1 at infty}, \eqref{Principal part of G1 at 0 written in G2}, and \eqref{Principal part of G1 at 0, equation}, we conclude the following proposition. 
\begin{proposition}
	With the choice of $\mathbf{X}_r$ in \eqref{X r define}, the principal parts of 
	\[\mathbf{G}_1(z;p)-\mathbf{P}_\infty(z;p,\tfrac 34,\tfrac12,0,\mu_p)-2\sum_{\substack{r\geq 0\\ x_r^{-1}<p}}\mathbf{P}_0(z;p,\tfrac 34,\tfrac12,r,\mu_p)\]
	are zero for both cusps $\infty$ and $0$ of $\Gamma_0(p)$. 
\end{proposition}

Now we start to prove Theorem~\ref{main theorem}. 

\begin{lemma}
	For $\mathbf{X}_r$ defined in \eqref{X r define}, the function
	\[\mathbf{G}(z)\defeq \mathbf{G}_1(z;p)-\mathbf{P}_\infty(z;p,\tfrac 34,\tfrac12,0,\mu_p)-2\sum_{\substack{r\geq 0\\ x_r^{-1}<p}}\mathbf{P}_0(z;p,\tfrac 34,\tfrac12,r,\mu_p)\]
	is a holomorphic modular form of weight $\frac12$ on $(\Gamma_0(p),\mu_p)$, i.e. $\mathbf{G}(z)\in M_{\frac12}(\Gamma_0(p),\mu_p)$. 
\end{lemma}
\begin{proof}
	By Lemma~\ref{Transformation laws by z to 24z}, \eqref{G1}, and Proposition~\ref{Maass Poincare series are harmonic Maass form proposition}, $G^{(\ell)}(z)$ is a harmonic Maass form in $H_{\frac12}(\Gamma_0(p^2)\cap \Gamma_1(p), \overline{\nu_\eta})$ whose Fourier exponents are supported on $n_{\infty}=n-\frac1{24}$ for $n\in \Z$. 
	
	Since the principal part of $\mathbf{G}(z)$ is zero for both cusps $\infty$ and $0$ of $\Gamma_0(p)$, the principal part of $G^{(\ell)}(z)$ for every cusp of $\Gamma_0(p^2)\cap \Gamma_1(p)$ is zero. By Proposition~\ref{Maass Poincare series are harmonic Maass form proposition}, we know that 
	\[G^{(\ell)}(24z)\in H_{\frac12}(\Gamma_1(576p^2), {\nu_\theta})\] has Fourier exponents supported on $24n-1$. We also have that the principal part of $G^{(\ell)}(24z)$ for every cusp of $\Gamma_1(p^2)$ is zero. 
By Lemma~\ref{xiOperatorMapToCuspForms}, Lemma~\ref{xiOperatorMapToCuspForms, 2} and the Remark after Lemma~\ref{xiOperatorMapToCuspForms, 2}, 
the function $G^{(\ell)}(24z)$ is a holomorphic modular form in $M_{\frac12}(\Gamma_1(576p^2), \nu_\theta)$. 
	
	Since $\mathbf{G}(z)$ follows the modular transformation law on $(\Gamma_0(p),\mu_p)$ and each entry $G^{(\ell)}(z)$ is holomorphic, we get the desired result. 
\end{proof}

By Lemma~\ref{Theta function space only eta pz}, since $G^{(\ell)}(24z)$ has Fourier coefficients only supported on $24n-1$ for $n\geq 1$,  combining the above lemma with Lemma~\ref{mu p space is empty zero} we have
\begin{corollary}\label{Gz is 0}
	$\mathbf{G}(z)=\mathbf{0}$. 
\end{corollary}

\begin{proof}[Proof of Theorem~\ref{main theorem}]
	The theorem follows directly by combining Corollary~\ref{Gz is 0}, Proposition~\ref{Fourier exp of Maass Poincare at infty} and Proposition~\ref{Fourier exp of Maass Poincare at zero}. Note that the $n$-th Fourier coefficient of $\mathcal{G}_1(\frac \ell p;z)$ is $\csc(\frac{\pi\ell}p)A(\frac \ell p;n)$, hence we need to multiply the Fourier expansion of the Maass-Poincar\'e series by $\sin(\frac{\pi\ell}p)$ to get \eqref{main theorem formula}. 
\end{proof}

In the following two subsections, we will prove the claim that Bringmann's asymptotic formula \eqref{Bringmann formula}, when summing up to infinity, matches our exact formula \eqref{main theorem formula}. To be precise, we will show that the Fourier expansion of the $\ell$-th component of $\mathbf{P}_{\infty}$ matches the first sum in \eqref{Bringmann formula}, and the Fourier expansion of the $\ell$-th component of $\mathbf{P}_{0}$ matches the second sum on $r$ in \eqref{Bringmann formula}.

\subsubsection{Contribution from \texorpdfstring{$\mathbf{P}_\infty$}{Poincar\'e series defined at the cusp infinity}}
\label{Subsection: Bringmann formula match, cusp infty}

Recall that for a prime $p\geq 5$, a positive integer $c$ such that $p|c$, and $0<d<c$ such that $(d,c)=1$, the Dedekind sum $s(d,c)$ is defined in \eqref{etaMultiplier}. As $c$ is always clear in this subsection, we denote $d'_c$ as $d'$ for simplicity, i.e. it is defined by $dd'\equiv -1\Mod c$ if $c$ is odd and $dd'\equiv -1\Mod{2c}$ if $c$ is even. Also recall the notation that $a$ is given by $0<a<c$ with $ad\equiv 1\Mod c$ and $[a\ell]$ is defined by $0\leq [a\ell]<p$ such that $[a\ell]\equiv a\ell \Mod p$. 

In this subsection we prove \eqref{KL sums match with Bringmann, cusp infty} in the Remark of Theorem~\ref{main theorem}. We conjugate \eqref{Bringmann formula} since its left side is real and see that the first sum is 
\[\frac{2\pi e(-\frac 18)}{(24n-1)^{\frac14}}\sum_{p|c\leq \sqrt n}\frac{e(-\frac 18)\overline{B_{\ell ,p,c}(-n,0)}}cI_{\frac12}\(\frac{\pi\sqrt{24n-1}}{6c}\). \]
Then \eqref{B lpc sum def} gives
\[\overline{B_{\ell,p,c}(-n,0)}=\!\!\!\!\sum_{d\Mod c^*}\!\!\!(-1)^{\ell c+1}\frac{\sin(\frac {\pi\ell }p)}{\sin(\frac{\pi\ell d' }p)}\exp\(-\pi is(d,c)+\frac{3\pi i \ell^2 cd'}{p^2}\)e\(\frac{nd}c\). \]
On the other hand, recall \eqref{S infty infty def, in scalar val}: 
\[S_{\infty\infty}^{(\ell)}(m,n,c,\mu_p)=e(-\tfrac 18)\!\!\!\!\sum_{\substack{d\Mod c^*\\ad\equiv 1\Mod c}}\!\!\!\!\frac{\overline{\mu(c,d,[a\ell],p)}}{\sin(\frac{\pi[a\ell]}p)}e^{-\pi i s(d,c)}e\(\frac{m a+n d}c\). \]

To prove \eqref{KL sums match with Bringmann, cusp infty}, it suffices to show that for all $d\Mod c^*$, we have
\begin{equation}\label{matching - original equation for infty}
	\frac{(-1)^{\ell c+1}}{\sin(\frac{\pi \ell d'}p)}\exp\(\frac{3\pi i  cd'\ell^2}{p^2}\)=\frac{\overline{\mu(c,d,[a\ell ],p)}}{\sin(\frac{\pi [a\ell]}p)}. 
\end{equation}
We will show that both sides are equal to
\begin{equation}\label{matching - mid term for infty}
	\frac{(-1)^{\ell c}}{\sin(\frac{\pi a\ell}p)}\exp\(-\frac{3\pi i c a \ell^2}{p^2}\).
\end{equation} 

First we prove that the left side of \eqref{matching - original equation for infty} equals \eqref{matching - mid term for infty}. When $c$ is odd, we write $c=(2k+1)p$ for some integer $k$. Since $dd'\equiv -1\Mod c$, we can pick $d'=c-a$. Then
\begin{align*}
	\frac{(-1)^{\ell c+1}}{\sin(\frac{\pi \ell d'}p)}\exp\(\frac{3\pi i \ell^2 cd'}{p^2}\)&=\frac{	-(-1)^{\ell}}{\sin(\frac{\pi \ell c}p-\frac{\pi \ell a}p)}\exp\(\frac{3\pi i \ell^2 c^2}{p^2}-\frac{3\pi i \ell^2 ca}{p^2}\)\\
	&=\frac{	-(-1)^{\ell}}{-(-1)^{\ell(2k+1)}\sin(\frac{\pi \ell a}p)}(-1)^{\ell^2(2k+1)^2}\exp\(-\frac{3\pi i \ell^2 ca}{p^2}\)\\
	&=\frac{(-1)^{\ell}}{\sin(\frac{\pi \ell a}p)}\exp\(-\frac{3\pi i \ell^2 ca}{p^2}\),
\end{align*}
which equals \eqref{matching - mid term for infty}. When $c$ is even, we write $c=2kp$ for some positive integer $k$. We pick $0<a<2kp$ for $ad\equiv 1\Mod{2kp}$ and $0<d'<4kp$ for $d'd\equiv -1\Mod{4kp}$. Observe that \eqref{matching - mid term for infty} is the same if we change $a$ to $a\pm 2kp$, so we can pick $a=2c-d'$ here and a similar process shows that the left side of \eqref{matching - original equation for infty} equals \eqref{matching - mid term for infty} when $c$ is even. 

Next we prove that the right side of \eqref{matching - original equation for infty} equals \eqref{matching - mid term for infty}. Define the integer $t\geq 0$ by $[a\ell]=a\ell-tp$, $k$ by $c=kp$, and $b$ by $ad=1+bc$. By \eqref{mu cdlp def}, we have
\begin{align*}
	&\frac{\overline{\mu(c,d,[a\ell],p)}}{\sin(\frac{\pi[a\ell]}p)}\\
	&=\exp\(-\frac{3\pi i c d (a\ell -tp)^2}{p^2}\)(-1)^{\frac{c(a\ell-tp)}p}(-1)^{\floor{\frac{d(a\ell-tp)}p}}\Big/{\sin\(\frac{\pi a\ell}p-\pi t\)}\\
	&=\exp\(-\frac{3\pi i c a \ell^2}{p^2}-\frac{3\pi i c a b c \ell^2}{p^2}-3\pi i c d t^2 \)(-1)^{a\ell k-tc+\floor{\frac \ell p+b\ell k-td}+t}\Big/{\sin\(\frac{\pi a\ell}p\)}. 
\end{align*}
The above formula equals $\exp(-\frac{3\pi i c a \ell^2}{p^2})/\sin(\frac{\pi a\ell}p)$ times $(-1)$ to the power of
\begin{align*}
	&ab\ell^2 k^2 -cdt^2+a\ell k -tc+b\ell k-td+t\\
	&\equiv ab\ell k+cdt+a\ell k +tc+b\ell k+td+t\\
	&\equiv (a+1)(b+1)\ell k+\ell k+(c+1)(d+1)t\\
	&\equiv \ell k\equiv \ell c\Mod 2.
\end{align*}
The last step uses $(x+1)(y+1)\equiv 0\Mod 2$ whenever $(x,y)=1$.

\begin{remark}
	From the proof above, for $p|c$ and $0< [a\ell]=a\ell-tp<p$, we also have
	\begin{equation}\label{mu simplified}
		\overline{\mu(c,d,[a\ell],p)}=\exp\(-\frac{3\pi i c a \ell^2}{p^2}\)(-1)^{\ell c+t}. 
	\end{equation}
\end{remark}

\subsubsection{Contribution from \texorpdfstring{$\mathbf{P}_0$}{Poincar\'e series defined at the cusp 0}}
\label{Subsection: Bringmann formula match, cusp 0}

In this subsection we prove \eqref{KL sums match with Bringmann, cusp 0} in the Remark of Theorem~\ref{main theorem}. 
Recall the definition of $\delta_{\ell,p,a,r}$ in \eqref{delta l p a r def}, of $m_{\ell,p,a,r}$ in \eqref{m l p a r def}, of $\alpha_{0}^{(\ell)}$ in \eqref{alpha 0 l def} and of $\mathbf{X}_r$ in \eqref{X r define}. In Bringmann's asymptotic formula \eqref{Bringmann formula}, the second sum on $r$ (after conjugation) becomes
\[\frac{4\pi\sin(\frac{\pi \ell}p)}{(n-\frac1{24})^{\frac 14}}\sum_{r\geq 0}\sum_{\substack{a>0:\,p\nmid a,\\\delta_{\ell,p,a,r}>0}}\frac{\overline{D_{\ell,p,a,r}(-n,m_{\ell,p,a,r})}}{a\cdot \delta_{\ell,p,a,r}^{-\frac14}}I_{\frac12}\(\frac{4\pi}a \left|\delta_{\ell,p,a,r}\(n-\frac1{24}\)\right|^{\frac 12}\).\]
Recall \eqref{S 0 infty def, in scalar val}:
\[S_{0\infty}^{(\ell)}(X_{r}^{([a\ell])},n,a,\mu_p;r)=e(-\tfrac 18)\!\!\!\!\!\!\!\sum_{\substack{b:\ b\Mod a^*\\0<c<pa,\ p|c\\\text{s.t. }ad-bc=1}}\!\!\!\!\!\!\!\overline{\mu(c,d,[a\ell],p)}e^{-\pi i s(d,c)}e\(\frac{-X_{r,0}^{([a\ell])}\frac cp+ n_\infty b}a+\frac{a+d}{24c}\).\]

We denote $b'_a$ as $b'$ for simplicity, i.e. $b'$ is defined by $bb'\equiv -1\Mod a$ if $a$ is odd and by $bb'\equiv -1\Mod {2a}$ if $a$ is even. Moreover, we still denote positive integers $t$ by $a\ell-[a\ell]=tp$ and $k$ by $c=kp$.  

We have $c\equiv b'\Mod a$ for $\begin{psmallmatrix}
	a&b\\c&d
\end{psmallmatrix}\in \Gamma_0(p)$.
Hence we can rewrite $\overline{D_{\ell,p,a,r}}$ as:
\begin{align*}
	\overline{D_{\ell,p,a,r}(-n,m_{\ell,p,a,r})}&=(-1)^{a\ell+[a\ell]}\sum_{b\Mod a^*} \overline{\omega_{b,a}}\ \overline{e\(\frac{m_{\ell,p,a,r}b'-nb}{a}\)}\\
	&=(-1)^{a\ell-[a\ell]}\sum_{b\Mod a^*} e^{-\pi is(b,a)}e\(\frac{-m_{\ell,p,a,r}c+nb}{a}\).
\end{align*}
For $\gamma=\begin{psmallmatrix}
	a&b\\c&d
\end{psmallmatrix}\in \Gamma_0(p)$, with our choice  $c\geq 0$ and $a>0$, we need the relationship between $e(-\pi i s(b,a))$ and $e^{-\pi i s(d,c)}$. Denote $S=\begin{psmallmatrix}
	0&-1\\1&0
\end{psmallmatrix}$. Recall $w_{\frac 12}$ in Definition~\ref{multiplier system def}. We have
\[w_{\frac 12}(S,\gamma)=(cz+d)^{\frac 12}\(\frac{az+b}{cz+d}\)^{\frac 12}(az+b)^{-\frac 12}=1\]
because $cz+b$, $\frac{az+b}{cz+d}$ and $az+b$ are in $\HH$ for $z\in \HH$. Therefore, we have $\nu_\eta(S\gamma)=\nu_\eta(S)\nu_\eta(\gamma)$. With the help of $\nu_\eta(S)=e(-\frac 18)$ by \eqref{etaMultiplier}, we get
\[e^{-\pi i s(b,a)}=e(-\tfrac 18)e^{-\pi is(d,c)}e(\tfrac{a+d}{24c}+\tfrac {c-b}{24a}). \]
Then we continue: 
\begin{align*}
	\overline{D_{\ell,p,a,r}(-n,m_{\ell,p,a,r})}&=(-1)^{tp}e(-\tfrac 18)\sum_{b\Mod a^*} \overline{\omega_{d,c}}\ e\(\frac{-m_{\ell,p,a,r}c+nb}{a}+\frac{a+d}{24c}+\frac {c-b}{24a}\)\\
	&=(-1)^{t}e(-\tfrac 18)\sum_{b\Mod a^*} \overline{\omega_{d,c}}\ e\(\frac{a+d}{24c}\)e\(\frac{(\frac1{24}-m_{\ell,p,a,r})c+(n-\frac 1{24})b}{a}\).
\end{align*}
Comparing with the formula of $S_{0\infty}^{(\ell)}(X_{r}^{([a\ell])},n,c,\mu_p;r)$ where $X_{r,0}^{([a\ell])}=-p\delta_{\ell,p,a,r}$, we are left to prove that
\begin{equation}\label{Equation muDelta equal -1 mlpar}
	\overline{\mu(c,d,[a\ell],p)}\ e\(\frac{\delta_{\ell,p,a,r}c}a\)=(-1)^te\(\frac{(\frac1{24}-m_{\ell,p,a,r})c}a\). 
\end{equation}
By \eqref{mu simplified}, when $0<\frac{[a\ell]}p<\frac 16$, recalling $[a\ell]=a\ell-tp$ and $c=kp$, we have
\begin{align*}
	&\overline{\mu(c,d,[a\ell],p)}\ e\(\frac{(\delta_{\ell,p,a,r}-\frac1{24})c}a\)\\
	&=(-1)^{\ell c+t} e\(-\frac{3c a \ell^2}{2p^2}-\frac{c(1+2r)(a\ell-tp)}{2ap}+\frac{3c(a\ell-tp)^2}{2ap^2}\)\\
	&=(-1)^{\ell k+t}e\(-\frac{3c a \ell^2}{2p^2}-\frac{k\ell(1+2r)}{2}+\frac{ct(1+2r)}{2a} +\frac{3ca\ell^2}{2p^2}-\frac{3c\ell t}{p}+\frac{3ct^2}{2a}\)\\
	&=(-1)^{t} e\(\frac ca\(\frac{(1+2r)t}2+\frac32 t^2\)\).
\end{align*}
On the other hand, by \eqref{m l p a r def}, we have
\[-m_{\ell,p,a,r}=\frac1{2p^2}\Big(3(a\ell-[a\ell])^2+p(1+2r)(a\ell-[a\ell])\Big)=\frac 32t^2+\frac{(1+2r)t}2. \]
This gives \eqref{Equation muDelta equal -1 mlpar}. The proof when $\frac 56<\frac{[a\ell]}p<1$ is similar: we have
\begin{align*}
	&\overline{\mu(c,d,[a\ell],p)}\ e\(\frac{(\delta_{\ell,p,a,r}-\frac1{24})c}a\)\\
	&=(-1)^{\ell c+t} e\(-\frac{3c a \ell^2}{2p^2}-\frac{5c(a\ell-tp)}{2ap}+\frac{3c(a\ell-tp)^2}{2ap^2}+(1-r)\frac ca+\frac{cr(a\ell-tp)}{ap}\)\\
	&=(-1)^{\ell k+t}e\(-\frac{5k\ell}{2}+\frac{5ct}{2a} -3k\ell t+\frac{3ct^2}{2a}+(1-r)\frac{c}a+rk\ell -rt\frac{c}a\)\\
	&=(-1)^{t} e\(\frac ca\(\frac{(5-2r)t}2+\frac32 t^2+1-r\)\). 
\end{align*}
When $\frac 56<\frac{[a\ell]}p<1$, by \eqref{m l p a r def} we also have
\[-m_{\ell,p,a,r}=\frac 32t^2+\frac{(5-2r)t}2+1-r. \]
Now we still get \eqref{Equation muDelta equal -1 mlpar} and \eqref{KL sums match with Bringmann, cusp 0} follows.

\begin{remark}
	From the proof we can rewrite
	\begin{equation}\label{S 0 infty Simiplified for Appendix}
		S_{0\infty}^{(\ell)}\Big(\ceil{-p\delta_{\ell,p,a,r}},n,a,\mu;r\Big)=(-1)^{a\ell-[a\ell]}\sum_{b\Mod a^*}e^{-\pi i s(b,a)}e\(\frac{-m_{\ell,p,a,r}c+nb}a\)
	\end{equation}
	where $0<[a\ell]=a\ell- tp<p$ and 
	\[-m_{\ell,p,a,r}=\left\{\begin{array}{ll}
		\frac32 t^2+\frac{1+2r}2 t,&\text{ when } 0<\frac{[a\ell]}p<\frac 16,\\
		\frac 32 t^2+\frac{5-2r}2 t+1-r, &\text{ when }\frac 56<\frac{[a\ell]}p<1. 
	\end{array}
	\right.\] 
\end{remark}

\section{Sums of vector-valued Kloosterman sums}
\label{Section: sums of v-val KL sums}

Recall that we have cited Proposition~\ref{convergence in general, main contribution} in the proofs of Proposition~\ref{Fourier exp of Maass Poincare at infty} and Proposition~\ref{Fourier exp of Maass Poincare at zero}, in order to guarantee the convergence of certain sums of Kloosterman sums. To finish the proof of Theorem~\ref{main theorem}, we prove Proposition~\ref{convergence in general, main contribution} in the remaining part of this paper.

We first follow \cite{gs} to prove asymptotic formulae (Theorem~\ref{Sums of Kloosterman sums}) for sums of certain vector-valued Kloosterman sums. Let $p\geq 5$ be a prime number. Let $(k,\mu)$ be either $(\frac12,\mu_p)$ or $(-\frac12,\overline{\mu_p})$. For $n\in \Z$ and $\mathbf{m}\in \Z^{p-1}$, recall the notations $\alpha_{\pm\infty}$, $n_{\pm\infty}$, $\alpha_{\pm0}^{(\ell)}$ and $m_{\pm0}^{(\ell)}$ for $1\leq \ell\leq p-1$ introduced before Lemma~\ref{mu p is a vec val mult sys} and recall the Kloosterman sums defined in \eqref{S infty infty def} and \eqref{S 0 infty def}. 

Recall the $W$-Whittaker function $W_{\beta,\mu}$ defined as in \cite[(13.14.3)]{dlmf}. 
The following Mellin transform can be found in \cite[(2.3)]{gs} and will be used frequently in the next section. For $\alpha>0$ and $\re(s+\frac12\pm \mu)>0$, we have
\begin{equation}\label{integral e y W for inner product}
	\int_0^\infty e^{-2\pi \alpha y}y^s W_{\beta,\mu}(4\pi \alpha y)\frac{dy}y=(4\pi \alpha)^{-s}\frac{\Gamma(s+\frac12+\mu)\Gamma(s+\frac12-\mu)}{\Gamma(s-\beta+1)}. 
\end{equation}

By Proposition~\ref{LEigenform mu p r is finite dim}, for every spectral parameter $r_j$ of $\lambda_j=\frac 14+r_j^2$ in the discrete spectrum, we can pick an orthonormal (under the inner product \eqref{vec val Petersson inner prod}) basis of $\LEigenform_{\frac12}(\Gamma_0(p),\mu_p,r_j)$ denoted as $\mathrm{OB}(r_j)$. For every $\mathbf{V}(z;r_j)\in \mathrm{OB}(r_j)$, we have $	V^{(\ell)}(z;r_j)\in \LEigenform_{\frac12}(\Gamma_0(p^2)\cap \Gamma_1(p),\overline{\nu_\eta},r_j)$ which has the Fourier expansion
\begin{equation}\label{V Fourier expansion at the cusp infty}
	V^{(\ell)}(z;r_j)=c_\infty^{(\ell)}(0,y)+\sum_{\substack{n\in \Z\\ n_{+\infty}\neq 0}} \rho_{j,\infty}^{(\ell)}(n)W_{\frac k2 \sgn  n_{+\infty},ir_j}(4\pi |n_{+\infty}|y)e(n_{+\infty}x)
\end{equation}
as in \cite[(2.2)]{gs}. Since $\alpha_{+\infty}\neq 0$, we have $c_\infty^{(\ell)}(0,y)=0$. The Fourier expansion of $\mathbf{V}(z;r_j)$ at the cusp $0$ is given by
\begin{equation}\label{V Fourier expansion at the cusp 0}
	(V^{(\ell)}|_k\sigma_0)(z;r_j)=c_0^{(\ell)}(0,y)+\sum_{\substack{n\in \Z\\n_{+ 0}^{(\ell)}\neq 0}}\rho_{j,0}^{(\ell)}(n)W_{\frac 14\sgn n,\,ir_j}\(4\pi |n_{+ 0}^{(\ell)}|y\)e(n_{+ 0}^{(\ell)}x). 
\end{equation}
Here $c_{0}^{(\ell)}(0,y)=0$ because $\alpha_{+0}^{(\ell)}\neq 0$ for $1\leq \ell\leq p-1$. 

Specially, for the case of $r_0=\frac i4$, by the proof of Proposition~\ref{LEigenform mu p r is finite dim}, any $\mathbf{V}(z;r_0)\in \LEigenform_{\frac 12}(\Gamma_0(p),\mu_p,r_0)$ satisfies $\mathbf{V}^{(\ell)}(z;r_0)\in \LEigenform_{\frac 12}(\Gamma_0(p^2)\cap \Gamma_1(p),\overline{\nu_\eta},r_0)$. From \eqref{CuspFormR0}, there exists a one-to-one correspondence between $\mathbf{V}(z;r_0)\in \LEigenform_{\frac 12}(\Gamma_0(p),\mu_p,r_0)$ and $F\in M_{\frac 12}(\Gamma_0(p^2)\cap \Gamma_1(p),\overline{\nu_\eta})$ by 
\begin{equation}\label{CuspFormR0 Vrj vec val} 
	V^{(1)}(z;r_0)=y^{\frac 14}F(z). 
\end{equation}
By \eqref{Coeffi CuspFormR0}, for the Fourier expansion of $\mathbf{V}(z;r_0)$, since $\alpha_{+\infty}^{(\ell)}>0$ \eqref{alpha infty l def} and $\alpha_{+0}^{(\ell)}>0$ \eqref{alpha 0 l def}, we have
\begin{align}\label{Coeffi CuspFormR0 Vrj vec val}
	\begin{split}
		V^{(\ell)}(z;r_0)&=\sum_{n_{+\infty}>0} \rho_{0,\infty}^{(\ell)}(n)W_{\frac 14,ir_0}(4\pi n_{+\infty}y)e(n_{+\infty}x),\\
		(V^{(\ell)}|_k\sigma_0)(z;r_0)&=\sum_{n_{+ 0}^{(\ell)}> 0}\rho_{0,0}^{(\ell)}(n)W_{\frac 14,\,ir_0}\(4\pi n_{+ 0}^{(\ell)}y\)e(n_{+0}^{(\ell)}x), 
	\end{split}
\end{align}
i.e. $\rho_{0,\infty}^{(\ell)}(n)=\rho_{0,0}^{(\ell)}(n)=0$ if $n\leq 0$.

We prove the following theorem in this section and then Proposition~\ref{convergence in general, main contribution} follows in \S\ref{Subsection: Convergence of sums of KL sums}. 
\begin{theorem}\label{Sums of Kloosterman sums}
	Fix an integer $r\geq 0$, a prime $p\geq 5$, and let $1\leq L,\ell\leq p-1$. Let $m\in \Z$ with $m\leq 0$, $\mathbf{m}\in \Z^{p-1}$ with $\mathbf{m}\leq 0$, and $\displaystyle{M=\max_{\ell\in \condr} \{|m_{+0}^{(\ell)}|\}}$. Recall \eqref{alpha infty l def} for $m_{+\infty}$ and \eqref{alpha 0 l def} for $m_{+0}^{(\ell)}$. We have the following results: 
	\begin{align}\label{Kloosterman sums asymptotics infty infty}
		\sum_{c\leq x:\,p|c}\frac{S_{\infty\infty}^{(\ell)}(m,n,c,\mu_p)}{c}&=\sum_{\frac12<s_j\leq \frac 34}\tau_{j,\infty}^{(\ell)}(m,n)\frac{x^{2s_j-1}}{2s_j-1}+O_{p,\ep}\(|m_{+\infty}n_{+\infty}|^{3}x^{\frac13+\ep}\),\\
		\label{Kloosterman sums asymptotics 0 infty, every L}
		\sum_{\substack{a\leq x:\\p\nmid a,\,[a\ell]=L}}
		\frac{S_{0\infty}^{(\ell)}(m^{(L)},n,a,\mu_p;r)}{a\sqrt p}&=\sum_{\frac12<s_j\leq \frac 34}\tau_{j,0,(L)}^{(\ell)}(m^{(L)},n)\frac{x^{2s_j-1}}{2s_j-1}+O_{p,\ep}\(|m_{+0}^{(L)}n_{+\infty}|^{3}x^{\frac13+\ep}\),\\
		\label{Kloosterman sums asymptotics 0 infty}
		\sum_{\substack{a\leq x:\\p\nmid a,\,[a\ell]\in \condr}}\frac{S_{0\infty}^{(\ell)}(m^{([a\ell])},n,a,\mu_p;r)}{a\sqrt p}&=\sum_{\frac12<s_j\leq \frac 34}\tau_{j,0}^{(\ell)}(\mathbf{m},n)\frac{x^{2s_j-1}}{2s_j-1}+O_{p,\ep}\(|Mn_{+\infty}|^{3}x^{\frac13+\ep}\),
	\end{align}
	where 
	\[\tau_{j,\infty}^{(\ell)}(m,n)=e(\tfrac 18)\(\sum_{L=1}^{p-1}\frac{\overline{\rho_{j,\infty}^{(L)}(m)}}{\sin(\frac{\pi L}p)}\)\frac{\rho_{j,\infty}^{(\ell)}(n)}{\sin(\frac{\pi\ell}p)}\cdot\frac{\Gamma(s_j+\frac14 \sgn n_{+\infty})\Gamma(2s_j-1)}{\pi^{2s_j-1}|4m_{+\infty}n_{+\infty}|^{s_j-1}\Gamma(s_j-\frac14)},\]
	\[\tau_{j,0,(L)}^{(\ell)}(m^{(L)},n)=e(-\tfrac 18)\overline{\rho_{j,0}^{(L)}(m^{(L)})}\cdot\frac{\rho_{j,\infty}^{(\ell)}(n)}{\sin(\frac{\pi\ell}p)}\cdot\frac{\Gamma(s_j+\frac14 \sgn n_{+\infty})\Gamma(2s_j-1)}{\pi^{2s_j-1}|4m_{+0}^{(L)}n_{+\infty}|^{s_j-1}\Gamma(s_j-\frac14)},\]
	and $\displaystyle \tau_{j,0}^{(\ell)}(\mathbf{m},n)=\sum_{L\in \condr}\tau_{j,0,(L)}^{(\ell)}(m^{(L)},n)$. 
	Here all the sums on $s_j$ run over the exceptional eigenvalues $\lambda_j=s_j(1-s_j)\in [\frac 3{16},\frac14)$ of $\Delta_{\frac 12}$ on $\Lform_{\frac12}(\Gamma_0(p),\mu_p)$. The coefficients $\rho_{j,\infty}^{(\ell)}(n)$ and $\rho_{j,0}^{(\ell)}(m^{(\ell)})$ are the Fourier coefficients of an eigenform $\mathbf{V}(z;r_j)$ of $\Delta_{\frac12}$ in a orthonormal basis $\mathrm{OB}(r_j)$ of $\LEigenform_{\frac12}(\Gamma_0(p),\mu_p,r_j)$,  defined in \eqref{V Fourier expansion at the cusp infty} and \eqref{V Fourier expansion at the cusp 0}. The summation term corresponding to a single $s_j$ should be understood as the sum over all $\mathbf{V}(z;r_j)\in \mathrm{OB}(r_j)$. 
\end{theorem}

\begin{remark} We have the following remarks of the above theorem. 
	\begin{enumerate}
		\item It is possible to improve the exponents on $m_{+\infty}$, $m_{+0}^{(\ell)}$, $M$ and $n_{+\infty}$. In this paper, any polynomial growth is enough. One may refer to \cite{pribitkin,QihangFirstAsympt} for improvements on scalar-valued Kloosterman sums. 
		\item The notation $r$ always means the integer $r\geq 0$ which appears in the Kloosterman sum $S_{0\infty}^{(\ell)}(m^{(L)},n,a,\mu_p;r)$. The notation $r_j$, with subscript $j\geq 0$, is the spectral parameter of the eigenvalue $\lambda_j=\frac 14+r_j^2$ of $\Delta_{\frac 12}$ on $(\Gamma_0(p),\mu_p)$. 
	\end{enumerate}
	
\end{remark}

\begin{corollary}\label{Sums of Kloosterman sums, whole growth rate for convergence}
	With the same setting as Theorem~\ref{Sums of Kloosterman sums}, there exists a $\delta>0$ such that for all $1\leq L,\ell\leq p-1$, we have
	\begin{align*}
		\sum_{c\leq x:\, p|c}\frac{S_{\infty\infty}^{(\ell)}(m,n,c,\mu_p)}{c}&\ll_p |m_{+\infty}n_{+\infty}|^3 x^{\frac12-\delta}\\
		\sum_{\substack{a\leq x:\\p\nmid a,\, [a\ell]=L}} \frac{S_{0\infty}^{(\ell)}(m^{(L)},n,a,\mu_p;r)}{a\sqrt{p}}&\ll_{p} |m_{+0}^{(L)}n_{+\infty}|^3x^{\frac12-\delta},\\
		\sum_{\substack{a\leq x:\\p\nmid a,\,[a\ell]\in \condr}} \frac{S_{0\infty}^{(\ell)}(m^{([a\ell])},n,a,\mu_p;r)}{a\sqrt{p}}&\ll_{p} |Mn_{+\infty}|^3x^{\frac12-\delta}. 
	\end{align*} 
\end{corollary}

\begin{proof}[Proof of Corollary~\ref{Sums of Kloosterman sums, whole growth rate for convergence}]
	Granted Theorem~\ref{Sums of Kloosterman sums}, it suffices to determine the growth rate of $\rho_{j,0}^{(\ell)}(n)$ and $\rho_{j,0}^{(\ell)}(m^{(\ell)})$. By the discussion in Proposition~\ref{LEigenform mu p r is finite dim}, these coefficients are also the Fourier coefficients of an eigenform of $\Delta_{\frac 12}$ in a orthonormal basis of $\LEigenform_{\frac12}(\Gamma_0(p^2)\cap \Gamma_1(p),\overline{\nu_\eta},r_j)$. Then we can get the growth rate from the spaces of scalar valued Maass eigenforms of $\Delta_{\frac 12}$. 
	
	For $\frac3{16}=\lambda_0=s_0(1-s_0)$, i.e. $s_0=\frac 34$, we know that $\LEigenform_{\frac12}(\Gamma_0(p^2)\cap \Gamma_1(p),\overline{\nu_\eta},r_j)$ corresponds to holomorphic modular forms $M_{\frac12}(\Gamma_0(p^2)\cap \Gamma_1(p),\overline{\nu_\eta})$ as in \eqref{CuspFormR0 Vrj vec val}. Since we have $m_{+\infty}<0$ and $m_{+0}^{(\ell)}<0$, by \eqref{CuspFormR0 Vrj vec val}, we get $\rho_{0,\infty}^{(\ell)}(m)=0$ and $\rho_{0,0}^{(\ell)}(m^{(\ell)})=0$. Thus, the term $x^{2s_0-1}=x^{\frac12}$ for $s_0=\frac 34$ is not contained in each sum. 
	
	Since the exceptional eigenvalues of $\Delta_{\frac12}$ on $\Lform_{\frac12}(\Gamma_0(p^2)\cap \Gamma_1(p),\overline{\nu_\eta},r_j)$ are discrete, there exists $\delta>0$ such that $2s_j-1<\frac12-\delta$ for all $s_j\in (\frac12,\frac 34)$. For these $j$, both $\rho_{j,\infty}^{(\ell)}(n)$ and $\rho_{j,0}^{(\ell)}(m^{(\ell)})$ are $O_p(1)$. The corollary follows. 
\end{proof}

We generalize the method in \cite{gs} (to the vector-valued setting) to prove Theorem~\ref{Sums of Kloosterman sums}. Define the Kloosterman-Selberg zeta functions as 
\begin{equation}\label{KloostermanSelbergZetaFunctionVVal}
	Z_{\substack{\infty\infty\\m,n,+}}^{(\ell)}(s)\defeq \sum_{c\leq x:\,p|c}\frac{S_{\infty\infty}^{(\ell)}(m,n,c,\mu_p)}{c^{2s}},\quad 
	Z_{\substack{0\infty;r\\m^{(L)},n,+}}^{(\ell)}(s)\defeq \sum_{\substack{a\leq x:\,p\nmid a,\\ [a\ell]=L}}\frac{S_{0\infty}^{(\ell)}(m^{(L)},n,a,\mu_p;r)}{(a\sqrt{p})^{2s}}.
\end{equation}
and $\displaystyle Z_{\substack{0\infty;r\\\mathbf{m},n,+}}^{(\ell)}(s)\defeq \sum_{L\in \condr} Z_{\substack{0\infty;r\\m^{(L)},n,+}}^{(\ell)}(s)$. 

We address the proof of \eqref{Kloosterman sums asymptotics infty infty} in the next subsection. The proof of \eqref{Kloosterman sums asymptotics 0 infty} is in the subsequent subsection. 

\subsection{\texorpdfstring{The cusp $\infty$}{The cusp infinity}}
Recall the notation $(k,\mu)=(\frac 12,\mu_p)$ or $(-\frac 12,\overline{\mu_p})$. For $m>0$, we define the weight $k$ non-holomorphic vector-valued Poincar\'e series as
\begin{equation}
	\mathbf{U}(z;s,k,m,\mu)\defeq \sum_{\ell=1}^{p-1} \sum_{\gamma\in \Gamma_\infty\setminus \Gamma_0(p)}\mu(\gamma)^{-1} j(\gamma,z)^{-k}\frac{y^s}{|cz+d|^{2s}}\,\frac{e( m_{\pm\infty} \gamma z) \e_{\ell}}{\sin(\frac{\pi \ell}p)} 
\end{equation}
where $\gamma=\begin{psmallmatrix}
	a&b\\c&d
\end{psmallmatrix}$ and the $\pm$ in $m_{\pm\infty}$ is chosen depending on the sign of $k$. We also denote
\begin{equation}\label{Poincare series generated from l}
	\mathbf{U}_{(\ell)}(z;s,k,m,\mu)\defeq  \sum_{\gamma\in \Gamma_\infty\setminus \Gamma_0(p)}\mu(\gamma)^{-1} j(\gamma,z)^{-k}\frac{y^s}{|cz+d|^{2s}}\,\frac{e( m_{\pm\infty} \gamma z) }{\sin(\frac{\pi \ell}p)} \e_{\ell}
\end{equation} 
as the part of the Poincar\'e series generated by $\e_{\ell}$.
\begin{remark}
	Note the difference between $\mathbf{U}_{(\ell)}(z;s,k,m,\mu)$ and $\mathbf{U}^{(\ell)}(z;s,k,m,\mu)$: $\mathbf{U}_{(\ell)}$ is defined only by the terms associated to $\e_\ell$, while $\mathbf{U}^{(\ell)}$ is the $\ell$-th entry of the vector $\mathbf{U}$. All of the components of $\mathbf{U}^{(\ell)}$ for $L:1\leq L\neq \ell\leq p-1$ are zero, but this is not true for $\mathbf{U}_{(\ell)}$. 
\end{remark}

The Poincar\'e series $\mathbf{U}(z;s,k,m,\mu)$ and $\mathbf{U}_{(\ell)}(z;s,k,m,\mu)$ converge absolutely and uniformly in any compact subset of $\re s>1$ and are in $\Lform_{k}(\Gamma_0(p),\mu)$. To show the transformation law, e.g. 
\[\mathbf{U}(\gamma_1 z;s,k,m,\mu)=\mu(\gamma_1)j(\gamma_1,z)^{k}\mathbf{U}(z;s,k,m,\mu) \quad \text{for }\gamma_1\in \Gamma_0(p), \]
it suffices to show that for $\gamma_1,\gamma_2\in \Gamma_0(p)$, we have
\begin{align}\label{TransformationLawOfPoincare Cusp Infty}
	\begin{split}
		\mu(\gamma_2\gamma_1^{-1})^{-1} &j(\gamma_2\gamma_1^{-1}, \gamma_1 z)^{-k}
		=\(\mu(\gamma_2)\mu(\gamma_1^{-1})\)^{-1} \overline{w_{k}(\gamma_2,\gamma_1^{-1})}j(\gamma_2\gamma_1^{-1}, \gamma_1 z)^{-k}\\
		&=\mu(\gamma_1)\mu(\gamma_2)^{-1} j(\gamma_2,z)^{-k}j(\gamma_1^{-1},\gamma_1 z)^{-k} j(\gamma_2\gamma_1^{-1},\gamma_1 z)^k j(\gamma_2\gamma_1^{-1}, \gamma_1 z)^{-k}\\
		&=\mu(\gamma_1)j(\gamma_1,z)^{k}\cdot \mu(\gamma_2)^{-1}j(\gamma_2,z)^{-k}, 
	\end{split}
\end{align}
where we have used this trick: since $w_{k}(\gamma,\gamma')$ does not depend on $z\in \HH$, we have
\begin{equation}\label{trick w k}
	w_{k}(\gamma,\gamma')=j(\gamma',\gamma''z)^k j(\gamma,\gamma'\gamma'' z)^kj(\gamma\gamma',\gamma''z)^{-k} \quad \text{for all }\gamma''\in \SL_2(\Z), 
\end{equation}
as well as the properties $\mu(\gamma_1^{-1})=\mu(\gamma_1)^{-1}$ and $j(\gamma_1^{-1},\gamma_1 z)^{k}j(\gamma_1,z)^{k}=1$. 

For $\re s>1$, we can compute the Fourier expansion of $\mathbf{U}(z;s,k,m,\mu)$ in the same way as the scalar-valued case. The contribution from $c=0$ equals
\begin{align*}
	\sum_{\ell=1}^{p-1}\frac{y^s}{\sin(\frac{\pi\ell}p)}e(m_{\pm \infty}z)\e_{\ell}. 
\end{align*}
When $c>0$, the contribution from a single $c$ equals
\begin{align*}
	&\sum_{\ell=1}^{p-1} \!\!\sum_{\substack{d\Mod c^*\\ad\equiv 1\Mod c}}\!\!\sum_{t\in \Z}\mu(\begin{psmallmatrix}
		a&b+ta\\c&d+tc
	\end{psmallmatrix})^{-1}\(\frac{cz+d+tc}{|cz+d+tc|}\)^{-k}\frac{y^s}{|cz+d+tc|^{2s}}\frac{e\(m_{\pm\infty}\frac{az+b+ta}{cz+d+tc}\)}{\sin(\frac {\pi\ell} p)}\e_\ell\\
	&=\sum_{\ell=1}^{p-1}\!\! \sum_{\substack{d\Mod c^*\\ad\equiv 1\Mod c}}\!\!\!\!\frac{\mu(\begin{psmallmatrix}
			a&*\\c&d
		\end{psmallmatrix})^{-1}}{c^{2s}} e\(\frac{m_{\pm\infty}a}c\)\sum_{t\in \Z}\(\frac{z+\frac dc+t}{|z+\frac dc+t|}\)^{-k}\frac{e(t\alpha_{\pm\infty})y^s}{|z+\frac dc+t|^{2s}}\frac {e\(-\frac {m_{\pm\infty}}{c^2(z+\frac dc+t)}\)\e_\ell}{\sin(\frac {\pi\ell} p)}, 
\end{align*}
where we have used $\frac{az+b}{cz+d}=\frac ac-\frac1{c(cz+d)}$, Definition~\ref{mu matrix define}, and the property \eqref{MultiplierSystemBasicProprety} for $\nu_\eta$. If we denote 
\[f_{\pm}(z)\defeq \sum_{t\in \Z}\(\frac{z+t}{|z+t|}\)^{-k}\frac{e(t\alpha_{\pm\infty})\,y^s}{|z+t|^{2s}}e\(-\frac {m_{\pm\infty}}{c^2(z+t)}\), \]
then $f_{\pm}(z)e(\alpha_{\pm \infty}x)$ has period $1$ and we get the following Fourier expansion by Poisson summation: 
\[f_{\pm}(z)=\sum_{n_{\pm\infty}>0}\frac{B(c,m_{\pm\infty},t_{\pm\infty},y,s,k)}{\sin(\frac{\pi\ell}c)}e(n_{\pm\infty}x)\quad \text{and}\quad f_{\pm}\(z+\frac dc\)=e\(\frac{n_{\pm\infty}d}{c}\)f_{\pm}(z),\]
where (with the substitution $x=yu$)
\begin{equation}\label{Bcmtysk function}
	B(c,m_{\pm\infty},t_{\pm\infty},y,s,k)\defeq y\int_{\R} \(\frac{u+i}{|u+i|}\)^{-k}e\(\frac{-m_{\pm\infty}}{c^2 y(u+i)}-t_{\pm\infty}yu\)\frac{du}{y^{2s}(u^2+1)^s}. 
\end{equation}
Therefore, the Fourier expansion of $\mathbf{U}(z;s,k,m,\mu)$ at the cusp $\infty$ is 
\begin{align}\label{Fourier expan Poincare series infty infty}
	\begin{split}
		\mathbf{U}(z;s,k,m,\mu)&=\sum_{\ell=1}^{p-1}\frac{y^se(m_{\pm \infty}z)\e_{\ell}}{\sin(\frac{\pi\ell}p)} \\
		&+y^s\sum_{t\in \Z}\sum_{p|c>0}\frac{\mathbf{S}_{\infty\infty}(m,t,c,\mu)}{c^{2s}}B(c,m_{\pm\infty},t_{\pm\infty},y,s,k)e(t_{\pm\infty} x), 
	\end{split}
\end{align} 
where $\mathbf{S}_{\infty\infty}(m,n,c,\mu)$ is defined in \eqref{S infty infty def}.

Moreover, for $\re s>1$ and $\lambda=s(1-s)$, we still have the recursion relation
\begin{align}\label{meromorphic continuation cusp infty}
	\mathbf{U}(z;s,k,m,\mu)&=-4\pi m_{\pm\infty}(s-\tfrac k2) \mathscr{R}_{\lambda}\(\mathbf{U}(z;s+1,k,m,\mu)\), \\
	\mathbf{U}_{(\ell)}(z;s,k,m,\mu)&=-4\pi m_{\pm\infty}(s-\tfrac k2) \mathscr{R}_{\lambda}\(\mathbf{U}_{(\ell)}(z;s+1,k,m,\mu)\), 
\end{align}
where $\mathscr{R}_{\lambda}=(\Delta_{k}+\lambda)^{-1}$ is the resolvent of $\Delta_{k}$. Since $\frac12<\re s<1$ implies $\lambda<\frac14$, we know that $\mathbf{U}(z;s,k,m,\mu)$ and every $\mathbf{U}_{(\ell)}(z;s,k,m,\mu)$ can be meromorphically continued to the half plane $\re s>\frac 12$ with a finite number of simple poles at $s_j$ for $\frac12<s_j<\frac 34$.

Recall that we choose $m_{+\infty}<0$ in the condition for Theorem~\ref{Sums of Kloosterman sums} (hence $m\leq 0$) and recall $(1-m)_{-\infty}=-m_{+\infty}$. Also recall that $\mathrm{OB}(r_j)$ is a orthonormal basis of $\LEigenform_{\frac 12}(\Gamma_0(p),\mu_p,r_j)$. 
The residue of $\overline{\mathbf{U}(\cdot;\overline{s},-\tfrac12,1-m,\overline{\mu_p})}$ at $s=s_j$ is then given by 
\[\sum_{\mathbf{V}\in \mathrm{OB}(r_j)}\Res_{s=s_j} \left \langle\overline{\mathbf{U}(\cdot;\overline{s},-\tfrac12,1-m,\overline{\mu_p})}, \mathbf{V}(\cdot;r_j)\right \rangle\mathbf{V}(z;r_j) .  \]
We compute the inner product (defined by \eqref{vec val Petersson inner prod}) by applying the Mellin transform \eqref{integral e y W for inner product}. Note that $W_{\frac k2 \sgn m_{+\infty}, ir_j}(4\pi|m_{+\infty}|y)\in \R$ and $s_j=\frac 12+ir_j$ for $r_j\in i(0,\frac 14)$. We get  
\begin{align*}
	&\left\langle\overline{\mathbf{U}(\cdot;\overline{s},-\tfrac12,1-m,\overline{\mu_p})}, \mathbf{V}(\cdot;r_j)\right \rangle\\
	&=\int_{\Gamma_0(p)\setminus \HH} \mathbf{V}(z;r_j)^{\mathrm H}\(\overline{\sum_{\ell=1}^{p-1}\sum_{\gamma\in \Gamma_\infty\setminus \Gamma_0(p)}\overline{\mu_p}(\gamma)^{-1}j(\gamma,z)^{\frac12}(\im \gamma z)^{\overline s}\frac{e((1-m)_{-\infty}\gamma z)}{\sin(\frac{\pi\ell}p)}\e_\ell}\)\frac{dxdy}{y^2}\\
	&=\sum_{\ell=1}^{p-1}\int_{\Gamma_{\infty}\setminus \HH} \mathbf{V}(z;r_j)^{\mathrm H}y^se(m_{+\infty}x)e^{2\pi m_{+\infty}y}\frac{\e_\ell}{\sin(\frac{\pi\ell}p)}\frac{dxdy}{y^2}\\
	&=\sum_{\ell=1}^{p-1}\frac{\overline{\rho_\infty^{(\ell)}(m)}}{\sin(\frac{\pi \ell}p)} \int_0^\infty e^{-2\pi |m_{+\infty}|y}W_{-\frac 14, ir_j}(4\pi |m_{+\infty}| y)y^{s-1}\frac{dy}y\\
	&= (4\pi |m_{+\infty}|)^{1-s}\frac{\Gamma(s-s_j)\Gamma(s+s_j-1)}{\Gamma(s+\frac14)}\sum_{\ell=1}^{p-1}\frac{\overline{\rho_\infty^{(\ell)}(m)}}{\sin(\frac{\pi \ell}p)}. 
\end{align*}
The residue of $\overline{\mathbf{U}(\cdot;\overline{s},-\tfrac12,1-m,\overline{\mu_p})}$ at $s=s_j$ is then
\begin{align}\label{U 1-m residue at sj}
	\begin{split}
		\sum_{\mathbf{V}\in \mathrm{OB}(r_j)}\Res_{s=s_j}& \left\langle\overline{\mathbf{U}(\cdot;\overline{s},-\tfrac12,1-m,\overline{\mu_p})}, \mathbf{V}(\cdot;r_j) \right\rangle \mathbf{V}(z;r_j)\\
		&=\sum_{\mathbf{V}\in \mathrm{OB}(r_j)} (4\pi |m_{+\infty}|)^{1-s_j}\frac{\Gamma(2s_j-1)}{\Gamma(s_j+\frac14)}\sum_{\ell=1}^{p-1}\frac{\overline{\rho_\infty^{(\ell)}(m)}}{\sin(\frac{\pi \ell}p)} \mathbf{V}(z;r_j).
	\end{split}
\end{align}

The following lemma still holds as \cite[Lemma~1]{gs} because the proofs are essentially the same using the vector-valued inner product defined in \eqref{vec val Petersson inner prod}. 
\begin{lemma}\label{Inner product U U U ell U ell}
	Let $s=\sigma+it$, $(k,\mu)=(\frac12,\mu_p)$ or $(-\frac12,\overline{\mu_p})$ and   $m>0$. For $\frac12<\sigma\leq 2$ and $|t|>1$ we have
	\begin{align*}
		\left\langle\mathbf{U}(\cdot;s,k,m,\mu), \mathbf{U}(\cdot;s,k,m,\mu)\right \rangle&\ll _p \frac{m^2}{(\sigma-\frac12)^2},\\
		\left\langle\mathbf{U}_{(\ell)}(\cdot;s,k,m,\mu), \mathbf{U}_{(\ell)}(\cdot;s,k,m,\mu)\right \rangle&\ll_p \frac{m^2}{(\sigma-\frac12)^2}.
	\end{align*}
\end{lemma}

The following useful equation follows from \cite[(3.384.9)]{table}: for $y>0$, $\beta\neq 0$, $k=\pm \frac12$, $\re s>\frac12$, we have
\begin{equation}\label{Integral of u to get Whittaker}
	\int_{\R}\frac{(x+i)^{-k}}{(x^2+1)^{s-\frac k2}}e^{-2\pi i \beta xy}dx=
	\frac{	e(-\tfrac k4)\pi (\pi |\beta| y)^{s-1}}{\Gamma(s+\frac k2\sgn{\beta})}W_{\frac k2\sgn \beta,\,\frac12-s}(4\pi |\beta| y). 
\end{equation}
In the next lemma we compute the inner product of two Poincar\'e series. Recall the definition in \eqref{KloostermanSelbergZetaFunctionVVal}. 
\begin{lemma}\label{Inner product U infty U infty}
	Suppose that $m_{+\infty}<0$ and let $1\leq \ell\leq p-1$. Then
	\[Z_{\substack{\infty\infty\\m,n,+}}^{(\ell)}(s) \quad \text{is meromorphic in }\re s>\tfrac12\]
	with at most a finite number of simple poles in $(\tfrac12,1)$. Moreover, when $\re s>\frac12$, for $n_{+\infty}>0$ we have
	\begin{align*}
		&\left\langle \overline{\mathbf{U}(\cdot;\overline s,-\tfrac12,1-m,\overline{\mu_p})},\mathbf{U}_{(\ell)}(\cdot;\overline{s}+2,\tfrac12,n,\mu_p)\right\rangle \\
		&=\frac{e(-\frac 18)}{4^{s+1}\pi n_{+\infty}^2}\cdot\frac{\Gamma(2s+1)}{\Gamma(s+\frac 14)\Gamma(s+\frac 74)}Z_{\substack{\infty\infty\\m,n,+}}^{(\ell)}(s)+O_p\(\frac{|m_{+\infty}n_{+\infty}|}{\sigma-\frac12}\),
	\end{align*}
	for $ n_{+\infty}<0$ we have
	\begin{align*}
		&\left\langle \overline{\mathbf{U}(\cdot;\overline s,-\tfrac12,1-m,\overline{\mu_p})},\ \overline{\mathbf{U}_{(\ell)}(\cdot;s+2,-\tfrac12,1-n,\overline{\mu_p})}\right\rangle \\
		&=\frac{e(-\frac 18)}{4^{s+1}\pi |n_{+\infty}|^{2}}\cdot\frac{\Gamma(2s+1)}{\Gamma(s-\frac 14)\Gamma(s+\frac 94)}Z_{\substack{\infty\infty\\m,n,+}}^{(\ell)}(s)+O_p\(\frac{|m_{+\infty}n_{+\infty}|}{\sigma-\frac12}\). 
	\end{align*}
\end{lemma}
\begin{proof}
	We compute as in \cite[Lemma~2]{gs} using the following properties: $\mu_p(\gamma)$ is unitary, $\mu_p(\gamma^{-1})=\mu_p(\gamma)^{-1}$, and $j(\gamma^{-1},z)^{\frac 12}j(\gamma,\gamma^{-1}z)^{\frac 12}=1$. For the first equality in the following computation, we write $\overline{\mathbf U}$ to denote the whole first term in the inner product. Suppose that $\re s_2>\re s_1>\frac 12$. When $n_{+\infty}>0$, we have 
	\begin{align*}
		&\left\langle \overline{\mathbf{U}(\cdot;\overline s_1,-\tfrac12,1-m,\overline{\mu_p})},\mathbf{U}_{(\ell)}(\cdot;\overline{s_2},\tfrac12,n,\mu_p)\right\rangle \\
		&=\int_{\Gamma_0(p)\setminus \HH}\(\sum_{\gamma\in \Gamma_\infty\setminus \Gamma_0(p)}\mu_p(\gamma)^{-1}j(\gamma,z)^{-\frac12}(\im \gamma z)^{\overline{s_2}} \frac{e(n_{+\infty}\gamma z)}{\sin(\frac{\pi \ell}p)}\e_{\ell}\)^{\mathrm H} \overline{\mathbf{U}}\;\frac{dxdy}{y^2}\\
		&=\int_{\Gamma_0(p)\setminus \HH}\sum_{\gamma\in \Gamma_\infty\setminus \Gamma_0(p)}j(\gamma,\gamma^{-1}z)^{\frac12}y^{s_2}e(-n_{+\infty}x)\frac{e^{-2\pi n_{+\infty}y}}{\sin(\frac{\pi \ell}p)}\e_\ell^{\mathrm T}\mu_p(\gamma)\\
		&\qquad \cdot\overline{\overline{\mu_p}(\gamma^{-1})}j(\gamma^{-1},z)^{\frac12}\overline{\mathbf{U}(z;\overline{s_1},-\tfrac12,1-m,\overline{\mu_p})}\frac{dxdy}{y^2}\\
		&=\int_{\Gamma_\infty\setminus\HH} y^{s_2}e(-n_{+\infty}x)\frac{e^{-2\pi n_{+\infty}y}}{\sin(\frac{\pi \ell}p)}\e_\ell^{\mathrm T}\,\overline{\mathbf{U}(z;\overline{s_1},-\tfrac12,1-m,\overline{\mu_p})}\frac{dxdy}{y^2}. 
	\end{align*}
	Then we use the Fourier expansion \eqref{Fourier expan Poincare series infty infty} with \eqref{S infty infty ell def} to continue: 
	\begin{align*}
		&=\sum_{p|c>0}\frac{S_{\infty\infty}^{(\ell)}(m,n,c,\mu_p)}{c^{2s_1}}\int_0^\infty y^{s_2-s_1}e^{-2\pi n_{+\infty}y}\overline{B(c,(1-m)_{-\infty},(1-n)_{-\infty},y,\overline{s_1},-\tfrac12)}\frac{dy}y\\
		&=Z_{\substack{\infty\infty\\m,n,+}}^{(\ell)}(s_1)\int_0^\infty\int_\R \frac{y^{s_2-s_1}e^{-2\pi n_{+\infty}y}}{(u^2+1)^{s_1}}\(\frac{u+i}{|u+i|}\)^{-\frac12}e(-n_{+\infty}uy)\frac{dxdy}y+R_{\substack{\infty\infty\\m,n,+}}^{(\ell)}(s_1,s_2)\\	
		&=\frac{e(-\frac 18)}{4^{s+1}\pi n_{+\infty}^2}\cdot\frac{\Gamma(2s_1+1)}{\Gamma(s_1+\frac 14)\Gamma(s_1+\frac 74)}Z_{\substack{\infty\infty\\m,n,+}}^{(\ell)}(s_1)+R_{\substack{\infty\infty\\m,n,+}}^{(\ell)}(s_1,s_2), 
	\end{align*}
	where 
	\begin{align*}
		R_{\substack{\infty\infty\\m,n,+}}^{(\ell)}(s_1,s_2)&=Z_{\substack{\infty\infty\\m,n,+}}^{(\ell)}(s_1)\int_0^\infty\int_\R \frac{y^{s_2-s_1}e^{-2\pi n_{+\infty}y}}{(u^2+1)^{s_1}}\(\frac{u+i}{|u+i|}\)^{-\frac12}\\
		&\qquad\quad \cdot e(-n_{+\infty}uy)\(e\(\frac{-m_{+\infty}(u+i)}{c^2y(u^2+1)}\)-1\)\frac{dxdy}y\\
		&=Z_{\substack{\infty\infty\\m,n,+}}^{(\ell)}(s_1)O\(\frac{|m_{+\infty}n_{+\infty}|}{c^2(\sigma_1-\frac12)}\)	=O\(\frac{|m_{+\infty}n_{+\infty}|}{\sigma_1-\frac12}\),
	\end{align*}
	which is holomorphic when $\sigma_1>\frac12$. The lemma follows by setting $s_2=s_1+2$. 
	
	Similarly, when $n_{+\infty}<0$ we have
	\begin{align*}
		&\left\langle \overline{\mathbf{U}(\cdot;\overline s_1,-\tfrac12,1-m,\overline{\mu_p})},\ \overline{\mathbf{U}_{(\ell)}(\cdot;s_2,-\tfrac12,1-n,\overline{\mu_p})}\right\rangle \\
		&=\int_{\Gamma_\infty\setminus\HH} y^{s_2}e(-n_{+\infty}x)\frac{e^{2\pi n_{+\infty}y}}{\sin(\frac{\pi \ell}p)}\e_\ell^{\mathrm T}\overline{\mathbf{U}(z;\overline{s_1},-\tfrac12,1-m,\overline{\mu_p})}\frac{dxdy}{y^2}\\
		&=\frac{e(-\frac 18)\pi^{1+s_1-s_2}}{4^{s_2-1} |n_{+\infty}|^{s_2-s_1}}\cdot\frac{\Gamma(s_2-s_1)\Gamma(s_2+s_1-1)}{\Gamma(s_1-\frac 14)\Gamma(s_2+\frac 14)}Z_{\substack{\infty\infty\\m,n,+}}^{(\ell)}(s_1)+R_{\substack{\infty\infty\\m,n,+}}^{(\ell)}(s_1,s_2)
	\end{align*}
	where
	\begin{align*}
		R_{\substack{\infty\infty\\m,n,\mu_p}}^{(\ell)}(s_1,s_2)
		&=\sum_{p|c>0}\frac{S_{\infty\infty}^{(\ell)}(m,n,c,\mu_p)}{c^{2s_1}}\int_0^\infty\int_\R \frac{y^{s_2-s_1}e^{2\pi n_{+\infty}y}}{(u^2+1)^{s_1}}\(\frac{u+i}{|u+i|}\)^{-\frac12}\\
		&\qquad \cdot e(-n_{+\infty}uy)\(e\(\frac{-m_{+\infty}(u+i)}{c^2y(u^2+1)}\)-1\)\frac{dxdy}y\\
		&=\sum_{p|c>0}\frac{S_{\infty\infty}^{(\ell)}(m,n,c,\mu_p)}{c^{2s_1}}O\(\frac{|m_{+\infty}n_{+\infty}|}{c^2(\sigma_1-\frac12)}\)	=O\(\frac{|m_{+\infty}n_{+\infty}|}{\sigma_1-\frac12}\). 
	\end{align*}
	The lemma follows by setting $s_2=\overline{s_1}+2$. 
\end{proof}

Combining Lemma~\ref{Inner product U U U ell U ell} and Lemma~\ref{Inner product U infty U infty}, we have the following proposition. 
\begin{proposition}\label{gs Theorem 1}
	For $m_{+\infty}<0$, $s=\sigma+it$, $\sigma>\frac12$ and $|t|\rightarrow \infty$, we have
	\[Z_{\substack{\infty\infty\\m,n,+}}^{(\ell)}(s)\ll_{p} \frac{|m_{+\infty}n_{+\infty}|^3|s|^{\frac12}}{\sigma-\frac12}. \]
\end{proposition}

Now we can prove the first case of Theorem~\ref{Sums of Kloosterman sums}.

\begin{proof}[Proof of \eqref{Kloosterman sums asymptotics infty infty} in Theorem~\ref{Sums of Kloosterman sums}]
	
	Denote $s=\sigma+it$. For any $\ep>0$, we have $Z_{\substack{\infty\infty\\m,n,+}}^{(\ell)}(1+\ep+it)\ll_{p,\ep} \zeta(1+\ep)$.  By Proposition~\ref{gs Theorem 1} and the Phragm\'en-Lindel\"of principle,  we have
	\begin{equation}\label{Zinftyinfty after phragmen lindelof}
		Z_{\substack{\infty\infty\\m,n,+}}^{(\ell)}(\tfrac{1+s}2)\ll_{p,\ep} |m_{+\infty}n_{+\infty}|^{3}|t|^{\frac12-\frac \sigma 2+\ep}\quad \text{for }0<\ep\leq \sigma\leq 1+\ep.
	\end{equation} 
	Then following the argument of \cite[\S3]{gs}, by the proof of the prime number theorem as in \cite[\S18]{davenport}, we have
	\[
	\sum_{p|c\leq x}\frac{S^{(\ell)}_{\infty\infty}(m,n,c,\mu_p)}c=\frac1{2\pi i}\int_{1-iT}^{1+iT}Z_{\substack{\infty\infty\\m,n,+}}^{(\ell)}(\tfrac{1+s}2)\frac{x^s}s ds+O\(|m_{+\infty}n_{+\infty}|^3\frac{x^{1+\ep}}T\). 
	\]
	By Lemma~\ref{Inner product U infty U infty}, the function $Z_{\substack{\infty\infty\\m,n,+}}^{(\ell)}(\tfrac{1+s}2)$ has at most a finite number of simple poles at $2s_j-1\in (0,1)$ (note that we are using $\frac{1+s}2$ rather than $s$), where $\lambda_j=s_j(1-s_j)<\frac 14$ are the discrete eigenvalues of $\Delta_{\frac 12}$ on $\Lform_{\frac 12}(\Gamma_0(p),\mu_p)$. Shifting the line of integration above to $\re s=\ep$ such that $2s_j-1>\ep$ for all $\lambda_j<\frac 14$, with the help of \eqref{Zinftyinfty after phragmen lindelof} we obtain
	\begin{align*}
		\sum_{p|c\leq x}\frac{S^{(\ell)}_{\infty\infty}(m,n,c,\mu_p)}c=\sum_{\frac12<s_j<\frac 34}\Res_{s=2s_j-1}Z_{\substack{\infty\infty\\m,n,+}}^{(\ell)}(\tfrac{1+s}2)\frac{x^{2s_j-1}}{2s_j-1}+O(|m_{+\infty}n_{+\infty}|^3x^{\frac 13+\ep}),
	\end{align*}
	where we have chosen $T=x^{\frac 23}$. 
	
	For the residue, by Lemma~\ref{Inner product U infty U infty}, it suffices to compute the residue of the two inner products in that lemma. 
	When $n_{+\infty}>0$, combining Lemma~\ref{Inner product U infty U infty} and \eqref{U 1-m residue at sj}, we are left to compute the following inner product using \eqref{integral e y W for inner product}: 
	\begin{align}\label{Inner product V and U infty n}
		\begin{split}
			\left\langle \mathbf{V}(\cdot;r_j),\mathbf{U}_{(\ell)}(\cdot;\overline{s_j}+2,\tfrac12,n,\mu_p)\right\rangle&=\int_{\Gamma_\infty\setminus \Gamma_0(p)} y^{s_j+2}e(-n_{+\infty}x)\frac{e^{-2\pi n_{+\infty}y}\e_\ell^{\mathrm T}}{\sin(\frac{\pi \ell}p)}\mathbf{V}(z;r_j)\frac{dxdy}{y^2}\\
			&=(4\pi n_{+\infty})^{-s_j-1}\frac{\rho_{j,\infty}^{(\ell)}(n)}{\sin(\frac{\pi\ell}p)}\frac{\Gamma(2s_j+1)}{\Gamma(s_j+\frac 74)}. 
		\end{split}
	\end{align}
	Similarly, when $n_{+\infty}<0$ we compute: 
	\begin{align}\label{Inner product V and U infty 1-n}
		\begin{split}
			\left\langle \mathbf{V}(\cdot;r_j),\overline{\mathbf{U}_{(\ell)}(\cdot;s_j+2,-\tfrac12,1-n,\overline{\mu_p})}\right\rangle&=\int_{\Gamma_\infty\setminus \Gamma_0(p)} y^{s_j+2}e(-n_{+\infty}x)\frac{e^{2\pi n_{+\infty}y}\e_\ell^{\mathrm T}}{\sin(\frac{\pi \ell}p)}\mathbf{V}(z;r_j)\frac{dxdy}{y^2}\\
			&=(4\pi |n_{+\infty}|)^{-s_j-1}\frac{\rho_{j,\infty}^{(\ell)}(n)}{\sin(\frac{\pi\ell}p)}\frac{\Gamma(2s_j+1)}{\Gamma(s_j+\frac 94)}. 
		\end{split}
	\end{align}
	Combining Lemma~\ref{Inner product U infty U infty} with \eqref{U 1-m residue at sj}, \eqref{Inner product V and U infty n} and \eqref{Inner product V and U infty 1-n}, we get
	\begin{align}
		\begin{split}
			\Res_{s=s_j} Z_{\substack{\infty\infty\\m,n,+}}^{(\ell)}(s)&=\!\!\!\sum_{\mathbf{V}\in \mathrm{OB}(r_j)}\!\!\! e(\tfrac 18)\(\sum_{L=1}^{p-1}\frac{\overline{\rho_{j,\infty}^{(L)}(m)}}{\sin(\frac{\pi L}p)}\)\frac{\rho_{j,\infty}^{(\ell)}(n)}{\sin(\frac{\pi\ell}p)}\cdot\frac{\Gamma(s_j+\frac14 \sgn n_{+\infty})\Gamma(2s_j-1)}{\pi^{2s_j-1}|4m_{+\infty}n_{+\infty}|^{s_j-1}\Gamma(s_j-\frac14)}
		\end{split}
	\end{align}
	and finish the proof of \eqref{Kloosterman sums asymptotics infty infty} in Theorem~\ref{Sums of Kloosterman sums}.
\end{proof}

\subsection{The cusp 0}
Recall from \eqref{scaling matrix def} that $\Gamma_0$ is the stabilizer of the cusp $0$ in $\Gamma_0(p)$ and $\sigma_0=\begin{psmallmatrix}
	0&-1/\sqrt p\\\sqrt p&0
\end{psmallmatrix}$ is the scaling matrix of the cusp $0$. They satisfy the property $\sigma_0^{-1}\Gamma_0\sigma_0=\Gamma_\infty$, where $\Gamma_\infty=\{\pm \begin{psmallmatrix}
	1&n\\0&1
\end{psmallmatrix}:n\in \Z\}$ is the stabilizer of the cusp $\infty$.

Fix an integer $r\geq 0$. Recall the definition of $x_r$ in \eqref{x r def}. For $(k,\mu)=(\frac12,\mu_p)$ or $(-\frac12,\overline{\mu_p})$, recall the definition of $\alpha_{\pm 0}^{(\ell)}$ and $m_{\pm 0}^{(\ell)}$ in \eqref{alpha 0 l def} and of $\condr$ in \eqref{(a,r) def}. 

For $\mathbf{m}>0$, we define
\begin{align}\label{U 0 l def}
	\begin{split}
		&\mathbf{U}_{0,(\ell)}(z;s,k,m^{(\ell)},\mu,r)\defeq\\
		&\left\{
		\begin{array}{ll}
			{\displaystyle\!\!\!\!\sum_{\gamma\in \Gamma_0\setminus \Gamma_0(p)}\!\!\!\!\!\mu(\gamma)^{-1}\overline{w_k(\sigma_0^{-1},\gamma)}j(\sigma_0^{-1}\gamma, z)^{-k}(\im \sigma_0^{-1}\gamma z)^s e\(m_{\pm0}^{(\ell)}\sigma_0^{-1}\gamma z\)\e_\ell ,}& \text{if }\ell\in \condr,\\
			0\e_\ell,&\text{otherwise, never used.} 
		\end{array}
		\right.
	\end{split}
\end{align}
We also define
\begin{equation}\label{U 0 def}
	\mathbf{U}_0(z;s,k,\mathbf{m},\mu,r)\defeq
	\sum_{\ell\in \condr} \mathbf{U}_{0,(\ell)}(z;s,k,m^{(\ell)},\mu,r). 
\end{equation}
Note that $\mathbf{U}_{0,(\ell)}(z;s,k,m^{(\ell)},\mu,r)$ is different from $\mathbf{U}_{0}^{(\ell)}(z;s,k,\mathbf{m},\mu,r)$, where $\mathbf{U}_{0}^{(\ell)}(z;s,k,\mathbf{m},\mu,r)$ is defined to have the same $\ell$-th entry as $\mathbf{U}_{0}(z;s,k,\mathbf{m},\mu,r)$ but has $0$ in the other entries. 
The Poincar\'e series $\mathbf{U}_0(z;s,k,\mathbf{m},\mu,r)$ and every $\mathbf{U}_{0,(\ell)}(z;s,k,m^{(\ell)},\mu,r)$ converge absolutely and uniformly in any compact subset of $\re s>1$, and are in $\Lform_{k}(\Gamma_0(p),\mu)$. To show the transformation law, e.g. 
\[\mathbf{U}_0(\gamma_1 z;s,k,\mm,\mu,r)=\mu(\gamma_1)j(\gamma_1,z)^{k}\mathbf{U}_0(z;s,k,\mm,\mu,r) \quad \text{for }\gamma_1\in \Gamma_0(p), \]
it suffices to show that for $\gamma_1,\gamma_2\in \Gamma_0(p)$, we have
\begin{align}\label{TransformationLawOfPoincare Cusp 0}
	\begin{split}
		&\mu(\gamma_2\gamma_1^{-1})^{-1}\cdot\overline{w_k(\sigma_0^{-1},\gamma_2\gamma_1^{-1})} j(\sigma_0^{-1}\gamma_2\gamma_1^{-1}, \gamma_1 z)^{-k}\\
		&=\(\mu(\gamma_2)\mu(\gamma_1^{-1})\)^{-1} \overline{w_{k}(\gamma_2,\gamma_1^{-1})}\\
		&\qquad \cdot 
		j(\sigma_0^{-1},\gamma_2 z)^{-k}j(\gamma_2\gamma_1^{-1},\gamma_1 z)^{-k} j(\sigma_0^{-1}\gamma_2\gamma_1^{-1},\gamma_1 z)^k
		j(\sigma_0^{-1}\gamma_2\gamma_1^{-1}, \gamma_1 z)^{-k}\\
		&=\mu(\gamma_1)\mu(\gamma_2)^{-1} j(\gamma_2,z)^{-k}j(\gamma_1^{-1},\gamma_1 z)^{-k} j(\gamma_2\gamma_1^{-1},\gamma_1 z)^{k} \cdot j(\sigma_0^{-1},\gamma_2 z)^{-k}j(\gamma_2\gamma_1^{-1},\gamma_1 z)^{-k}\\
		&=\mu(\gamma_1)j(\gamma_1,z)^{k}\cdot \mu(\gamma_2)^{-1}j(\sigma_0^{-1},\gamma_2 z)^{-k}j(\gamma_2,z)^{-k}\\
		&=\mu(\gamma_1)j(\gamma_1,z)^{k}\cdot \mu(\gamma_2)^{-1}\overline{w_k(\sigma_0^{-1},\gamma_2)}j(\sigma_0^{-1}\gamma_2,z)^{-k}, 
	\end{split}
\end{align}
where we have used the trick in \eqref{trick w k},  $\mu(\gamma_1^{-1})=\mu(\gamma_1)^{-1}$ and $j(\gamma_1^{-1},\gamma_1 z)^{k}j(\gamma_1,z)^{k}=1$.

Recall the double coset decomposition \eqref{double coset decomposition} and the choice of $\gamma_2$ below it: 
\[\sigma_0^{-1}\Gamma_0(p)\sigma_\infty=\bigcup_{\substack{a>0\\p\nmid a}}\bigcup_{b\Mod a^*}\Gamma_{\infty}\begin{pmatrix}
	\frac{c}{\sqrt p}&\frac{d}{\sqrt p}\\
	-a\sqrt p&-b\sqrt p
\end{pmatrix}\Gamma_\infty, \]
\[\sigma_0^{-1}(\Gamma_0\setminus\Gamma_0(p))=\{\sigma_0^{-1}\begin{psmallmatrix}
	a&b\\c&d
\end{psmallmatrix}\begin{psmallmatrix}
	1&t\\0&1
\end{psmallmatrix}:\, a>0,\ p\nmid a,\ b\Mod a^*,\  t\in \Z\}. \]
Recall \eqref{mu 0 infty def}, Definition~\ref{mu matrix define}, \eqref{alpha infty l def} and the property \eqref{MultiplierSystemBasicProprety} for $\nu_\eta$. We have
\[\mu_{0\infty}(\gamma)=\mu(\sigma_0\gamma)w_k(\sigma_0^{-1},\sigma_0\gamma)\quad \text{and}\quad \mu_{0\infty}(\sigma_0^{-1}\begin{psmallmatrix}
	a&b\\c&d
\end{psmallmatrix}\begin{psmallmatrix}
	1&t\\0&1
\end{psmallmatrix})=\mu_{0\infty}(\sigma_0^{-1}\begin{psmallmatrix}
	a&b\\c&d
\end{psmallmatrix})e(-t\alpha_{\pm\infty}).  \]
For $\gamma=\begin{psmallmatrix}
a&b\\c&d
\end{psmallmatrix}$, $a>0$ and $c>0$, we also have
\[w_{k}(\sigma_0^{-1},\gamma)j(\sigma_0^{-1}\gamma,z)^k=j(\gamma,z)^kj(\sigma_0^{-1},\gamma z)^k=e(-\tfrac k2)(\tfrac{az+b}{|az+b|})^{k}.\]
To compute the Fourier expansion of $\mathbf{U}_0(z;s,k,\mathbf{m},\mu,r)$ when $\re s>1$, we have
\begin{align*}
	&\mathbf{U}_0(z;s,k,\mathbf{m},\mu,r)\\
	&=\sum_{\ell\in \condr}\!\!\!\sum_{\substack{\gamma=\begin{psmallmatrix}
				\frac c{\sqrt p}&\frac d{\sqrt p}\\
				-a\sqrt p&-b\sqrt p
			\end{psmallmatrix}\\\gamma \in \Gamma_\infty\setminus\sigma_0^{-1}\Gamma_0(p)}}\!\!\!\!\mu_{0\infty}(\gamma)^{-1}\(\frac{-az-b}{|-az-b|}\)^{-k}\frac{y^s e(m_{\pm 0}^{(\ell)}\gamma z)}{|a\sqrt p z+b\sqrt p|^{2s}}\e_\ell\\
	&=\sum_{\ell\in \condr}\sum_{\substack{ a>0\\p\nmid a}}\sum_{\substack{b(a)^*\\0<c<pa,\,p|c\\ad-bc=1}}\!\!\!\!\!\!\frac{\mu_{0\infty}\(\sigma_0^{-1}\begin{psmallmatrix}
			a&b\\c&d
		\end{psmallmatrix}\)^{-1}}{p^s}\sum_{t\in \Z}e(t\alpha_{\pm\infty})\(\frac{-z-\frac ba-t}{|-z-\frac ba-t|}\)^{-k}\frac{y^se\(\frac{m_{\pm 0}^{(\ell)}(cz+d+tc)}{-paz-pb-pta}\)}{|az+b+ta|^{2s}}\e_\ell. 
\end{align*}
Note that $\frac{cz+d}{-paz-pb}=-\frac{c}{pa}-\frac 1{pa(az+b)}$ and $\mu(c,d+tc,\ell,p)=\mu(c,d,\ell,p)$ by \eqref{mu cdlp def} for all $\ell$ and $t$. Recall that the matrix $\mu_{0\infty}\(\sigma_0^{-1}\begin{psmallmatrix}
	a&b\\c&d
\end{psmallmatrix}\)^{-1}$ maps the entry at $[a\ell]$ to $\ell$ and $\nu_\eta\(\begin{psmallmatrix}
	a&b+ta\\c&d+tc
\end{psmallmatrix}\)=\nu_\eta(\begin{psmallmatrix}
	a&b\\c&d
\end{psmallmatrix})e(-\frac t{24})$. Then when $\ell\notin\condar$, the contribution from a single $a$ to the $\ell$-th entry is zero; when $\ell\in \condar$, such contribution is
\begin{align*}
	&\frac{e(-\frac k2)}{p^s a^{2s}}\sum_{\substack{b\Mod a^*\\0<c<pa,\ p|c\\\text{s.t.\,}ad-bc=1}} \mu_{0\infty}\(\sigma_0^{-1}\begin{psmallmatrix}
		a&b\\c&d
	\end{psmallmatrix}\)^{-1}e\(-\frac{m_{\pm0}^{([a\ell])}c}{pa}\)\e_{[a\ell]}\\
	&\quad \cdot\sum_{t\in \Z}  e(t\alpha_{\pm\infty})\(\frac{z+\frac ba+t}{|z+\frac ba+t|}\)^{-k} \frac{y^s}{|z+\frac ba+t|^{2s}}
	e\(\frac{-m_{ \pm 0}^{([a\ell])}} {pa^2(z+\frac ba+t)}  \)\\
	&=y^se(-\tfrac k2)\sum_{t\in \Z}\frac{\mathbf{S}_{0\infty}^{(\ell)}(m^{([a\ell])},t,a,\mu;r)}{(a\sqrt p)^{2s}}B(a\sqrt p,m_{ \pm 0}^{([a\ell])},t_{\pm\infty},y,s,k)e(t_{\pm \infty}x).
\end{align*}
Here we get the last step in the same way as the steps before \eqref{Bcmtysk function}. 
The Fourier expansion of $\mathbf{U}_0(z;s,k,\mathbf{m}, \mu,r)$ for $\re s>1$ is then
\begin{align}\label{Fourier expan Poincare series 0 infty}
	\begin{split}
		&\mathbf{U}_0(z;s,k,\mathbf{m},\mu,r)\\
		&=y^se(-\tfrac k2)\sum_{\substack{ a>0\\p\nmid a}}\sum_{\ell\in\condar}\sum_{t\in \Z}\frac{\mathbf{S}_{0\infty}^{(\ell)}(m^{([a\ell])},t,a,\mu;r)}{(a\sqrt p)^{2s}}B(a\sqrt p,m_{\pm 0}^{([a\ell])},t_{\pm\infty},y,s,k)e(t_{\pm \infty}x)
	\end{split}
\end{align}
where $\mathbf{S}_{0\infty}^{(\ell)}(m^{([a\ell])},t,c,\mu;r)$ is defined in \eqref{S 0 infty ell def}. 

Similarly, for $L\in \condr$, we can compute the Fourier expansion of $\mathbf{U}_{0,(L)}(z;s,k,m^{(L)},\mu,r)$ \eqref{U 0 l def} and get
\begin{align}\label{Fourier expan Poincare series 0 infty, U subscript 0, ell}
	\begin{split}
		&\mathbf{U}_{0,(L)}^{(\ell)}(z;s,k,m^{(L)},\mu,r)\\
		&=y^se(-\tfrac k2)\sum_{ \substack{a>0:\\p\nmid a,\, [a\ell]=L } }\sum_{t\in \Z}\frac{\mathbf{S}_{0\infty}^{(\ell)}(m^{(L)},t,a,\mu;r)}{(a\sqrt p)^{2s}} B(a\sqrt p,m_{ \pm 0}^{(L)},t_{\pm\infty},y,s,k)e(t_{\pm \infty}x). 
	\end{split}
\end{align}

If $\re s>1$ and $\lambda=s(1-s)$ is an eigenvalue of $\Delta_k$, we have the recurrence relation
\begin{align}
	\label{meromorphic continuation cusp 0, ell}
	\mathbf{U}_{0,(\ell)}(z;s,k,m^{(\ell)},\mu,r)&=-4\pi m_{ \pm 0}^{(\ell)} \mathscr{R}_\lambda\mathbf{U}_{0,(\ell)}(z;s+1,k,m^{(\ell)},\mu,r)\\
	\label{meromorphic continuation cusp 0}
	\mathbf{U}_{0}(z;s,k,\mathbf{m},\mu,r)&=-4\pi \sum_{\ell\in\condr} m_{\pm 0}^{(\ell)} \mathscr{R}_\lambda\mathbf{U}_{0,(\ell)}(z;s+1,k,m^{(\ell)},\mu,r)
\end{align}
for $\mathscr{R}_\lambda=(\Delta_k+\lambda)^{-1}$. Then $\mathbf{U}_{0}(z;s,k,\mm,\mu,r)$ and every $\mathbf{U}_{0,(\ell)}(z;s,k,m^{(\ell)},\mu,r)$ can be meromorphically continued to the half plane $\re s>\frac12$ except at most a finite number of simple poles at $s=s_j$ with $\frac12<s_j<1$.

Recall $\mathrm{OB}(r_j)$ at the paragraph of \eqref{V Fourier expansion at the cusp 0}. For $\mathbf{m}\leq 0$ (then $\mathbf{1-m}>0$),  the residue of $\overline{\mathbf{U}_{0}(z;\overline{s},-\tfrac12,\mathbf{1-m},\overline{\mu_p},r)}$ at $s=s_j$ is given by
\begin{equation}\label{residue U0 from Vrj}
	\sum_{\mathbf{V}\in \mathrm{OB}(r_j)}\Res_{s=s_j} \left\langle \overline{\mathbf{U}_{0}(\cdot;\overline{s},-\tfrac12,\mathbf{1-m},\overline{\mu_p},r)}, \mathbf{V}(\cdot;r_j)\right\rangle \mathbf{V}(z;r_j).  
\end{equation}
We will compute the inner product below and will finally get \eqref{inner product U0 and Vrj}. Recall \eqref{U 0 l def} for the definition of $\mathbf{U}_{0,(\ell)}$.  When $\ell\in \condr$, we use $\mathbf{V}$ to abbreviate the second term in the inner product $\mathbf{V}(\cdot;r_j)$ and get
\begin{align*}
	&\left\langle \overline{\mathbf{U}_{0,(\ell)}(\cdot;\overline{s},-\tfrac12,\mathbf{1-m},\overline{\mu_p},r)}, \mathbf{V}(\cdot;r_j)\right\rangle\\
	&=\int_{\Gamma_0(p)\setminus \HH}\!\!\!\!\!\mathbf{V}^{\mathrm H}\!\!\!\!\!\sum_{\gamma\in \Gamma_0\setminus \Gamma_0(p)}\!\!\!\!\! \mu_p(\gamma)^{-1}w_{-\frac12}(\sigma_0^{-1},\gamma) j(\sigma_0^{-1}\gamma,z)^{-\frac12} \im(\sigma_0^{-1}\gamma z)^{s} \overline{e((1-m)_{-0}^{(\ell)}z)} \e_\ell\frac{dxdy}{y^2}\\
	&=\sum_{\gamma\in \Gamma_0\setminus \Gamma_0(p)}\int_{\sigma_0^{-1}\gamma(\Gamma_0(p)\setminus \HH)}  \mathbf{V}(\sigma_0 z;r_j)^{\mathrm H}j(\sigma_0,z)^{\frac12} y^s e(m_{+0}^{(\ell)}x)e^{2\pi m_{+0}^{(\ell)}y}\e_\ell \frac{dxdy}{y^2}\\
	&=\int_0^\infty y^se^{2\pi m_{+0}^{(\ell)}y}\frac{dy}{y^2} \int_\R (\mathbf{V}|_{\frac12}\sigma_0)(z;r_j)^{\mathrm H} e(m_{+0}^{(\ell)}x)\e_\ell dx,
\end{align*}
where we have used the following properties: $\mu_p(\gamma)$ is unitary, $j(M,z)j(M^{-1},Mz)=1$ for $M\in \SL_2(\R)$, and the trick in \eqref{trick w k}: 
\begin{align*}
	&j(\gamma^{-1},\sigma_0 z)^{-\frac12}w_{-\frac12}(\sigma_0^{-1},\gamma)j(\sigma_0^{-1}\gamma,\gamma^{-1}\sigma_0 z)^{-\frac12}\\
	&=j(\gamma^{-1},\sigma_0 z)^{-\frac12} j(\gamma,\gamma^{-1}\sigma_0 z)^{-\frac12} j(\sigma_0^{-1},\gamma\gamma^{-1}\sigma_0z)^{-\frac12}j(\sigma_0^{-1}\gamma,\gamma^{-1}\sigma_0 z)^{\frac12}j(\sigma_0^{-1}\gamma,\gamma^{-1}\sigma_0 z)^{-\frac12}\\
	&=j(\sigma_0^{-1},\sigma_0z)^{-\frac12}=j(\sigma_0,z)^{\frac12}. 
\end{align*}
We have also used the property (see e.g. \cite[Proof of Lemma~3.1]{DeshIw1982}) that for every cusp $\ma$ of $\Gamma$, 
\[\bigcup_{\gamma\in \Gamma_\ma\setminus \Gamma}\sigma_\ma^{-1}\gamma (\Gamma\setminus \HH)=\Gamma_\infty\setminus \HH \quad \text{ up to a zero-measure set}.\]
Then we can apply \eqref{integral e y W for inner product} to get
\begin{align}\label{inner product U 0 l and Vrj}
	\begin{split}
		&\left\langle \overline{\mathbf{U}_{0,(\ell)}(\cdot;\overline{s},-\tfrac12,1-m^{(\ell)},\overline{\mu_p},r)}, \mathbf{V}(\cdot;r_j)\right\rangle\\
		&=\overline{\rho_{j,0}^{(\ell)}(m^{(\ell)})}\int_0^\infty y^{s-1}e^{2\pi m_{+0}^{(\ell)}y}W_{-\frac 14,ir_j}\(4\pi |m_{+0}^{(\ell)}|y\)\frac{dy}{y} \\
		&=\overline{\rho_{j,0}^{(\ell)}(m^{(\ell)})}(4\pi |m_{+0}^{(\ell)}|)^{1-s}\frac{\Gamma(s-s_j)\Gamma(s+s_j-1)}{\Gamma(s+\frac 14)}. 
	\end{split}
\end{align}
The residue of $\overline{\mathbf{U}_{0,(\ell)}(z;\overline{s},-\tfrac12,1-m^{(\ell)},\overline{\mu_p},r)}$ at $s=s_j$ is then given by the following sum combining \eqref{inner product U 0 l and Vrj}: 
\begin{equation}\label{residue U 0 l from Vrj}
	\sum_{\mathbf{V}\in \mathrm{OB}(r_j)}\Res_{s=s_j} \left\langle \overline{\mathbf{U}_{0,(\ell)}(\cdot;\overline{s},-\tfrac12,1-m^{(\ell)},\overline{\mu_p},r)}, \mathbf{V}(\cdot;r_j)\right\rangle \mathbf{V}(z;r_j).
\end{equation}
Summing up for $\ell\in \condr$, we get
\begin{equation}\label{inner product U0 and Vrj}
	\left\langle \overline{\mathbf{U}_{0}(z;\overline{s},-\tfrac12,\mathbf{1-m},\overline{\mu_p},r)}, \mathbf{V}(\cdot;r_j)\right\rangle=\sum_{\ell\in \condr }\overline{\rho_{j,0}^{(\ell)}(m^{(\ell)})}(4\pi |m_{+0}^{(\ell)}|)^{1-s}\frac{\Gamma(s-s_j)\Gamma(s+s_j-1)}{\Gamma(s+\frac 14)}. 
\end{equation}

The following lemma still holds as in \cite[Lemma~1]{gs} using the the vector-valued inner product in \eqref{vec val Petersson inner prod}. Recall \eqref{U 0 def}, \eqref{U 0 l def} and \eqref{Poincare series generated from l}. 
\begin{lemma}\label{Inner product U 0 U 0}
	Let $s=\sigma+it$ with $\frac 12<\sigma\leq 2$, $|t|>1$, and $\displaystyle M=\max_{\ell\in \condr}|m_{+0}^{(\ell)}|$. We have
	\begin{align*}
		\left\langle \mathbf{U}_0(\cdot;s,k,\mathbf{m},\mu,r),\mathbf{U}_0(\cdot;s,k,\mm,\mu,r)\right\rangle&\ll_{p} \frac{M^2}{(\sigma-\frac12)^2},\\
		\left\langle \mathbf{U}_{0,(\ell)}(\cdot;s,k,m^{(\ell)},\mu,r),\mathbf{U}_{0,(\ell)}(\cdot;s,k,m^{(\ell)},\mu,r)\right\rangle&\ll_{p} \frac{|m_{+0}^{(\ell)}|^2}{(\sigma-\frac12)^2}.
	\end{align*}
\end{lemma}

Recall \eqref{KloostermanSelbergZetaFunctionVVal} for our Kloosterman-Selberg zeta functions. We have the following lemma. 
\begin{lemma}\label{Inner product U 0 U ell}
	Let $\mathbf{m}\leq 0 $, $1\leq \ell,L\leq p-1$, and $\re s=\sigma>\frac12$. Then
	\[Z_{\substack{0\infty;r\\\mathbf{m},n,+}}^{(\ell)}(s) \quad \text{and}\quad  Z_{\substack{0\infty;r\\m^{(L)},n,+}}^{(\ell)}(s) \quad\text{are meromorphic in }\re s>\tfrac12\]
	with at most a finite number of simple poles in $(\tfrac12,1)$. Moreover, when  $n_{+\infty}>0$, we have	
	\begin{align*}
		&\left\langle \overline{\mathbf{U}_{0,(L)}(\cdot;\overline s,-\tfrac 12,1-m^{(L)},\overline{\mu_p},r)},\  \mathbf{U}_{(\ell)}(\cdot;\overline{s}+2,\tfrac 12,n,\mu_p,r) \right\rangle\\
		&\qquad =\frac{e(\tfrac 18)\Gamma(2s+1)}{4^{s+1}\pi n_{+\infty}^2\Gamma(s+\frac 14)\Gamma(s+\frac 74)}Z_{\substack{0\infty;r\\m^{(L)},n,+}}^{(\ell)}(s)+R_{\substack{0\infty;r\\m^{(L)},n,+}}^{(\ell)}(s),
	\end{align*}
	\begin{align*}
		&\left\langle \overline{\mathbf{U}_0(\cdot;\overline s,-\tfrac 12,\mathbf{1-m},\overline{\mu_p},r)},\  \mathbf{U}_{(\ell)}(\cdot;\overline{s}+2,\tfrac 12,n,\mu_p,r) \right\rangle\\
		&\qquad =\frac{e(\tfrac 18)\Gamma(2s+1)}{4^{s+1}\pi n_{+\infty}^2\Gamma(s+\frac 14)\Gamma(s+\frac 74)}Z_{\substack{0\infty;r\\\mathbf{m},n,+}}^{(\ell)}(s)+R_{\substack{0\infty;r\\\mathbf{m},n,+}}^{(\ell)}(s),
	\end{align*}
	and when $n_{+\infty}<0$, we have
	\begin{align*}
		&\left\langle \overline{\mathbf{U}_{0,(L)}(\cdot;\overline s,-\tfrac 12,1-m^{(L)},\overline{\mu_p},r)},\ \overline{\mathbf{U}_{(\ell)}(\cdot;s+2,-\tfrac 12,1-n,\overline{\mu_p},r)}\right\rangle\\
		&\qquad =\frac{e(\tfrac 18)\Gamma(2s+1)}{4^{s+1}\pi |n_{+\infty}|^2\Gamma(s-\frac 14)\Gamma(s+\frac 94)}Z_{\substack{0\infty;r\\m^{(L)},n,+}}^{(\ell)}(s)+R_{\substack{0\infty;r\\m^{(L)},n,+}}^{(\ell)}(s), 
	\end{align*}
	\begin{align*}
		&\left\langle \overline{\mathbf{U}_0(\cdot;\overline s,-\tfrac 12,\mathbf{1-m},\overline{\mu_p},r)},\ \overline{\mathbf{U}_{(\ell)}(\cdot;s+2,-\tfrac 12,1-n,\overline{\mu_p},r)}\right\rangle\\
		&\qquad =\frac{e(\tfrac 18)\Gamma(2s+1)}{4^{s+1}\pi |n_{+\infty}|^2\Gamma(s-\frac 14)\Gamma(s+\frac 94)}Z_{\substack{0\infty;r\\\mathbf{m},n,+}}^{(\ell)}(s)+R_{\substack{0\infty;r\\\mathbf{m},n,+}}^{(\ell)}(s). 
	\end{align*}
	Here both $R_{\substack{0\infty;r\\m^{(L)},n,+}}^{(\ell)}(s)$ and $R_{\substack{0\infty;r\\\mathbf{m},n,+}}^{(\ell)}(s)$ are holomorphic in $\sigma>\frac12$ and
	\[R_{\substack{0\infty;r\\m^{(L)},n,+}}^{(\ell)}(s)\ll_p \frac{|m_{+0}^{(L)}n_{+\infty}|}{\sigma-\frac 12}, \qquad R_{\substack{0\infty;r\\\mathbf{m},n,+}}^{(\ell)}(s)\ll_p \frac{|Mn_{+\infty}|}{\sigma-\frac 12}. \]
\end{lemma}

\begin{proof}
	Set $\re s_2=\sigma_2>\re s_1=\sigma_1>\frac12$. We only prove the formulas for the inner products involving $\mathbf{U}_{0,(L)}(\cdot;\overline s,-\tfrac 12,1-m^{(L)},\overline{\mu_p},r)$, while the other two formulas for $\mathbf{U}_{0}(\cdot;\overline s,-\tfrac 12,\mathbf{1-m},\overline{\mu_p},r)$ follow by summing on $L\in \condr$ and by $\displaystyle M=\max_{L\in \condr} |m_{+0}^{(L)}|$.

	After we prove the formulas for these inner products, the meromorphic property of the Kloosterman-Selberg zeta functions follows directly from the meromorphic continuation \eqref{meromorphic continuation cusp 0, ell} and \eqref{meromorphic continuation cusp 0}. 
	
	We start with the first inner product using unfolding. Here $\overline{\mathbf{U}_{0,(L)}}$ abbreviates the first term in the inner product. Recall \eqref{U 0 l def} and \eqref{Poincare series generated from l}. We have
	\begin{align*}
		&\left\langle \overline{\mathbf{U}_{0,(L)}(\cdot;\overline{s_1},-\tfrac 12,m^{(L)},\overline{\mu_p},r)}, \mathbf{U}_{(\ell)}(\cdot;\overline{s_2},\tfrac 12,n,\mu_p,r) \right\rangle\\
		&=\int_{\Gamma_0(p)\setminus \HH} \(\sum_{\gamma\in \Gamma_\infty\setminus \Gamma_0(p)}\mu_p(\gamma)^{-1} j(\gamma,z)^{-\frac12}(\im \gamma z)^{\overline{s_2}}\,\frac{e( n_{+\infty} \gamma z) }{\sin(\frac{\pi \ell}p)} \e_{\ell}\)^{\mathrm H}\overline{\mathbf{U}_{0,(L)}}\frac{dxdy}{y^2}\\
		&=\sum_{\Gamma_\infty\setminus \Gamma_0(p)}\int_{\Gamma_0(p)\setminus \HH}  y^{s_2}e^{-2\pi n_{+\infty}y}\frac{e(-n_{+\infty}x)\e_\ell^{\mathrm T}}{\sin(\frac{\pi \ell }p)}\overline{\mathbf{U}_{0,(L)}(z;\overline{s_1},-\tfrac 12,m^{(L)},\overline{\mu_p},r)}\frac{dxdy}{y^2}. 
	\end{align*}
	Then we apply the Fourier expansion of $\mathbf{U}_{0,(L)}(\cdot;\overline{s_1},-\tfrac 12,m^{(L)},\overline{\mu_p},r)$ in \eqref{Fourier expan Poincare series 0 infty, U subscript 0, ell} with \eqref{Integral of u to get Whittaker} to continue:
	\begin{align*}
		&=e(-\tfrac 14)\sum_{\substack{a>0:\,p\nmid a,\\ [a\ell]=L}}\frac{\overline{S_{0\infty}^{(\ell)}(1-m^{(L)},1-n,a,\overline{\mu_p})}}{(a\sqrt p)^{2s_1}}\int_0^\infty y^{s_2-s_1-1}e^{-2\pi n_{+\infty}y}\\
		&\qquad\cdot \int_\R\(\frac{u+i}{|u+i|}\)^{-\frac12}e\(\frac{-m_{+0}^{(L)}(u+i)}{pa^2y(u^2+1)}-n_{+\infty}yu\)\frac{du}{(u^2+1)^{s_1}} dy\\
		&=\frac{e(-\tfrac 38)4^{1-s_2}\pi^{1+s_1-s_2}}{n_{+\infty}^{s_2-s_1}}\cdot\frac{\Gamma(s_2+s_1-1)\Gamma(s_2-s_1)}{\Gamma(s_1+\frac14)\Gamma(s_2-\frac 14)}Z_{\substack{0\infty;r\\ m^{(L)},n,+}}^{(\ell)}(s_1)+R_{\substack{0\infty;r\\m^{(L)},n,+}}^{(\ell)}(s_1,s_2). 
	\end{align*}
	Here
	\begin{align*}
		R_{\substack{0\infty;r\\m^{(L)},n,+}}^{(\ell)}(s_1,s_2)&=e(-\tfrac 38)\sum_{\substack{a>0:\,p\nmid a,\\ [a\ell]=L}}\frac{S_{0\infty}^{(\ell)}(m^{(L)},n,a,\mu_p)}{(a\sqrt p)^{2s_1}}\int_0^\infty y^{s_2-s_1-1}e^{-2\pi n_{+\infty}y} \\
		&\cdot\int_\R\(\frac{u+i}{|u+i|}\)^{-\frac12}e\(-n_{+\infty}yu\) \(e\(\frac{-m_{+0}^{(L)}(u+i)}{pa^2y(u^2+1)}\)-1\)\frac{du}{(u^2+1)^{s_1}} dy\\
		&=\sum_{\substack{a>0:\,p\nmid a,\\ [a\ell]=L}}\frac{S_{0\infty}^{(\ell)}(m^{(L)},n,a,\mu_p)}{(a\sqrt p)^{2s_1}}O_p\(\frac{|m_{+0}^{(L)}n_{+\infty}|}{pa^2(\sigma_1-\frac12)}\)=O_p\(\frac{|m_{+0}^{(L)}n_{+\infty}|}{\sigma_1-\frac12}\)
	\end{align*}
	and is holomorphic in the region $\re s_1>\frac12$ with $s_2=s_1+2$. The last step is by the trivial bound $S_{0\infty}^{(\ell)}(m^{(L)},n,a,\mu_p)\ll a$. 
	
	Then we deal with the case $n_{+\infty}<0$. As before, we have
	\begin{align*}
		&\left\langle \overline{\mathbf{U}_{0,(L)}(\cdot;\overline{s_1},-\tfrac 12,1-m^{(L)},\overline{\mu_p},r)},\  \overline{\mathbf{U}_{(\ell)}(\cdot;s_2,-\tfrac 12,1-n,\overline\mu_p,r)} \right\rangle\\
		&=\int_{\Gamma_\infty \setminus \HH}  y^{s_2}\frac{e(-n_{+\infty}z)}{\sin(\frac{\pi \ell }p)}\e_\ell^{\mathrm T}\,\overline{\mathbf{U}_{0,(L)}(\cdot;\overline{s_1},-\tfrac 12,1-m^{(L)},\overline{\mu_p},r)}\frac{dxdy}{y^2}\\
		&=\frac{e(-\tfrac 38)4^{1-s_2}\pi^{1+s_1-s_2}}{|n_{+\infty}|^{s_2-s_1}}\cdot\frac{\Gamma(s_2+s_1-1)\Gamma(s_2-s_1)}{\Gamma(s_1-\frac14)\Gamma(s_2+\frac 14)}Z_{\substack{0\infty;r\\m^{(L)},n,+}}^{(\ell)}(s_1)+R_{\substack{0\infty;r\\m^{(L)},n,+}}^{(\ell)}(s_1,s_2). 
	\end{align*}
	We still have that $R_{\substack{0\infty;r\\m^{(L)},n,+}}^{(\ell)}(s_1,s_2)=O_p\(\frac{|m_{+0}^{(L)}n_{+\infty}|}{\sigma_1-\frac12}\)$ and is holomorphic for $\sigma_1>\frac 12$ and $s_2=s_1+2$. This finishes the proof.

\end{proof}

The following proposition follows directly from Lemma~\ref{Inner product U 0 U ell}, Cauchy-Schwarz, Lemma~\ref{Inner product U 0 U 0} and Lemma~\ref{Inner product U U U ell U ell}. Note that the norms involving $s+2$ and $\overline{s}+2$ have $\sigma+2-\frac 12>2$ and do not contribute to the denominator. 
\begin{proposition}
	Let $\mathbf{m}\leq 0$, $\displaystyle M=\max_{\ell\in \condr}|m_{+0}^{(\ell)}|$ and $s=\sigma+it$ with $\re s=\sigma>\frac12$. When $|t|\rightarrow \infty$, we have the growth condition
	\[Z_{\substack{0\infty;r\\m^{(L)},n,+}}^{(\ell)}(s)\ll_{p}\frac{|m_{+0}^{(L)}n_{+\infty}|^3|t|^{\frac12}}{\sigma-\frac12}\quad \text{and} \quad Z_{\substack{0\infty;r\\\mathbf{m},n,+}}^{(\ell)}(s)\ll_{p}\frac{|Mn_{+\infty}|^3|t|^{\frac12}}{\sigma-\frac12}. \]
\end{proposition}

Then we can prove the remaining bounds \eqref{Kloosterman sums asymptotics 0 infty, every L} and \eqref{Kloosterman sums asymptotics 0 infty} in Theorem~\ref{Sums of Kloosterman sums}.  

\begin{proof}[Proof of \eqref{Kloosterman sums asymptotics 0 infty, every L} and \eqref{Kloosterman sums asymptotics 0 infty} in Theorem~\ref{Sums of Kloosterman sums}]
	Take any small $\ep>0$. Since $Z_{\substack{0\infty;r\\ m^{(L)},n,+}}^{(\ell)}(s)\ll_{\ep} \zeta(1+\ep)$ for $\re s=1+\ep$, by the Phragm\'en-Lindel\"of principle, we have
	\begin{equation}\label{Z0infty after phragmen lindelof}
		Z_{\substack{0\infty;r\\m^{(L)},n,+}}^{(\ell)}(\tfrac{1+s}2)\ll_{p,\ep} |m^{(L)}n_{+\infty}|^{3}|t|^{\frac12-\frac \sigma 2+\ep} \quad \text{for }0<\ep\leq \sigma\leq 1+\ep.
	\end{equation} 
	Following the similar step after \eqref{Zinftyinfty after phragmen lindelof}, we have
	\begin{align}
		\sum_{\substack{a>0:\\p\nmid a,\, [a\ell]=L}}\!\!\!\frac{S_{0\infty}^{(\ell)}(m^{(L)},n,a,\mu_p)}{a\sqrt p}=\frac1{2\pi i}\int_{1-iT}^{1+iT}Z_{\substack{0\infty;r\\m^{(L)},n,+}}^{(\ell)}(\tfrac{s+1}2)\frac{x^s}s ds+O\(\frac{|m_{+0}^{(L)}n_{+\infty}|^{3}x^{1+\ep}}T\). 
	\end{align}
	By Lemma~\ref{Inner product U 0 U ell}, the function $Z_{\substack{0\infty;r\\m^{(L)},n,+}}^{(\ell)}(\tfrac{1+s}2)$ has at most a finite number of simple poles at $2s_j-1\in (0,1)$, where $\lambda_j=s_j(1-s_j)<\frac 14$ are the discrete eigenvalues of $\Delta_{\frac 12}$ on $\Lform_{\frac 12}(\Gamma_0(p),\mu_p)$. Shifting the line of integration above to $\re s=\ep$ ($\ep-iT\rightarrow \ep+iT$) such that $2s_j-1>\ep$ for all $\lambda_j<\frac 14$, with the help of \eqref{Z0infty after phragmen lindelof} we obtain
	\begin{align*}
		\sum_{\substack{a>0:\\p\nmid a,\, [a\ell]=L}}\!\!\!\frac{S_{0\infty}^{(\ell)}(m^{(L)},n,a,\mu_p)}{a\sqrt p}=\sum_{\frac12<s_j<\frac 34}\Res_{s=2s_j-1}Z_{\substack{0\infty;r\\m^{(L)},n,+}}^{(\ell)}(\tfrac{1+s}2)\frac{x^{2s_j-1}}{2s_j-1}+O(|m_{+0}^{(L)}n_{+\infty}|^3x^{\frac 13+\ep}),
	\end{align*}
	where we have chosen $T=x^{\frac 23}$. 
	
	For the residues, by Lemma~\ref{Inner product U 0 U ell}, it suffices to compute the residue of the first and third inner products in that lemma. 
	Combining Lemma~\ref{Inner product U 0 U ell}, \eqref{residue U 0 l from Vrj}, \eqref{Inner product V and U infty n}, and \eqref{Inner product V and U infty 1-n}, the proof of \eqref{Kloosterman sums asymptotics 0 infty, every L} follows. 
	
	The proof of \eqref{Kloosterman sums asymptotics 0 infty} follows by summing up on $L\in \condr$ and by $\displaystyle M=\max_{L\in \condr}|m_{+0}^{(L)}|$. 
	
\end{proof}

\subsection{Convergence}
\label{Subsection: Convergence of sums of KL sums}
In this subsection we show the growth rates and convergence properties of sums of Kloosterman sums. 

\begin{proposition}\label{convergence in general, tail bound}
	With the same setting as Theorem~\ref{Sums of Kloosterman sums}, we have that all the following sums are convergent with bounds
	\begin{align*}
		&\sum_{\substack{c>4\pi \la m_{+\infty}n_{+\infty}\ra^{\frac12}\\p|c}}\frac{S_{\infty\infty}^{(\ell)}(m,n,c,\mu_p)}{c} \mathscr{M}\(\frac{4\pi |m_{+\infty}n_{+\infty}|^{\frac12}}c\)\ll_p |m_{+\infty}n_{+\infty}|^{4},\\
		&\sum_{\substack{a>4\pi \la m_{+0}^{(L)}n_{+\infty}\ra^{\frac12}\\p\nmid a,\,[a\ell]=L}}\left|\frac{m_{+0}^{(L)}}{n_{+\infty}}\right|^{\frac 14}\frac{S_{0\infty}^{(\ell)}(m^{(L)},n,a,\mu_p)}{a\sqrt p}\mathscr{M}\(\frac{4\pi |m_{+0}^{(L)}n_{+\infty}|^{\frac12}}c\)\ll_p |m_{+0}^{(L)}n_{+\infty}|^4,\\
		&\sum_{\substack{a>4\pi \la Mn_{+\infty}\ra ^{\frac12}\\p\nmid a,}}\sum_{\ell\in \condar}\left|\frac{m_{+0}^{([a\ell])}}{n_{+\infty}}\right|^{\frac 14}\frac{S_{0\infty}^{(\ell)}(m^{([a\ell])},n,a,\mu_p)}{a\sqrt p}\mathscr{M}\(\frac{4\pi |m_{+0}^{([a\ell])}n_{+\infty}|^{\frac12}}c\)\ll_p |Mn_{+\infty}|^4. 
	\end{align*}
	Here $\mathscr{M}$ is either the Bessel function $I_{\frac12}$ or $J_{\frac12}$. 
\end{proposition}

\begin{proof}
	The proof is similar to the scalar-valued case. We use the properties of Bessel functions: by \cite[(10.14.4)]{dlmf}, $|J_{\alpha}(x)|\leq_\alpha x^\alpha$ for $x>0$ and $\alpha\geq -\frac12$, and by \cite[(10.30.1)]{dlmf}, for $0\leq x\leq \beta$, 
	\[I_{\alpha}(x)\ll_{\alpha,\beta}x^{\alpha} \quad \text{for }\alpha>-1. \]
	Let $\phi \defeq 4\pi |m_{+\infty} n_{+\infty}|^{\frac12}$. We have $\phi\geq \frac \pi 6$ by $\alpha_{+\infty}=\frac 1{24}$ in \eqref{alpha infty l def}. When $t\geq \phi$, we have $0\leq \frac \phi t\leq 1$, hence by \cite[(10.29.1), (10.6.1)]{dlmf}, we get 
	\begin{align}
		\label{I Bessel derivative bound}
		\frac{d}{dt}I_{\frac 12}\(\frac \phi t\)&=-\frac \phi{2t^2}\(I_{-\frac12 }\(\frac \phi t\)+I_{\frac32 }\(\frac \phi t\)\)\ll \phi^{\frac12}t^{-\frac32}+ \phi^{\frac52} t^{-\frac72}\ll\phi^{\frac12}t^{-\frac32} ,\\
		\label{J Bessel derivative bound}
		\frac{d}{dt}J_{\frac 12}\(\frac \phi t\)&=-\frac \phi {2t^2}\(J_{-\frac12 }\(\frac \phi t\)-J_{\frac32 }\(\frac \phi t\)\)\ll \phi^{\frac12}t^{-\frac32}+\phi^{\frac 52}t^{-\frac72}\ll\phi^{\frac12}t^{-\frac32}.
	\end{align}
	By Corollary~\ref{Sums of Kloosterman sums, whole growth rate for convergence}, we write
	\[\mathrm{SC}(x)\defeq \sum_{p|c\leq x}\frac{S_{\infty\infty}^{(\ell)}(m,n,c,\mu_p)}{c}\ll_p |m_{+\infty}n_{+\infty}|^3 x^{\frac12-\delta}. \]
	By partial integration, for $T>\phi$ we have
	\begin{align}\label{partial sum KL sum: from phi to T}
		\begin{split}
			\sum_{\substack{\phi<c\leq T\\p|c}}&
			\frac{S_{\infty\infty}^{(\ell)}(m,n,c,\mu_p)}{c}\mathscr{M}\(\frac \phi c\)\\
			&\ll  \left|\mathscr{M}\(\frac \phi T\) \mathrm{SC}(T)\right|+\left|\mathscr{M}(1) \mathrm{SC}(\phi)\right|+\left|	\int_\phi^T \mathrm{SC}(t)\phi^{\frac 12} t^{-\frac 32}dt\right|\\
			&\ll |m_{+\infty}n_{+\infty}|^{3}\phi^{\frac 12}(1+T^{-\delta})\\ &\ll |m_{+\infty}n_{+\infty}|^{4}. 
		\end{split}
	\end{align} 
	For $Y>X>\phi$, we also have
	\begin{align}\label{partial sum KL sum: from X to Y}
		\begin{split}
			\sum_{\substack{X<c\leq Y\\p|c}}&
			\frac{S_{\infty\infty}^{(\ell)}(m,n,c,\mu_p)}{c}\mathscr{M}\(\frac \phi c\)\\
			&\ll  \left|\mathscr{M}\(\frac \phi Y\)\mathrm{SC}(Y)-\mathscr{M}\(\frac \phi X\) \mathrm{SC}(X)\right|+
			\left|	\int_X^Y \mathrm{SC}(t)\phi^{\frac 12} t^{-\frac 32}dt\right|\\ 
			&\ll |m_{+\infty}n_{+\infty}|^{3}\phi^{\frac 12}\left| Y^{-\delta}-X^{-\delta}\right| . 
		\end{split}
	\end{align} 
	The estimate \eqref{partial sum KL sum: from X to Y} proves the convergence by Cauchy's criterion, so we can let $T\rightarrow \infty$ in \eqref{partial sum KL sum: from phi to T}. We have proved the first estimate of the proposition.  
	
	Again by Corollary~\ref{Sums of Kloosterman sums, whole growth rate for convergence}, we write
	\[\mathrm{SA}(x)\defeq \sum_{\substack{a\leq x:\\p\nmid a,\,[a\ell]=L}} \frac{S_{0\infty}^{(\ell)}(m^{(L)},n,a,\mu_p)}{a\sqrt{p}}\ll_{p} |m_{+0}^{(L)}n_{+\infty}|^3x^{\frac12-\delta}.\]
	Let $\phi \defeq 4\pi |m_{+0}^{(L)} n_{+\infty}|^{\frac12}$. A similar application of the partial sum concludes the second formula. The third formula follows from summing up $L\in \condr$ and by $\displaystyle M=\max_{L\in \condr}|m_{+0}^{(L)}|$. Here we can re-order the sum because the convergence can be easily derived by separating the partial sum  into $\# \condr$ parts. 
\end{proof}
For $\beta>0$, recall $\Gamma(\alpha,\beta)= \int_{\beta}^{\infty} t^{\alpha-1}e^{-t} dt$ as the incomplete Gamma function. We have $\Gamma(\alpha,x)\sim x^{\alpha-1}e^{-x}$ when $x\rightarrow \infty$. 
\begin{proposition}\label{convergence in general, main contribution}
	With the same setting as Proposition~\ref{convergence in general, tail bound}, we have
	\begin{align*}
		&	\sum_{\substack{c>0:\,p|c}}\frac{S_{\infty\infty}^{(\ell)}(m,n,c,\mu_p)}{c} I_{\frac12}\(\frac{4\pi |m_{+\infty}n_{+\infty}|^{\frac12}}c\)\ll_p |m_{+\infty}n_{+\infty}|^{5} e^{4\pi \la m_{+\infty}n_{+\infty}\ra^{\frac12}},\\
		&	\sum_{\substack{c>0:\,p|c}}\frac{S_{\infty\infty}^{(\ell)}(m,n,c,\mu_p)}{c} J_{\frac12}\(\frac{4\pi |m_{+\infty}n_{+\infty}|^{\frac12}}c\)\ll_p |m_{+\infty}n_{+\infty}|^{5} ,\\
		&\sum_{\substack{a>0\\p\nmid a}}\sum_{\ell\in \condar}\left|\frac{m_{+0}^{([a\ell])}}{n_{+\infty}}\right|^{\frac 14}\frac{S_{0\infty}^{(\ell)}(m^{([a\ell])},n,a,\mu_p)}{a\sqrt p}I_{\frac12}\(\frac{4\pi |m_{+0}^{([a\ell])}n_{+\infty}|^{\frac12}}{a\sqrt p}\)\ll_p |Mn_{+\infty}|^5e^{4\pi \la Mn_{+\infty}\ra^{\frac12}},\\
		&\sum_{\substack{a>0\\p\nmid a}}\sum_{\ell\in \condar}\left|\frac{m_{+0}^{([a\ell])}}{n_{+\infty}}\right|^{\frac 14}\frac{S_{0\infty}^{(\ell)}(m^{([a\ell])},n,a,\mu_p)}{a\sqrt p}J_{\frac12}\(\frac{4\pi |m_{+0}^{([a\ell])}n_{+\infty}|^{\frac12}}{a\sqrt p}\)\ll_p |Mn_{+\infty}|^5. 
	\end{align*} 
	Hence, by denoting $q=e(z)$ with $y=\im z>0$, the following series converge: 
	\begin{align*}
		&\sum_{n_{+\infty}>0}\left|\frac{m_{+\infty}}{n_{+\infty}}\right|^{\frac 14}\la \sum_{\substack{c>0:\,p|c}}\frac{S_{\infty\infty}^{(\ell)}(m,n,c,\mu_p)}{c} I_{\frac12}\(\frac{4\pi |m_{+\infty}n_{+\infty}|^{\frac12}}c\) \ra  |q^{n_{+\infty}}|,\\
		&\sum_{n_{+\infty}<0}\left|\frac{m_{+\infty}}{n_{+\infty}}\right|^{\frac 14}\la \sum_{\substack{c>0:\,p|c}}\frac{S_{\infty\infty}^{(\ell)}(m,n,c,\mu_p)}{c}J_{\frac 12}\(\frac{4\pi |m_{+\infty}n_{+\infty}|^{\frac12}}c\) \ra  \la\Gamma(\tfrac12,4\pi |n_{+\infty}|y)q^{n_{+\infty}}\ra\\
		&\sum_{n_{+\infty}>0}\la \mathop{\sum\sum}_{\substack{a>0:\,p\nmid a,\\ \ell\in \condar}}\left|\frac{m_{+0}^{([a\ell])}}{n_{+\infty}}\right|^{\frac 14}\frac{S_{0\infty}^{(\ell)}(m^{([a\ell])},n,a,\mu_p)}{a\sqrt p}I_{\frac12}\(\frac{4\pi |m_{+0}^{([a\ell])}n_{+\infty}|^{\frac12}}c\)\ra |q^{n_{+\infty}}|,\\
		&\sum_{n_{+\infty}<0}\la \mathop{\sum\sum}_{\substack{a>0:\,p\nmid a,\\ \ell\in \condar}}\left|\frac{m_{+0}^{([a\ell])}}{n_{+\infty}}\right|^{\frac 14}\frac{S_{0\infty}^{(\ell)}(m^{([a\ell])},n,a,\mu_p)}{a\sqrt p}J_{\frac 12}\(\frac{4\pi |m_{+0}^{([a\ell])}n_{+\infty}|^{\frac12}}c\)\ra \la \Gamma(\tfrac12,4\pi |n_{+\infty}|y)q^{n_{+\infty}}\ra. 
	\end{align*}
\end{proposition}

\begin{proof}
	We first prove the formulas involving $I_{\frac 12}$. By \cite[(10.30.4)]{dlmf}, we have $I_{\beta}(\frac \phi t)\ll_\beta e^{\phi}(t/\phi)^{\frac 12}$ when $t\leq \phi$. For the first formula, we let $\phi=4\pi|m_{+\infty}n_{+\infty}|^{\frac12}$ and use partial summation with \cite[(10.29.1)]{dlmf} to get
	\begin{align*}
		\sum_{1\leq c\leq \phi:\,p|c }\!\!\!\frac{S_{\infty\infty}^{(\ell)}(m,n,c,\mu_p)}{c} I_{\frac12}\(\frac{4\pi |m_{+\infty}n_{+\infty}|^{\frac12}}c\)\ll  |m_{+\infty}n_{+\infty}|^{3} e^\phi \phi^{\frac 32-\delta }\ll |m_{+\infty}n_{+\infty}|^{5} e^\phi. 
	\end{align*}
	We combine this with Proposition~\ref{convergence in general, tail bound} to get the desired bound. Since $|e(n_{+\infty} z)|=e^{-2\pi n_{+\infty} y}$, and considering that the inner sum grows at a rate of $e^{O(n_{+\infty}^{1/2})}$, the fifth formula is clear. 
	
	To prove the third formula, we start by choosing a fixed $L\in \condr$. Then we set $\phi=4\pi|m_{+0}^{(L)}n_{+\infty}|^{\frac12}$ and apply the second formula from Corollary~\ref{Sums of Kloosterman sums, whole growth rate for convergence}. Summing over $L\in \condr$, we get the desired bounds here. The seventh formula follows from the same reason as the fifth one.  
	
	For the formulas involving $J_{\frac12}$ we apply \eqref{J Bessel derivative bound} and get the similar conclusion without the exponential growth term $e^\phi$. Since $n_{+\infty}<0$, we have
	\[\la\Gamma(\tfrac12,4\pi |n_{+\infty}|y)e(n_{+\infty} z)\ra \ll e^{-2\pi |n_{+\infty}| y}\]
	and still get the convergence as $y>0$. 
\end{proof}

\begin{remark}
	The step of repeating the selection of $L$ and then summing up on $L\in \condr$ is necessary. This is because $m_{+0}^{([a\ell])}$ changes when $a$ varies, but it remains constant when we specify $[a\ell]=L$. 
\end{remark}

\section*{Acknowledgements}
The author thanks the referee for their careful reading and elucidate comments. The author thanks Professor Scott Ahlgren for his careful reading of a previous version of this paper and for plenty of insightful suggestions. The author also thanks Nick Andersen and Alexander Dunn for helpful comments.

\bibliographystyle{alpha}
\bibliography{allrefs}

\end{document}